\documentclass[12pt]{article}
\usepackage{graphicx}
\usepackage{subcaption}
\usepackage{tikz}
\usetikzlibrary{calc,arrows.meta,positioning}

\usepackage[english]{babel}
\usepackage{amsfonts}
\usepackage{mathtools}
\usepackage[left=2cm,right=2cm,
top=2cm,bottom=2cm,bindingoffset=0cm]{geometry}

\usepackage{amssymb,amsmath,amsthm,amsfonts}
\usepackage[pdftex,colorlinks=true,linkcolor=blue,citecolor=blue,urlcolor=red,unicode=true,hyperfootnotes=false,bookmarksnumbered]{hyperref}

\usepackage{enumerate}

\theoremstyle{plain}
\newtheorem{theorem}{Theorem}[section]
\newtheorem{lemma}[theorem]{Lemma}
\newtheorem{claim}[theorem]{Claim}
\newtheorem{proposition}[theorem]{Proposition}
\newtheorem{corollary}[theorem]{Corollary}

\theoremstyle{definition}
\newtheorem{conjecture}[theorem]{Conjecture}

\newtheorem{remark}[theorem]{Remark}
\newtheorem{definition}[theorem]{Definition}

\makeatletter

\renewcommand{\leq}{\leqslant}
\renewcommand{\geq}{\geqslant}

\author{Georgii Zakharov\footnote{Mathematical Institute, University of Oxford, UK; georgii.zakharov@exeter.ox.ac.uk. This research was conducted while the author was supported by Hill Foundation Scholarship.}}
\date{April 3, 2026}

\title{Sharp threshold for reconstructing points on the line}

\begin{document}

\maketitle

\begin{abstract}
    For a set of $n$ points $V \subseteq \mathbb{R}$ let $G(V, p)$ be the random graph on $V$ where each possible edge is present independently with probability $p$. We call a subset $U \subseteq V$ {\emph {reconstructible}} if every injection $\varphi:V\to \mathbb{R}$ that preserves the distances along the edges of $G(V, p)$ also preserves all pairwise distances in $U$. How large is the size $\mathsf{R}$ of a largest reconstructible subset? Gir\~ao, Illingworth, Michel, Powierski and Scott conjectured that the answer is linear whp when $p = (1+\varepsilon)/n$ for every $\varepsilon > 0$. 
    
    In this paper, we show that for every $\varepsilon>0$ whp there exists a reconstructible subset $U$ of the largest component $\mathcal{C}$ of the 2-core satisfying $|U| = |V(\mathcal{C})|(1-o(1))$, proving a stronger form of the conjecture. The bound is asymptotically best possible, since for $V \subseteq \mathbb{R}$ linearly independent over $\mathbb{Q}$ it is straightforward to verify that $\mathsf{R} \leq \max(2, |V(\mathcal{C})|)$. Furthermore, we extend these results to every $\varepsilon:= \varepsilon(n)$ satisfying $\varepsilon = \omega(1/\ln n)$. 
    
\end{abstract}

\section{Introduction}
\label{sc:1}

For a finite set $V$ with $|V| = n$ and $p\in [0, 1]$, let $\mathcal{G}:=(V, E) =G(V, p)$ be the random graph with vertex set $V$ and each possible edge present independently with probability $p$. Regarding $V \subseteq \mathbb{R}$ as points and $E$ as distances between them, we call an injection $\varphi: V \rightarrow \mathbb{R}$ a {\it $\mathcal{G}$-rigid map} if it preserves the distances of edges, i.e.
$|\varphi(u) -  \varphi(v)| = |u - v|$ holds for every $uv \in E$. Note that every isometry of $V$ (i.e., a transformation that involves adding a constant and/or applying a symmetry) is a $\mathcal{G}$-rigid map. We call these $\mathcal{G}$-rigid maps {\it trivial} and the others {\it non-trivial}. A subset $U \subseteq V$ is {\it reconstructible}\footnote{Following Benjamini and Tzalik~\cite{BeTz}, in this paper we omit the assumption that $V$ is linearly independent over $\mathbb{Q}$,  which is often made in other areas of rigidity theory}  if, for every $\mathcal{G}$-rigid map $\varphi$, the restriction $\varphi|_{U}$ is an isometry. In this paper, we study the size of a largest reconstructible subset.

For simplicity, we write $G(n, p):= G([n], p)$, where $[n] := \{1, \ldots, n\}$. 
We say that an event occurs {\emph{with high probability}} (whp), if its probability goes to 1 as $n \to \infty$.

We study the property of reconstructing a linearly large subset of $V$ in $G(V, p)$. 
Rigidity theory has been extensively studied for some time (see, for example,~\cite{GHH}), but this branch of the subject was initiated relatively recently by Benjamini and Tzalik~\cite{BeTz}, who, among other things, studied when \emph{all} pairwise distances between $V$ can be determined in $G(V, p)$ whp. Later, Gir\~ao, Illingworth, Michel, Powierski and Scott~\cite{GIMPS} proved a hitting time result for the time step when $V$ becomes reconstructible in the Erd\H{o}s--R\'enyi evolution. At the same time, they also showed that in $G(V, p)$ with $p > 42/n$, whp one can reconstruct a linearly large subset of $V$.

Montgomery, Nenadov, Portier and Szab\'o~\cite{MNPS} studied the same problem under the stronger notion of {\it global rigidity}: the exact positions of the vertices $V$ on the real line $\mathbb{R}$ are chosen \emph{after} the exposure of the graph $G(V, p)$ in a way that minimises the size of a largest reconstructible set. We note that their methods can be used to show that a linear-sized subset is reconstructible whp in this stronger `globally rigid' setting for, say, $pn > 12$ (see Theorem 1.3 from~\cite{MNPS}). 

For further historical context, we refer to the paper by Benjamini and Tzalik~\cite{BeTz}, where they discuss developments and related results that have appeared since their original preprint on arXiv.

In~\cite{GIMPS}, Gir\~ao, Illingworth, Michel, Powierski and Scott also conjectured that for every fixed $pn>1$ one can reconstruct a linear-sized subset $U$ of $V$. Note that in a disconnected graph $\mathcal{G}$ every (maximal) reconstructible subset should be contained in a component of $\mathcal{G}$. Indeed, adding a specific constant to a component $\mathcal{S}$ of $\mathcal{G}$ changes all distances from the vertices of $\mathcal S$ to the vertices of $\mathcal{G}\setminus\mathcal{S}$. Also, in the case when the vertices of $V(\mathcal{G})$ are linearly independent over $\mathbb{Q}$, every (maximal) reconstructible subset of size at least $3$ is a subset of a 2-connected subgraph in $\mathcal{G}$. Indeed, if $S \subseteq V(\mathcal{G})$ has just one neighbour $v$ in $V(\mathcal{G}) \setminus S$ then applying a symmetry to $S$ with respect to its only neighbour $v$ changes the distances from $S$ to $V(\mathcal{G}) \setminus (\{v\}\cup S)$. So, in the worst case, one cannot hope to reconstruct a subset larger then the largest connected component of the {\it 2-core} of the graph, where the 2-core $\mathcal{G}^{(2)}$ is the largest subgraph of a graph $\mathcal{G}$ with every degree at least 2. Note that if $pn < 1$ is fixed then the set of edges is so sparse that every connected component is of size $O(\ln n)$, so the largest reconstructible subset is at most logarithmic. Also, in the case of global rigidity, Montgomery, Nenadov, Portier and Szab\'o~\cite{MNPS} showed the corresponding conjecture fails whenever $pn < 1.1$.

In this work, we prove a stronger result that is consistent with the conjecture of Gir\~ao, Illingworth, Michel, Powierski and Scott. Furthermore, we give an asymptotic lower bound on the size of a largest reconstructible subset, which is best possible since it reaches the upper bound when $V$ is linearly independent over $\mathbb{Q}$.

\begin{theorem}
    Let $\lambda > 1$ be constant. Let $V \subseteq \mathbb{R}$ with $|V|=n$ and let the set of known distances be distributed as in $G(V, \lambda/n)$. Then, whp there exists a reconstructible subset containing all but $o(1)$-proportion of vertices from the largest component of the 2-core.
    \label{th:1}
\end{theorem}

A weaker form of this result was recently published Portier~\cite{Po} (see as well the PhD thesis by Portier in joint work with Sahasrabudhe~\cite{PoSa}). More precisely, they show that there is a reconstructible subset containing
a linearly large subset of the kernel, thus confirming the conjecture. In contrast, in Theorem~\ref{th:1} we show that, for every $\lambda>1$, one can reconstruct not just a linear proportion but almost all vertices of the kernel, and even almost all vertices of the 2-core. This result, in particular, confirms a recently stated conjecture in Portier~\cite{Po}.

Furthermore, we establish a similar result for some functions $\lambda \to 1$, $\lambda>1$; this can be viewed as a natural generalisation of the conjecture of Gir\~ao, Illingworth, Michel, Powierski and Scott. As in Theorem~\ref{th:1}, the following theorem gives the best possible asymptotic lower bound on the size of a largest reconstructible subset of the random graph, as it reaches the upper bound when $V$ is linearly independent over $\mathbb{Q}$.

\begin{theorem}
    Let $\lambda := \lambda(n) \to 1$ with $\lambda \geq 1+\omega({\ln^{-1} n})$. Let $V \subseteq \mathbb{R}$ with $|V|=n$ and let the set of known distances be distributed as in $G(V, \lambda/n)$. Then, whp there exists a reconstructible subset containing all but $o(1)$-proportion of vertices from the largest component of the 2-core.
    \label{th:2}
\end{theorem}

    Whp, for constant $\lambda>1$ there is a unique linearly large component $\mathcal{C}$ of the 2-core satisfying $|\mathcal{C}| = (c_{\lambda}+o(1))n$, where $c_\lambda$ is monotone in $\lambda$ and tends to 1 as $\lambda \to \infty$ (see, Lemma~2.16 from~\cite{FrKa}). So, by standard arguments, Theorem~\ref{th:1} can be first extended to every $\lambda:=\lambda(n) \to 1+c$, $c > 0$ and, then, further extended to every
    $\lambda:=\lambda(n) >1$ bounded away from $1$. Combining that with Theorem~\ref{th:2} we conclude the following:

\begin{corollary}
     The conclusion of Theorem~\ref{th:2} can be extended to every $\lambda := \lambda(n) = 1+\omega(\ln^{-1} n)$.
\label{cl:for_theorems}
\end{corollary}

Our methods, might be useful for finding a lower bound on the size of the largest reconstructible subset for some $\lambda = 1 + O({\ln^{-1} n})$ as well. 
However, in order to solve the problem completely, new ideas seem to be needed, since our methods most likely break around $p = \frac{1 +c(\ln^{-1}n)}{n}$ for small $c>0$. We suspect that, for such $p$, a linearly large reconstructible subset of the kernel might not exist. For further discussion and other open problems see Section~\ref{sc:7}.

In addition, from our proofs it will follow that Theorem~\ref{th:1} and Theorem~\ref{th:2} might be naturally extended to other graph models: we will show that if a graph sequence $\mathcal{C}$ satisfies several deterministic properties (given below), and $\pi:{V(\mathcal{C})} \to V$ is a uniformly random injection, then the conclusion of Theorem~\ref{th:2} holds for $\pi(\mathcal{C})$. The complete set of the required deterministic properties consists of the fact that $\mathcal{C}$ is a connected graph of order at most $n$, R1---R5 from Theorem~\ref{th:super1}, R1$'$ and R2$'$ from Lemma~\ref{lm:super2}, the conclusion of Lemma~\ref{lm:3}, and R2$^*$---R4$^*$ from Lemma~\ref{lm:super4}.

We hope that Theorem~\ref{th:super1} from Section~\ref{sc:3} may be interesting as a stand-alone statement. Roughly speaking, for a random graph $\mathcal{G}$ satisfying certain properties, it follows that every $\mathcal{G}$-rigid map contains a large (map-dependent) vertex subset on which all pairwise distances are preserved. This observation is useful, since there may be several ways to deduce from that the existence of a linearly large reconstructible subset. (We give one approach in this paper; for a different approach see Theorem~1.3 by Montgomery, Nenadov, Portier and Szab\'o~\cite{MNPS}).

Recall that the 2-core is unique and can be obtained by iteratively deleting every vertex with degree less than 2. Let $\mathcal{C}$ be a largest connected component of the 2-core. We call a {\it 2-path} any path with every inner vertex of degree 2. The {\it kernel} $\mathcal{K}$ of a graph $\mathcal{C}$ is the {\it multigraph} defined as follows
\begin{itemize}
    \item the vertex set $V(\mathcal{K})$ consists of the vertices of $\mathcal{C}$ with degree at least 3;
    \item for every 2-path in $\mathcal{C}$ with end-vertices $u, v \in V(\mathcal{K})$ we include an edge between $u$ and $v$ in $E(\mathcal{K})$.
\end{itemize}

We develop our methods to prove Theorem~\ref{th:1} and Theorem~\ref{th:2} simultaneously. In addition to the ideas used to prove Theorem~\ref{th:1}, the proof of Theorem~\ref{th:2} will require an additional trick --- we will return to that after a general overview of our proofs. 

The proof of the Theorem~\ref{th:1} has three steps and consists of four lemmas. In Lemma~\ref{lm:1} we show that whp for any $\mathcal{C}$-rigid map $\varphi$ there is a {\it map-dependent} linearly large subset $U$ of $V(\mathcal{K})$ such that 
$$|\varphi(u) - \varphi(v)| = |u-v|,\quad \textnormal{for every }u, v \in U.$$ In Lemma~\ref{lm:2} we define a property $D(v)$ for $v \in V(\mathcal{K})$ and show that whp the set $\{v \in V(\mathcal{K}) : D(v) \textnormal{ holds}\}$ is reconstructible. In Lemma~\ref{lm:3} we show that whp almost all vertices from $V(\mathcal{K})$ satisfy the property $D(v)$. This shows that all but $o(1)$-proportion of vertices from the kernel make up a reconstructible subset. 
In Lemma~\ref{lm:4}, a small additional work is needed to show that, for almost all 2-paths $\mathcal{P}$ from $\mathcal{C}$ with endpoints $u, v \in V(\mathcal{K})$, every inclusion-maximal reconstructible set that contains $u$ and $v$ also contains $\mathcal{P}$. This concludes the proof of Theorem~\ref{th:1}.

The same proofs can be generalised immediately to show Theorem~\ref{th:2} for $\lambda = 1 + \Omega(\ln \ln n / \ln n)$: one can apply all the statements verbatim, and only Lemma~\ref{lm:5_ESleq} requires a simple union bound-type argument to show that~\eqref{eq:DS} holds in this case as well. However, to strengthen this to $\lambda = 1 + \omega(\ln^{-1} n)$ we will need a new idea. The main obstacle to extending the methods as they stand happens in the final union bound of Lemma~\ref{lm:5_ESleq}, which fails as the maximal degree in $\mathcal{K}$ is $\omega(1)$. That led us to the following trick: considering large subgraphs of $\mathcal{C}$ and $\mathcal{K}$ that do not contain any vertices of degree greater than 3. Using this idea, we prove Theorem~\ref{th:2} in Section~\ref{sc:ad}.

To simplify the proofs of Theorems~\ref{th:1} and \ref{th:2}, we present in Section~\ref{sc:toy} a toy example. Namely, in Proposition~\ref{pr:toy}, we show that the entire vertex set of a uniformly random 17-regular graph is reconstructible with high probability. The proof of Proposition~\ref{pr:toy} follows the same structure as the proofs of Lemma~\ref{lm:1} and Lemma~\ref{lm:2}, and, hence, may provide additional intuition for understanding their proofs on a larger scale. Even though the proof of Proposition~\ref{pr:toy} makes references to later sections, it is logically independent of them and can be read as a stand-alone statement that illustrates the overall flavour of the paper.

The paper is organized as follows. In Section~\ref{sc:toy} we show a toy example --- whp the whole vertex set of the uniformly random 17-regular graph is reconstructible. Section~\ref{sc:2} contains preliminaries and consists of five subsections. In Subsection~\ref{sc:2.1} we recall several basic inequalities. Subsection~\ref{sc:2.2} explains the contiguous model for Theorem~\ref{th:1} and Theorem~\ref{th:2} we use. Subsection~\ref{sc:2.3} provides properties of fixed degree sequence random graphs. In Subsection~\ref{sc:2.4} we show some properties of the largest component of the 2-core for the random graph. In Subsection~\ref{sc:2.5} we give some properties of $\mathcal{G}$-rigid maps. Section~\ref{sc:3} and Section~\ref{sc:4} are devoted to the proofs of Lemma~\ref{lm:1} and Lemma~\ref{lm:2} respectively. In Section~\ref{sc:5} the proofs of the remaining Lemmas~\ref{lm:3}~and~\ref{lm:4} are given and Theorem~\ref{th:1} is concluded. Theorem~\ref{th:2} is proved in Section~\ref{sc:ad}.
In Section~\ref{sc:7} we discuss our findings and propose some open problems.

\section{Toy example}
\label{sc:toy}

In this section we consider a toy example: we prove that a uniformly random $17$-regular graph on the real line is reconstructible whp. More precisely, we prove the following claim.

\begin{proposition}
    Let $V \subseteq \mathbb{R}$ and let $|V| = n$ be even. Let $\mathcal{\bar K}$ be a uniformly random 17-regular graph with $V\left(\mathcal{\bar K}\right) = V$. Then, whp, $V$ is a reconstructible subset.
    \label{pr:toy}
\end{proposition}

    Actually, Theorem~\ref{th:super1}, Lemma~\ref{lm:super2}, Lemma~\ref{lm:3}, and Lemma~\ref{lm:super4} apply to a broad class of random graphs and, in particular, imply the analogue of Proposition~\ref{pr:toy} for a uniformly random $d$-regular graph with $d \geq 3$. Nevertheless, the proof of Proposition~\ref{pr:toy} loosely follows the structure of the general proof of Theorem~\ref{th:1} and Theorem~\ref{th:2} (further referred as {\it the general proof}). So, the proof of Proposition~\ref{pr:toy} may provide additional intuition for understanding the general proof on large scale.

    The proof of Proposition~\ref{pr:toy} will be simpler than the general proof due to the following three major advantages:

    \begin{itemize}
        \item $\mathcal{\bar K}$ is regular,
        \item $\mathcal{\bar K}$ has a large number of edges, and
        \item the kernel of $\mathcal{\bar K}$ coincides with the graph itself.
    \end{itemize}

    We will deduce Proposition~\ref{pr:toy} from Proposition~\ref{pr:toy_main} given below. In Proposition~\ref{pr:toy_main} we show that for specific random graphs $\mathcal{K}$, whp $V$ is reconstructible in $\mathcal{K}$. In order to deduce Proposition~\ref{pr:toy} from Proposition~\ref{pr:toy_main} we will show that the uniformly random 17-regular graph is contiguous to a properly chosen distribution on the graphs $\mathcal{K}$ from Proposition~\ref{pr:toy_main}.

    \begin{proposition}
    Let $V \subseteq \mathbb{R}$ and let $|V| = n$ be even. Let $\mathcal{K}^*$ be an arbitrary 17-regular graph with $V(\mathcal{K}^*) = V$ satisfying the following:
    \begin{equation}
        \textnormal{The eigenvalues }\lambda_1 \geq \ldots \geq \lambda_{n} \textnormal{ of the adjency matrix of }\mathcal{K^*}\textnormal{ staisfy }|\lambda_2|, \ldots, |\lambda_n| \leq 8.01.
        \label{eq:prob7}
    \end{equation}
    Let $\mathcal{K} = \pi(\mathcal{K}^*)$, where $\pi:V\to V$ is a uniformly random permutation. Then, whp, $V$ is a reconstructible subset in $\mathcal{K}$.
    \label{pr:toy_main}
    \end{proposition}

    \begin{proof}[Proof that Proposition~\ref{pr:toy_main}$\implies$Proposition~\ref{pr:toy}]

    Notice that the uniformly random 17-regular graph $\mathcal{\bar K}$ satisfies~\eqref{eq:prob7} whp due to Theorem~1 by Friedman~\cite{Fr}. We construct a contiguous model $\mathcal{K}$ for  $\mathcal{\bar K}$ using a distribution-preserving transformation. The uniform distribution over 17-regular graphs is invariant under every vertex permutation $\pi$. Note that the vertex permutations break the 17-regular graphs into equivalence classes in a way that all graphs from the same equivalence class have the same probability. So, we simply set $\mathcal{K}^{*}$ in the following way. We set $\mathcal{K}^* = \mathcal{\bar K}$ when $\mathcal{\bar K}$ satisfies~\eqref{eq:prob7} and we set $K^*$ to be an arbitrary graph otherwise. 
    \begin{remark}
        This probabilistic trick is simple but powerful --- several times in the general proof it allows us to overcome obstacles which would otherwise require entirely new ideas.
    \end{remark}

    \end{proof}

    It remains for us to show Proposition~\ref{pr:toy_main}.

\begin{proof}[Proof of Proposition~\ref{pr:toy_main}]

    Let us prove the following auxiliary result: for every $17$-regular graph $\mathcal{\hat K}$ satisfying~\eqref{eq:prob7} and every $\mathcal{\hat K}$-rigid map $\varphi$, there is a linearly large subset $U \subseteq V$ satisfying 
    \begin{equation}
        |u - v| = |\varphi(u) - \varphi(v)|,\quad \textnormal{for every }u, v \in U.
        \label{eq:repl_dep_eq}
    \end{equation}
    
    This auxiliary statement corresponds to Lemma~\ref{lm:1} in the general proof.

    Consider a $\mathcal{\hat K}$-rigid map $\varphi:V \to \mathbb{R}$ such that there is no set $U \subseteq V$ of size $0.01|V|$ satisfying~\eqref{eq:repl_dep_eq}. Wlog suppose that the set $F :=\{uv \in E(\mathcal{\hat K}) : \varphi(u) - \varphi(v) = u-v\}$ contains at least half of the edges $E(\mathcal{\hat K})$. Let $\mathfrak{P}$ denote the natural partition of the vertex set $V$ into the connected components of the graph $(V, F)$. Note that equation~\eqref{eq:repl_dep_eq} holds for every part $U$ of $\mathfrak{P}$ and, hence, $|U| < 0.01|V|$, by the definition of $\varphi$. In order to find a contradiction, it remains to show that $|F| <  \frac{|E(\mathcal{\hat K})|}{2}$. If $uv \in F$ then $u$ and $v$ lie in the same part $U$ of $\mathfrak{P}$. Due to~\eqref{eq:prob7} and the expander mixing lemma, we have
    \begin{equation}
        \textnormal{For every } U \subseteq V\textnormal{ with } |U| \leq 0.01|V| \textnormal{ it holds that } |\{uv \in E(\mathcal{K}): u \in U, v \notin U\}| > \frac{17}{2}|U|.
        \label{eq:more_than_half}
    \end{equation}
    So, applying~\eqref{eq:more_than_half} to every $U$ from $\mathfrak{P}$, we conclude that $|F| < \frac{|E(\mathcal{\hat K})|}{2}$ --- a contradiction to the initial assumption on the size of $F$.

    \begin{remark}
        Compared to the proof shown above, the general proof of Lemma~\ref{lm:1} will require a more careful analyses of the structure of the graph's partition.
    \end{remark}

    Now, let us prove that, whp, the whole vertex set $V$ is reconstructible. This statement corresponds to Lemma~\ref{lm:2} and  Lemma~\ref{lm:3} in the general proof. 
    
    We will prove that $V$ is reconstructible by taking a union bound. First, for every $U \subseteq V$, we define an event $B(U)$ and show that $\bigcup_{U \subseteq V}B(U)$ covers the event that $V$ is not reconstructible in $\mathcal{K}$. Next, we show a sufficiently strong upper bound on $\mathbb{P}(B(U))$ so that $\sum_{U \subseteq V}\mathbb{P}(B(U)) = o(1).$

    Consider an arbitrary fixed permutation $\pi:V \to V$. Suppose that $V$ is not reconstructible in $\mathcal{K} = \pi(\mathcal{K}^*)$. We claim that there exists a (non-empty) subset $U \subseteq V$ such that the following event $B(U)$ holds: there exists a (not necessarily injective) $\mathcal{K}$-rigid map $\varphi: \pi(V)\to \mathbb{R}$ such that
    \begin{itemize}
        \item[B1] $U$ induces a connected subgraph in $\mathcal{K}^*$,
        \item[B2] $|U| \leq 0.99|V|$,
        \item[B3] $\varphi|_{\pi(V\setminus U)} = id$, and
        \item[B4] $\varphi(u) \neq u$ for every $u \in {\pi(U)}$.
    \end{itemize}
    Let $\varphi: \pi(V)\to \mathbb{R}$ be a $\mathcal{K}$-rigid map certifying that $V$ is not reconstructible in $\mathcal{K}$. As shown at the start of the proof there is a subset $W\subseteq V$ of size at least $0.01|V|$ such that $\pi(W)$ satisfies~\eqref{eq:repl_dep_eq}. Wlog, we can assume that  
    $\varphi|_{\pi(W)} = id$, and then wlog $W = \{x : \varphi(\pi(x))=\pi(x)\}$. Now set $U\subseteq V$ to be an arbitrary inclusion-maximal subset such that 
    \begin{itemize}
        \item $\mathcal{K}^*[U]$ is connected and
        \item $\varphi(\pi(u)) \neq \pi(u)$ for every $u \in U$.
    \end{itemize}
    First, B1 and B4 follow instantly. Next, $U \cap W = \emptyset$ and, thus, B2 holds. Finally, in order to conclude B3 it remains to show that wlog we can set $\varphi|_{\pi(V\setminus (W\sqcup U))} = id$. We need to show that the distance along the edges adjacent to $\pi(V\setminus (W\sqcup U))$ is preserved. Since $U$ is maximal and connected there are no edges between $U$ and  $V\setminus(U \sqcup W)$. So, the distance the distance along the edges adjacent to $\pi(V\setminus (W\sqcup U))$ is preserved as $\varphi|_{V\setminus U} = id$. 

    Now, let us assume that $\pi: V\to V$ is uniformly random. Consider a subset $U \subseteq V$ with $|U| \leq 0.99|V|$ and let us bound from above the probability $\mathbb{P}(B(U))$.
    
    Due to~\eqref{eq:prob7}, the expander mixing lemma and the fact that $\mathcal{K}^*$ is 17-regular imply the following property, which is slightly stronger than being a (0.99,~0.003) vertex expander (see, Subsection~\ref{sc:2.5} for the definition of $(c, \alpha)$ vertex expansion).
    $
        \textnormal{For every } W \subseteq V\textnormal{ with } |W| \leq 0.99|V| \textnormal{ it holds that}$ 
        $$ |N_{\mathcal{K}*}(W)| \geq \max(0.003|W|, 3).
    $$
    Hence, since $|U| \leq 0.99|V|$, for $N:= N_{\mathcal{K}^*}(U)$ it holds that 
    \begin{equation}
    |N|\geq \max(0.003|U|, 3).
        \label{eq:more_than_2}
    \end{equation}

    \begin{figure}[!ht]
    \centering
    \includegraphics[scale=0.2]{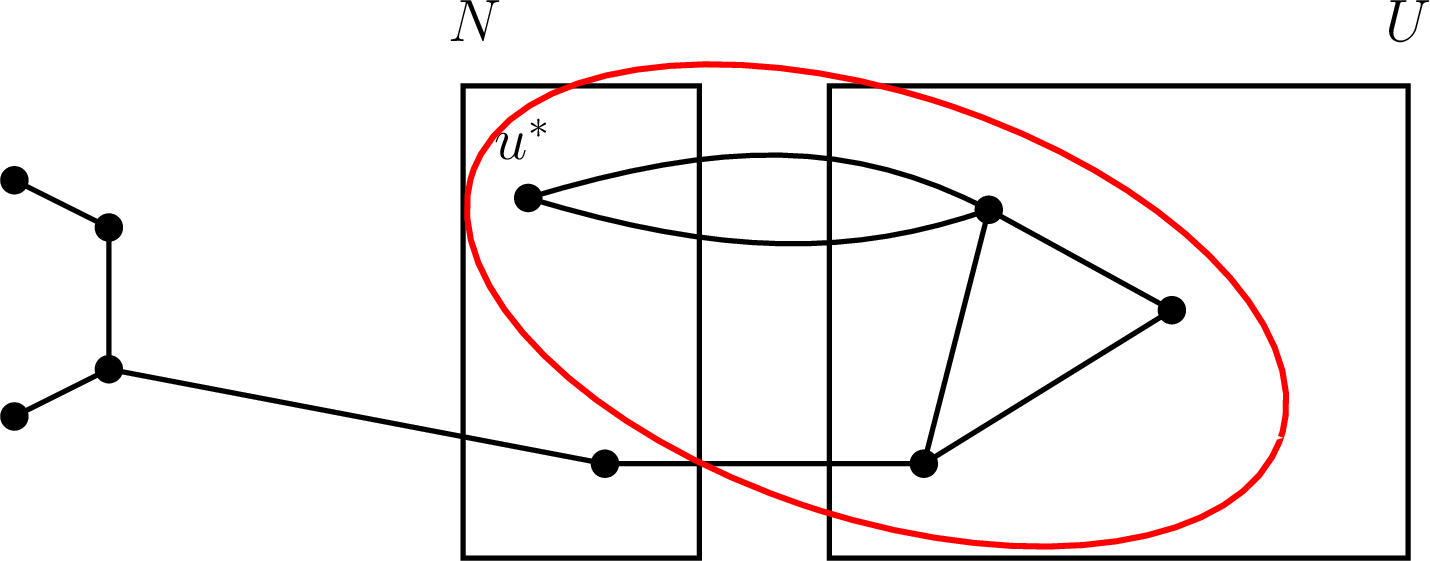}
    \caption{The figure illustrates the  structure of $\mathcal{K^*}, U$, and $N$. The red subset consists of vertices $v \in V(\mathcal{K})$ with fixed random value $\pi(v)$.
}
    \label{fig:2}
\end{figure}

   Let $u^* \in N$ be chosen arbitrary (see, Figure~\ref{fig:2}). Let us fix the values of $\pi|_{\{u^*\} \sqcup U}$. Hence, the remaining probability space is described by the uniformly random bijection 
   $$\pi:{V \setminus (\{u^*\} \sqcup U)} \to V\setminus \pi(\{u^*\} \sqcup U).$$

    Let us observe that if B1 holds then for all extensions of $\pi$ to $V$ the total number of (not necessarily injective) $\mathcal{K}$-rigid maps $\varphi:V \to \mathbb{R}$ satisfying B3 is at most $2^{|U|}$. Indeed, regardless of the choice of the extension of $\pi$ to $V$, by B3 we have $\varphi|_{V\setminus \pi(U)} = id$ and, hence, we need to bound the number of ways to define $\varphi$ on $U$. In  order to do that consider a subtree $\mathcal{T}$ with vertex set $\{u^*\} \sqcup U$, which exists due to B1. Then, 
    for every edge $uv \in E(\mathcal{T})$ we have 
    \begin{equation}
        \varphi(\pi(v)) - \varphi(\pi(u)) = \pm (\pi(v) -\pi(u)).
        \label{eq:decide_sign}
    \end{equation}
    Since by B3 $\varphi(\pi(u^*)) = \pi(u^*)$ deciding the signs in~\eqref{eq:decide_sign} over the edges of $\mathcal{T}$ uniquely determines $\varphi|_{U}$ and, hence, concludes the bound (see Remark~\ref{rm:sigma} for further information).
    
    Fix an arbitrary $\mathcal{K}|_{\pi(\{u^*\} \sqcup U)}$-rigid map $\varphi:\pi(\{u^*\} \sqcup U) \to  \mathbb{R}$ of the graph such that
    \begin{itemize}
        \item $\varphi(\pi(u^*)) = \pi(u^*)$ and
        \item $\varphi(\pi(u)) \neq \pi(u)$ for every $u \in U$.
    \end{itemize}
    Let us bound the probability that this $\mathcal{K}|_{\pi(\{u^*\} \sqcup U)}$-rigid map can be extended to the whole graph $\mathcal{K}$ in a way that it satisfies B3 and B4. Notice that we actually already know the only possible extension --- $\varphi|_{\pi(V) \setminus \pi(\{u^*\} \sqcup U)} = id$ due to B3.
    
    Consider an arbitrary vertex $w \in N \setminus\{u^*\}$ and let $u \in U$ be a neighbour of $w$. Recall that, by B3, 
    $$|\pi(w) - \varphi(\pi(u))| = |\varphi(\pi(w)) - \varphi(\pi(u))|  = |\pi(w) - \pi(u)|.$$ 
    As $\varphi(\pi(u)) \neq \pi(u)$, there is just one solution in $\pi(w)$ to this equations --- that is the midpoint between $\pi(u)$ and $\varphi(\pi(u))$. So, there is at most one position on the real line, where the random value of $\pi(w)$ can land. Let us fix one by one the values $\pi(w)$ for $\min(|N \setminus\{u^*\}|, 0.003|V|)$ of vertices $w \in N \setminus\{u^*\}$. Each of them lands in the desired position with probability at most $(0.007|V|)^{-1}$ conditioned on the whole history of the process. So, by~\eqref{eq:more_than_2}, we get an upper bound
    $$\mathbb{P}(\varphi \textnormal{ can be extended to }\mathcal{K}\textnormal{ in a way that satisfies B3 and B4}) \leq \frac{1}{(0.007|V|)^{\max(0.003|U|-1, 2)}},$$
    which gives us an overall upper bound
    \begin{equation}
        \mathbb{P}(B(U)) \leq \frac{2^{|U|}}{(0.007|V|)^{\max(0.003|U|-1, 2)}}.
        \label{eq:bound9}
    \end{equation}
    
    Recall, that we proved that the event saying $V$ is not reconstructible in $\mathcal{K} = \pi(\mathcal{K}^*)$ implies the event $B(U)$ for some $U\subseteq V$. So, by B1 and B4, it remains for us to show that
    \begin{equation*}
        \sum_{U\subseteq V,\ |U|\leq 0.99|V|,\ U \textnormal{ induces a connected subgraph in }\mathcal{K}^*}\mathbb{P}(B(U)) = o(1).
    \end{equation*}

   By a well-known lemma, the number of sets $U \subseteq V$ of order $t$ inducing a connected subgraph in $\mathcal{K}^*$ is at most $|V|(16e)^{t-1}$ (see Bollob\'as~\cite{Bo} or Claim~\ref{cl:graph_bound} for the full statement). So, it sufficient to show the following upper bound uniformly over $t \in [0.99|V|]$ and $U \subseteq V$ with $|U| = t$
    \begin{equation*}
        \mathbb{P}(B(U)) = o\left(\frac{1}{(50e)^{t}|V|}\right),
    \end{equation*} 
    which is indeed true due to~\eqref{eq:bound9}.

\end{proof}

   \begin{remark}
       In the general proof, by extending these ideas in Lemma~\ref{lm:2} we define a specific event $D$ and show that the set of all kernel vertices satisfying $D$ is reconstructible. Then, in Lemma~\ref{lm:3} an additional technical argument will be needed to conclude that almost all vertices of the kernel satisfy event $D$. 
   \end{remark}

\section{Preliminaries}
\label{sc:2}
\subsection{Basic inequalities}
\label{sc:2.1}

In this subsection we give upper bounds on the following three quantities: the number of connected induced subgraphs with some constraints, the number of partitions of a graph into connected subgraphs, and the number of spanning subtrees.

The following claim can be found, for example, in Bollob\'as~\cite{Bo}, pp130---133.

\begin{claim}
    A (multi-)graph with maximum degree $\Delta \geqslant 3$ has at most $(e(\Delta - 1))^n$ connected induced subgraphs with $n + 1$ vertices that include a given vertex.
    \label{cl:graph_bound} 
\end{claim}

\begin{claim}
    For $c\in(0,1)$ and a (multi-)graph $\mathcal{K}$ of maximal degree at most $|V(\mathcal{K})|^c/10$, the number of partitions $V_1 \sqcup \ldots \sqcup V_s = V(\mathcal{K})$ such that, for every $i \in [s]$, $V_i$ induces a connected (multi-)graph is at most $|V(\mathcal{K})|^{c|V(\mathcal{K})|}$.
    \label{cl:partitions_number}
\end{claim}

\begin{proof}
    For a (multi-)graph $\mathcal{K}$ let $B_\mathcal{K}$ denote the number of partitions of the vertex set $V(\mathcal{K})$ into connected subgraphs. Fix $\Delta > 0$. For $k >  0$ let
    $$B_{k} := \max\{B_{\mathcal{K}}: \Delta(\mathcal{K}) = \Delta, |V(\mathcal{K})| = k\}.$$
    
    It suffices to prove by induction on $k$ that $B_{k} \leqslant (10\Delta)^{k}$. The base of the induction $B_{1} \leqslant 10\Delta$ is trivial. Set $B_{0} = 1$.

    Now, suppose the statement holds for every $k \leqslant t-1$ and consider $k:=t$. Consider a $\textnormal{(multi-)graph}$ $\mathcal{K}$ with $|V(\mathcal{K})|=k$ and consider an arbitrary vertex $v \in V(\mathcal{K})$. Note that every partition of $V(\mathcal{K})$ into connected subgraphs induces a part that contains $v$. Suppose this part has order $\ell > 0$. By Claim~\ref{cl:graph_bound}, the number of vertex sets $S$ of order $\ell$ that induce a connected subgraph and contain $v$ is at most $(e(\Delta-1))^{\ell-1}$. Once we chose the part $S$ containing $v$, the number of partitions of the remaining graph is at most $B_{k-\ell}$. So, we get the overall bound
    $$B_k \leqslant \sum_{\ell =1}^k (e(\Delta-1))^{\ell-1}B_{k-\ell} \leqslant (10\Delta)^{k-1}(1 + e/10 + e^2/100 +\ldots) \leqslant (10\Delta)^k.$$
\end{proof}

\begin{claim}
     A (multi-)graph $\mathcal{G}$ with maximum degree $\Delta > 0$ has at most $e(e\Delta/2)^{|V(\mathcal{G})|-1}$ spanning trees.
     \label{cl:num_of_spanning_trees}
\end{claim}
    
\begin{proof}
    A spanning tree has $n-1$ edges and the number of ways to choose $n-1$ edges from the set of all edges is at most ${n\Delta/2 \choose n-1}\leqslant \left(\frac{en\Delta/2}{n-1}\right)^{n-1}\leq \left(1+\frac{1}{(n-1)}\right)^{n-1}(e\Delta/2)^{n-1} \leqslant e(e\Delta/2)^{n-1}$.
\end{proof}

\subsection{Contiguous model}
\label{sc:2.2}

In this subsection, we present in Claim~\ref{cl:contiguous} a contiguous model for the 2-core of the graph formed by the set of known distances, as described in Theorem~\ref{th:1}.

We use the following claim, that is a corollary of Theorem~1 in Ding, Lubetzky and Peres~\cite{DLP}, for a constant $\lambda$, and of Theorem~2 in Ding, Kim,  Lubetzky and Peres~\cite{DKLP}, for $\lambda \to 1$.

\begin{claim}
    Let $\lambda > 1$ be either a constant or a function tending to $1$ with $(\lambda -1) = \Omega( \frac{1}{\ln n})$. Let $\mu < 1$ be the conjugate of $\lambda$, that is $\mu e^{-\mu} = \lambda e^{-\lambda}$. Then the largest component of the 2-core of $G(n, \lambda/n)$ is contiguous to the following model $\mathcal{L}$:
    \begin{enumerate}
        \item Let $\Lambda$ be Gaussian $\mathcal{N}(\lambda - \mu, 1/n)$ and let $D_u \sim \textnormal{Poisson}(\Lambda)$ for $u \in [n]$ be i.i.d., conditioned that $\sum D_u\mathbf{1}_{D_u \geqslant 3}$ is even. 
        
        Let $N_k := |\{u : D_u = k\}|$ and $N = \sum_{k \geqslant 3}N_k$.

        \item Select a random multigraph $\mathcal{K}$ on $N$ vertices, uniformly among all multigraphs with $N_k$ vertices of degree $k$ for $k \geqslant 3$.

        \item Replace the edges of $\mathcal{K}$ by paths of i.i.d. $\textnormal{Geom}(1 - \mu)$ lengths, where $\textnormal{Geom}(1 - \mu)$ has support $\mathbb{Z}_{>0}$.
    \end{enumerate}
    \label{cl:L-contiguous}
\end{claim}

The original descriptions of the models~$\mathcal{L}$ have three steps, which differ slightly from the three steps in Claim~\ref{cl:L-contiguous}. We omit the third step of the original models~$\mathcal{L}$, which generates the vertices outside of the 2-core. We also split the first step of the original $\mathcal{L}$ into two, since it will be more convenient for us to use it this way.

Note that the first two steps of the model~$\mathcal{L}$ from Claim~\ref{cl:L-contiguous} generate the kernel $\mathcal{K}$. In particular, the first step gives us the degree sequence of the kernel.

\begin{claim}
    Let $\lambda > 1$ be either a constant or a function tending to $1$ with $(\lambda -1) = \Omega( \frac{1}{\ln n})$. Let $V \subseteq \mathbb{R}$ with $|V| = n$. Define the edge set $E$ according to the following model $\mathcal{U}$:
    \begin{itemize}
        \item let $\mathcal{C}$ be  a random graph distributed according to the model $\mathcal{L}$ from Claim~\ref{cl:L-contiguous};
        \item if $|V(\mathcal{C})| \leqslant n$, sample a uniformly random injection $\pi: V(\mathcal{C}) \rightarrow V$ and define $E  := \pi(E(\mathcal{C}))$;
        \item if $|V(\mathcal{C})| > n$ define $E := V^{(2)}$.
    \end{itemize}
    Then, the largest component of the 2-core of $G(V, \lambda/n)$ is contiguous to $\mathcal{U}$.

    \label{cl:contiguous}
\end{claim}

\begin{proof}
    It is easy to see that the graph $\mathcal{G}$ generated according to $\mathcal{L}$ from Claim~\ref{cl:L-contiguous} has whp less than $n$ vertices. 
    In particular, in the case $|V(\mathcal{C})| > n$, we can define the edge set $E$ as $E := V^{(2)}$, $E := \emptyset$, or any other edge set.

    By symmetry, the image $\pi$ of $\mathcal{G}(n, \lambda/n)$ coincides with $\mathcal{G}(n, \lambda/n)$ up to a $o(1)$-measure. Since the largest component of the 2-core of $G(n, \lambda/n)$ is contiguous with $\mathcal{L}$, so the largest component of the 2-core of $G(V, \lambda/n)$ is contiguous to $\mathcal{U}$.

\end{proof}

\subsection{Fixed degree sequence random graph}
\label{sc:2.3}

In this subsection, we give several properties of the degree sequence of the kernel of the graph $\mathcal{C}$ from Claim~\ref{cl:contiguous}.

For a given degree sequence $\mathbf{d} := (d_{n,1}, \dots, d_{n,n})$, let $\mathbf{G}(\mathbf{d})$ denote the number of labelled multigraphs with that degree sequence. Define $d_{\max} := \max(d_{n,1}, \dots,d_{n,n})$ and let 
$$
M_r := \sum_{i=1}^{n} [d_i]_r.
$$

Recall that the first step of the model $\mathcal{L}$ described in Claim~\ref{cl:L-contiguous} gives us the degree sequence of the kernel. Wlog, we may assume this degree sequence is non-increasing. 
The following claim estimates the number of multigraphs with given degree sequence (Theorem 1.1 from~\cite{GrMc} by Greenhill and McKay applied to multigraphs with small maximal degree).

\begin{claim}
    Let $\mathbf{d} := (d_1, \ldots, d_n)$ be a degree sequence. Let $M_1 \to \infty$ and $d_{max}^3 = o(M_1)$. Then,
$$\mathbf{G}(\mathbf{d})
= \frac{M_1!}{(M_1/2)!2^{M_1/2}d_1!\ldots d_n!} \exp\left(\frac{M_2}{2M_1} + \left(\frac{M_2}{2M_1}\right)^2 + O\left(\frac{d_{max}^4}{M_1}\right)\right)
$$
\label{cl:graphs_number}
\end{claim}

The benefit of this claim is that it allows us to bound the probability that a multigraph contains a fixed subgraph. Indeed, let $\mathcal{G}$ be a uniformly random multigraph with degree sequence $\mathbf{d}$, and let $\mathcal{X}$ be a fixed graph with degree sequence $\mathbf{x}$ and the same vertex set. Then, 
$$
\mathbb{P}(\mathcal{X}\subseteq\mathcal{G}) \leq \frac{\mathbf{G}(\mathbf{d - x})}{\mathbf{G}(\mathbf{d})}.
$$

Also, a simple yet useful corollary is as follows.

\begin{claim}
    Let $\mathbf{d} := (d_1, \ldots, d_n)$ be a degree sequence with $M_1 \to \infty$ and $d_{max}^5 = o(M_1)$. Let $\mathcal{G}$ be a multigraph taken uniformly at random from the set of all multigraphs satisfaying the degree seuence $\mathbf{d}$. Then, for every $i, j \in [n]$,
    $$\mathbb{P}(ij \in E(\mathcal{G})) = \frac{d_id_j}{M_1} + O\left(\frac{d_{max}^5}{M_1^2}\right).$$
\label{cl:edge_prob}
\end{claim}

\subsection{The properties of random graph}
\label{sc:2.4}

In this subsection, we give several properties that the graph $\mathcal{C}$ from Claim~\ref{cl:contiguous} satisfies whp.

Recall that the first step of the model $\mathcal{L}$ from Claim~\ref{cl:L-contiguous} generates the degree sequence of the kernel $\mathcal{K}$.  Notice that in this degree sequence every entry is either 0 or at least 3. The following claim gives several properties that a typical degree sequence satisfies.

\begin{claim}
    Let $\lambda$ be as in Claim~\ref{cl:L-contiguous}. Let $\mathbf{d} := (d_1, \ldots, d_n)$ be the non-increasing degree sequence generated during the first step of $\mathcal{L}$, Claim~\ref{cl:L-contiguous}.
    There exist constants $c_v, c_a,c_e>0$ that are absolute for all $\lambda \to 1$ such that whp:
\begin{itemize}
    \item[S1] the number of non-zero entries in $\mathbf{d}$ equals $({c_v+o(1)}){(\lambda-1)^3}n$;
    \item[S2] the sum $d_1 + \ldots + d_n$ equals $({c_e+o(1)}){(\lambda-1)^3} n$;
    \item[S3] the maximum degree $d_{\max} := d_1$ satisfies $d_{\max} = o(\ln n)$;
    \item[S4] for every positive integer $i \leqslant n/2$, the following holds:
    $$
    d_1 + \ldots + d_i \leqslant c_a i \frac{\ln n}{\ln i};
    $$
\end{itemize}

    \label{cl:S}
\end{claim}

\begin{proof}

The only non-trivial statement is S4. Let us show it.

First, note that, whp, $\Lambda = \lambda - \mu + o(n^{2/3})$ for $\Lambda$ as in $\mathcal{L}$, Claim~\ref{cl:L-contiguous}. Also, observe that the sum of $i$ independent Poisson($\Lambda$) variables follows a Poisson($i\Lambda$) distribution. 

Next, let us recall the Chernoff bound for Poisson distributed variables. For a real $\mu > 0$ and integer $t> \mu$ the following holds.
\begin{equation}
    \mathbb{P}(\textnormal{Poisson}(\mu) > t) \leqslant \exp(t + t \ln \mu - \mu - t\ln t).
    \label{eq:2.4}
\end{equation}
We assume that $c_a$ is large enough to make the Chernoff bound applicable. Also, denote 
$$N:= c_a i \frac{\ln n}{\ln i}.$$

Now, consider $n$ independent Poisson($\Lambda$) variables. If the inequality in S4 fails for some $i$, then there exists an $i$-subset of these variables that violates the inequality. Therefore, we apply a union bound over every $i \in [\lfloor n/2\rfloor]$ and every $i$-subset of the variables:

\begin{align*}
    \mathbb{P} \biggl(\lnot \textnormal{ S4} \biggr ) &\leqslant \sum_{i =1}^{n/2} {n \choose i} \mathbb{P}\biggl(\textnormal{Poisson}(i \Lambda) >  N \biggr ) \\ 
    &\stackrel{\eqref{eq:2.4}}\leqslant \sum_{i =1}^{n/2} {n \choose i} \exp(N + N \ln (i\Lambda) - i\Lambda - N \ln N)\\
    &\stackrel{\textnormal{for sufficiently large } c_a\textnormal{ and }n}\leqslant \sum_{i =1}^{n/2} {n \choose i} \exp(- 0.5 N\ln N)\\
    &\leqslant \sum_{i =1}^{n/2} \left(\frac{ne}{i}\right)^i \exp(- 0.5 N\ln N) \\
    &= \sum_{i =1}^{n/2} \exp\left(i\bigl(1 + \ln n - \ln i - 0.5c_a \frac{\ln n}{\ln i} \ln N\bigr)\right) = o(1),
\end{align*}
for sufficiently large $c_a>0$.

\end{proof}

The following claim consists of connectivity and two bounds on the size of a typical graph generated according to~$\mathcal{L}$, Claim~\ref{cl:L-contiguous}. The claim follows trivially from simple calculations for $\mathcal{L}$.

\begin{claim}
    For $\lambda$ and $\mathcal{C}$ from Claim~\ref{cl:L-contiguous} it holds that
    \begin{itemize}
        \item $\mathcal{C}$ is whp connected;
        \item there exists $c>0$ such that if $\lambda \to 1$ then whp $|V(\mathcal{C})| = (c+o(1))(\lambda - 1)^2n$;
        \item if $\lambda$ is a constant then there exists $c>0$ such that whp $|V(\mathcal{C})| \leq (1-c)n$.
    \end{itemize}
    \label{cl:core_size}
\end{claim}

Below, we give three claims that describe the properties typically satisfied by the kernel $\mathcal{K}$, which is generated in the first two steps of the model~$\mathcal{L}$, Claim~\ref{cl:L-contiguous}. 

\begin{claim}
    Let $\mathcal{K}$ be the kernel generated in the first two steps of Claim~\ref{cl:L-contiguous}. For every $\varepsilon > 0$ there exists $c>0$ such that whp for every $U \subseteq V(\mathcal{K})$ with $|U| \leqslant c|V(\mathcal{K})|$ it holds that 
    \begin{equation}
        |E(\mathcal{K}|_{U})| \leqslant (1 + {\varepsilon})|U|.
        \label{eq:kernel1}
    \end{equation}
    \label{cl:kernel1}
\end{claim}

\begin{proof}
    Let $k := |V(\mathcal{K})|$ and let $V := V(\mathcal{K})$. Let us expose the first step of $\mathcal{L}$ form Claim~\ref{cl:L-contiguous} and let the exposed degree sequence $\mathbf{k}$ satisfy S1---S4 from Claim~\ref{cl:S}, which it satisfies whp. So, now, we can assume that $\mathcal{K}$ is a uniformly random multigraph with degree sequence $\mathbf{k}$. 

    We want to prove using a union bound that no subset $U \subseteq V$ violates~\eqref{eq:kernel1}. We will first note that if $U \subseteq V$ violates~\eqref{eq:kernel1} then it induces a subgraph with at least $\lceil(1+\varepsilon)|U|\rceil$ edges and bounded number of loops and multiedges. Then, in order to show the union bound, we will bound the number of graphs on $U$ with $\lceil(1+\varepsilon)|U|\rceil$ and bounded of loops and multiedges. Then, we will show that none of this graphs appear in $\mathcal{K}$ whp.
    
    Note that by S4 of Claim~\ref{cl:S}, whp all but $\sqrt{k}$ vertices have degree at most $3c_a$ and the remaining $\sqrt{k}$ vertices have degree at most $\ln k$. By S1 and S2 the number of edges in $\mathcal{K}$ is linear in $k$, and so, for some $C>0$, by Claim~\ref{cl:edge_prob} whp the number of multiedges and loops in $\mathcal{K}$ is at most $C$. 
    
    For every subset of vertices $U \subseteq \mathcal{K}$, the number of graphs on $U$ with $\lceil(1+\varepsilon)|U|\rceil$ edges and at most $C$ loops and multiple edges is at most $$|U|^{O(1)}{|U|^2/2 \choose (1+\varepsilon)|U|} \leqslant |U|^{O(1)}\left(10|U|\right)^{(1+\varepsilon)|U|}.$$

    Additionally, the number of ways to choose $U \subseteq V$ with $|U| = m$, $m>0$, is at most
    $${k \choose m} \leqslant \left(\frac{ek}{m}\right)^m.$$

    For a fixed graph $\mathcal{G}$ on $U$ with at least $(1+\varepsilon)|U|$ edges let us bound the probability that $\mathcal{G} \subseteq \mathcal{K}$. Consider an arbitrary order on the edges of $\mathcal{G}$. Sequentially we expose every edge $e \in E(\mathcal{G})$ in $\mathcal{K}$. If $e \not \in E(\mathcal{K})$, the process fails. Hence, at every step of the process the remaining graph is uniformly random among the graphs with the remaining degree sequence. This allows us to use Claim~\ref{cl:edge_prob} to upper bound $\mathbb{P}(e \in E(\mathcal{K}))$ at every step of the process. We give two upper bounds on $\mathbb{P}(e \in E(\mathcal{K}))$. The weaker bound $$\mathbb{P}(e \in E(\mathcal{K})) \leqslant\frac{c_a^2\ln^2 k}{k}$$
    holds for every edge, and the stronger bound
    $$\mathbb{P}(e \in E(\mathcal{K})) \leqslant\frac{9c_a^2}{k}$$
    holds for all but $3c_a\sqrt{k}$ edges that are adjacent to $\sqrt{k}$ vertices with the largest degrees in $\mathcal{K}$.
    
    Let us take a union bound over $m \in [ \lfloor c k \rfloor]$, $c >0$ is small enough, over all subsets $U \subseteq V(\mathcal{K})$ of size $m$ and all multigraphs $\mathcal{G}$ on $U$ with $\lceil(1+\varepsilon)|U|\rceil$ edges and at most $C$ loops and multiedges of the probability that $\mathcal{G}$ is a subgraph of $\mathcal{K}$. Our bound is
    \begin{align*}
        \mathbb{P}(\textnormal{there exists }U &\subseteq V\textnormal{ violating~\eqref{eq:kernel1}}) \leq\\
        &\leq\sum_{m \leq c k }\left(\frac{ek}{m}\right)^m \cdot m^{O(1)}\left(10m\right)^{(1+\varepsilon)m} \cdot (\ln ^2k)^{\min(m, 3c_a\sqrt{k})}\left(\frac{9c_a^2}{k}\right)^{(1+\varepsilon)m}\\
        &\leqslant \sum_{m \leq c k }\left(\frac{O(1)\cdot m}{k}\right)^{\varepsilon m/10} = o(1),
    \end{align*}
    for small enough $c > 0$.
    
\end{proof}

\begin{claim}
    Let $\mathcal{K}$ be the kernel generated in the first two steps of Claim~\ref{cl:L-contiguous} and let $k:= |V(\mathcal{K})|$. Whp, 
    the ball of radius $\lfloor \ln \ln k \rfloor$ centred at $v \in V(\mathcal{K})$ is
    \begin{itemize}
        \item a tree, for all but $o(k)$ vertices $v$;
        \item a tree plus at most one additional edge, for every $v$.
    \end{itemize}
    \label{cl:kernel2}
\end{claim}

\begin{proof}
Consider a vertex $v \in V(\mathcal{K})$. We expose the ball of radius $\lfloor \ln \ln k \rfloor$ around $v$ level by level: at the first step, we expose the neighbours of $v$, at the second step, the vertices at distance $2$, and so on. At each step, we reveal the edges connecting the current boundary to new vertices, one edge at a time.

By S1 and S3 from Claim~\ref{cl:S}, we can assume that the degree of every vertex in $\mathcal{K}$ is $o(\ln k)$. Therefore, the total number of vertices and edges exposed in the process is at most
$$
o\left((\ln k)^{\ln \ln k}\right) = o(k^{1/3}).
$$
When we reveal a new edge, the probability that it connects to a vertex already exposed (i.e., forms a cycle) is $O(\ln^2 k / k)$, by Claim~\ref{cl:edge_prob} and by the upper bound on the maximal degree of $\mathcal{K}$. Hence, the probability that such a cycle appears at least twice during the exposure of the ball is $o(k^{-1})$.

So, the first part of Claim~\ref{cl:kernel2} follows from Markov's inequality, and the second part follows form taking the union bound over all $v \in V(\mathcal{K})$.
\end{proof}

Let $\mathcal{G}$ be a (multi-)graph, and let $U \subseteq V(\mathcal{G})$. We define the set of neighbours of $U$ in $\mathcal{G}$ as
$$
N_{\mathcal{G}}(U) := \{ v \in V(\mathcal{G}) \setminus U : vu \in E(\mathcal{G}) \text{ for some } u \in U\}.
$$

Given parameters $c \in (0,1)$ and $\alpha > 0$, we say that $\mathcal{G}$ is a \emph{$(c, \alpha)$ vertex expander} if, for every subset $U \subseteq V(\mathcal{G})$, the following holds:
$$
\textnormal{if } |U| \leqslant c|V(\mathcal{G})|,\ \textnormal{then } |N_{\mathcal{G}}(U)| \geqslant \alpha|U|.
$$

The expansion properties  of the kernel were studied in Benjamini, Gady and Nicholas~\cite{BGN} and in Ding, Kim, Lubetzky and Peres~\cite{DKLP}. For further reading about expansion see, for example, Alon and Spencer~\cite{AlSp}. 

The following claim states that the kernel of $G(n, \lambda/n)$ exhibits good vertex expansion properties.

\begin{claim}
Let $\mathcal{K}$ be the random multigraph generated in the first two steps of $\mathcal{L}$ from Claim~\ref{cl:L-contiguous}. Then, for any $c \in (0,1)$, there exists $\alpha > 0$ such that, whp, $\mathcal{K}$ is a $(1-c, \alpha)$ vertex expander.

\label{cl:expansion}
\end{claim}

\begin{proof}
    Let $k := |V(\mathcal{K})|$. Recall that every vertex in $\mathcal{K}$ is of degree at least $3$. So, by S1 and S4 from Claim~\ref{cl:S} and Claim~\ref{cl:kernel1}, for every $\varepsilon > 0$ there exists $\alpha > 0$ such that whp for every $A \subseteq V(\mathcal{K})$ with $|A| \in \{\lfloor k^{\varepsilon}\rfloor, \ldots, \lceil(1-c)k\rceil\}$ it holds that $|N_{\mathcal{K}}(A)| \geq \alpha |A|$. So, it remains to show that
    \begin{equation}
        \textnormal{for every $A \subseteq V(\mathcal{K})$ with $|A| < \sqrt{k}$ it holds that $|N_{\mathcal{K}}(A)| = \Omega( |A|)$.}
        \label{eq:expan}
    \end{equation}
    
    Let us recall that the first step of the model $\mathcal{L}$, Claim~\ref{cl:L-contiguous}, gives us the degree sequence of the random multigraph. Fix a degree sequence $\mathbf{k}$ of length $k$ satisfying 
    \begin{equation}
        \textnormal{the degrees from }\mathbf{k}\textnormal{ lie in }\{3, \ldots, \lfloor\ln k\rfloor\}.
        \label{eq:deg_pr}
    \end{equation}
    Below, we assume that $\mathcal{K}$ is the uniformly random multigraph with degree sequence $\mathbf{k}$.
    Notice that the degree sequence generated in the first step of the model $\mathcal{L}$, Claim~\ref{cl:L-contiguous}, satisfies~\eqref{eq:deg_pr} whp due to S1 and S3 of Claim~\ref{cl:S}. So, it remains for us to show that~\eqref{eq:expan} holds for $\mathcal{K}$ whp.
    
    Let $\mathcal{A}_{\alpha}$ be the set of all partitions of the vertex set $A \sqcup B \sqcup C = V(\mathcal{K})$ with $1 \leqslant |A| \leqslant \sqrt{k}$ and $|B| \leqslant \alpha |A|$.  If~\eqref{eq:expan} does not hold for some $\alpha>0$, then for some triple $(A, B, C) \in \mathcal{A}_{\alpha}$ there is no edge between $A$ and $C$. 

    Consider a triple $(A, B, C) \in \mathcal{A}_{\alpha}$, let $a := |A|$. Let us show that
    \begin{equation}
        \mathbb{P}(\textnormal{there is no edge between }A \textnormal{ and }C) =  \left(O\left(\frac{a\ln^2 k}{k}\right)\right)^{1.5 a}.
        \label{eq:2.6}
    \end{equation}
    
    In order to show this bound we iteratively expose edges from $A$ one by one and bound the probability that the other end of the exposed edge lies in $A \cup B$. At every step of the process we arbitrarily choose a vertex from $A$ that has at least one unexposed edge and expose an edge from it. So, at every step we fixed a graph $\mathcal{G}$ with degree sequence $\mathbf{g}$ and the remaining graph is a uniformly random multigraph with degree sequence $\mathbf{k}-\mathbf{g}$.  
    
    Consider an arbitrary step in the process. Let us choose a vertex $v \in A$ and expose its edge $vu$ according to the process. Notice that $|A\sqcup B| \leq (1+\alpha)a$. So, by Claim~\ref{cl:edge_prob}, conditioned on the history of the process that never had an edge between $A$ and $C$, the probability that $u \in A \cup B$ is at most $O(a\ln^2 k/k)$, due to~\eqref{eq:deg_pr}.
    By~\eqref{eq:deg_pr}, we make at least $1.5a$ steps in this process, so multiplying the bounds over the steps we conclude~\eqref{eq:2.6}.

    So, the union bound gives us
    \begin{align*}  
        \mathbb{P}(\eqref{eq:expan}\textnormal{ fails})&\leqslant \sum_{1 \leqslant a \leqslant \lfloor\sqrt{k}\rfloor} \sum_{(A, B, C) \in \mathcal{A}_{\alpha},\ |A| = a} \mathbb P (\textnormal{there is no edge between }A \textnormal{ and }C) \\
        &\stackrel{\eqref{eq:2.6}}\leqslant \sum_{1 \leqslant a \leqslant \lfloor\sqrt{k}\rfloor} \sum_{(A, B, C) \in \mathcal{A}_{\alpha},\ |A| = a} \left(O\left(\frac{a\ln^2 k}{k}\right)\right)^{1.5 a} \\
        &\leqslant \sum_{1 \leqslant a \leqslant \lfloor\sqrt{k}\rfloor} \left(O\left(\frac{k}{a}\right)\right)^{(1+\alpha)a} \left(O\left(\frac{a\ln^2 k}{k}\right)\right)^{1.5 a} = o(1),
    \end{align*}
    for small enough $\alpha > 0$.
\end{proof}

\subsection{$\mathcal{G}$-rigid map properties}
\label{sc:2.5}

In this subsection we will give two simple remarks and a claim that are useful in several steps of the proofs of Theorem~\ref{th:1} and Theorem~\ref{th:2}.

\begin{remark}
    Let $\mathcal{G}$ be a connected graph and let $\pi: V(\mathcal{G}) \to \mathbb{R}$ be an injection. To each $\pi(\mathcal{G})$-rigid map $\varphi:\pi(V(\mathcal{G})) \to \mathbb{R}$, we can associate a binary representation $\sigma \in \{-1, 1\}^{E(\mathcal{G})}$ defined as follows:
\begin{equation*}
    \textnormal{For every edge }uv \in E(\mathcal{G})\textnormal{, set }\sigma(uv) = \frac{\varphi(\pi(u)) - \varphi(\pi(v))}{\pi(u) - \pi(v)} \in \{-1, 1\}.
    \label{rm:sigma}
\end{equation*}
Recall that we call $\pi(\mathcal{G})$-rigid maps equivalent if they differ by the application of a trivial $\pi(\mathcal{G})$-rigid map (i.e. adding a constant and/or symmetry). Note that for the equivalence classes of $\pi(\mathcal{G})$-rigid maps $[\varphi]$ the map $[\varphi] \mapsto \pm\sigma$ is injective.
In particular, given $\sigma$, one can reconstruct the corresponding $\pi(\mathcal{G})$-rigid map $\varphi: \pi(V(\mathcal{G})) \rightarrow \mathbb{R}$ (if it exists) up to a trivial $E$-rigid map.
\end{remark}

\begin{remark}
    Let $\mathcal{G}$ be a tree and let $\pi: V(\mathcal{G}) \to \mathbb{R}$ be an injection. Then, for a $\pi(\mathcal{G})$-rigid map $\varphi: \pi(V(\mathcal{G})) \to \mathbb{R}$, the binary representation $\sigma$ uniquely (up to an additive constant) defines $\varphi(v) - v$, for every $v \in \pi(V(\mathcal{G}))$. This also implies that a connected graph $\mathcal{G}'$ on the real line has at most $2^{V(\mathcal{G}')-1}$ $\mathcal{G}$-rigid maps, up to isometry.
    \label{rm:1}
\end{remark}

Suppose we have a short random path on the real line. How large is the probability that we can move the vertices of the path but preserve both the distances along edges and the positions of the end vertices? A bound is given in the following Claim~\ref{cl:3.4}, which is a variant of Lemma~2.3 from Gir\~ao, Illingworth, Michel,  Powierski and Scott~\cite{GIMPS}.

\begin{claim}
    Let $V \subseteq \mathbb{R}$ with $|V| = n$. Let $u, v \in V$, and let $\varphi: \{u, v\} \to \mathbb{R}$. Let $s \leqslant \ln n / 50$, and let $\mathcal{P}$ be a path of length $s$ from $u$ to $v$, with internal vertices chosen uniformly at random from the vertices of $V$. 
    
    {\bf 1.} If $|u - v| \neq |\varphi(u) - \varphi(v)|$ then the probability that $\varphi$ can be extended to a $\mathcal{P}$-rigid map $\varphi:V(\mathcal{P}) \to \mathbb{R}$ is at most $1/\sqrt{n}$;

    {\bf 2.} If $|u - v| = |\varphi(u) - \varphi(v)|$ then the probability that $\varphi$ can be extended to a $\mathcal{P}$-rigid map $\varphi:V(\mathcal{P}) \to \mathbb{R}$ in a non-trivial way is at most $1/\sqrt{n}$.

    \label{cl:3.4}
\end{claim}

\begin{proof}

We prove the second statement of Claim~\ref{cl:3.4}. The proof of the first statement is similar up to omitting the word non-trivial in the proof.
    
Wlog, let $u-v = \varphi(u) - \varphi(v)$. Let $\mathcal{P} := u=v_0v_1\ldots v_s=v$ be a random path. For each $i \in [s]$, define the random variable $d_i := v_i - v_{i-1}$. We reveal and fix the vertices $v_i$ one by one.

Let $t\geq0$ be the first index such that there is a non-trivial solution (i.e. a solution with not all $\pm$ being pluses) of
\begin{equation}
    (v - v_t)\pm d_1 \pm d_2 \pm \cdots \pm d_t = \varphi(v) - \varphi(u).
    \label{eq:pm_prev}
\end{equation}

Note that for $t=0$ the equation does not have non-trivial solutions, so $t > 0$. If no such $t$ exists, then $\varphi$ cannot be extended to $V(\mathcal{P})$ in a non-trivial way, since this is equivalent to a non-trivial solution of $\pm d_1 \ldots \pm d_t = \varphi(v) - \varphi(u)$. Suppose such $t$ does exist and consider the smallest $t$ with non-trivial solution of~\eqref{eq:pm_prev}. Then, by minimality of $t$, the non-trivial solution of~\eqref{eq:pm_prev} has a minus before $d_t$. Then $v_t$ must satisfy
\begin{equation}
    v_t + d_t = v - (\varphi(v) - \varphi(u)) \pm d_1 \pm \cdots \pm d_{t-1}.
    \label{eq:pm}
\end{equation}
Since $v_{t-1}$ is already fixed, the left hand-side $v_t + d_t$ equals $2v_t - v_{t-1}$ and, hence, is linear as a function of $v_t$. Therefore, there are at most $2^{t-1}$ possible positions for $v_t$ on the real line that satisfy~\eqref{eq:pm}, and, hence, the probability that $v_t$ lands at some suitable point is at most 
$
\frac{2^{t-1}}{n/2 - t}.
$
Note that $t \leqslant \ln n / 50$, so taking a union bound over all such $t$ yields an overall upper bound of, say, $1/\sqrt{n}$.

\end{proof}

\section{Reconstructing a linearly-sized equivalence-dependent subset}
\label{sc:3}

In this section we shall show that for $\pi(\mathcal{C})$ the following event holds whp.

\begin{definition} {\bf Event $A$.}
    Let $\mathcal{C}$ be a graph with $V(\mathcal{C}) \subseteq \mathbb{R}$ and minimum degree at least $2$. Let $\mathcal{K}$ be its kernel. Let $c > 0$. Let $A:= A(\mathcal{C}, c)$ denote the event that, for every $\mathcal{C}$-rigid map $\varphi:V(\mathcal{C}) \to \mathbb{R}$, there exists a subset $U \subseteq V(\mathcal{K})$ with $|U| \geq c|V(\mathcal{K})|$ such that
    $$|u - v| = |\varphi(u) - \varphi(v)|\textnormal{ for every }u, v \in U.$$
    \label{df:A(G, c)}
\end{definition}

 The formal statement is as follows.

\begin{lemma}
There exists $c>0$ such that, for $\lambda$ and $\mathcal{G} \sim \mathcal{U}$ from Claim~\ref{cl:contiguous}, if $\lambda$ satisfies $\lambda \geq 1+\frac{1}{c\ln n}$ for every large enough $n$, then $A(\mathcal{G} , c)$ holds whp.

\label{lm:1}
\end{lemma}

Observe that a weaker version of this result --- for small enough constant $\lambda$ --- follows from Theorem~1.1.9 of the PhD Thesis by Julien Portier obtained in joint work with Julian Sahasrabudhe~\cite{PoSa}. Their result relies on a variant of Garamv{\"o}lgyi~\cite{Ga}, which, roughly speaking, describes a structural property that arises when some edge in a graph on the line cannot be reconstructed. The main similarity between their methods and ours can be summarised as follows. For a graph $G$ on the real line $\mathbb{R}$, a $G$-rigid mapping $\varphi$ describes a natural partition of $V(G)$ into equivalence classes with nice properties: the vertices $u\in V \subseteq \mathbb{R}$ and $v\in V\subseteq \mathbb{R}$ are equivalent iff $|\varphi(u) - \varphi(v)| = |u - v|$. However, our subsequent developments of this idea differ significantly from those of~\cite{PoSa}.

We will deduce Lemma~\ref{lm:1} from the following result, which gives the same conclusion for an appropriate fixed graph $\mathcal{C}$, 
after randomly permuting its vertices.

\begin{theorem}
    Let $V \subseteq \mathbb{R}$ with $|V| = n$. Let $\mathcal{C}$ be a connected graph with minimum degree $2$ of size at most $n$ whose kernel $\mathcal{K}$ satisfies the following properties: for $\varepsilon\in (0, 10^{-9})$, $c \in (0, 1)$, and $k := |V(\mathcal{K})|$,
    \begin{itemize}

         \item[R1] $\Delta(\mathcal{K}) \leqslant k^\varepsilon$;

        \item[R2] $|V(\mathcal{C})| \leqslant  \varepsilon k\ln k$ and $\ln k = (1- o(1))\ln n$;
        
        \item[R3] $|\{e\in E(\mathcal{K}): \textnormal{the corresponding path }\mathcal{P}_e\textnormal{ in }\mathcal{C}\textnormal{ satisfies }|V(\mathcal{P}_e)| > \ln n/100\}| \leqslant  \varepsilon k$;
        
        \item[R4] For all $U \subseteq V(\mathcal{K})$ with $|U| \leqslant \varepsilon k$, the number of edges incident to at at least one vertex of $U$ is at most $10^{-9}k$;

         \item[R5] For any $s>0$ and any disjoint sets $V_1, \ldots, V_s \subseteq V(\mathcal{K})$ with $|V_i| \leqslant ck$ for all $i$ we have
         $$\sum_{i=1}^s|E(\mathcal{K}|_{V_i})| \leqslant 1.078\left(\sum_{i=1}^s|V_i| - s\right) + o(k).$$

    \end{itemize}
    Let $\pi: V(\mathcal{C}) \to V$ be a uniformly random injection. Then, $A(\pi(\mathcal{C}), c)$ holds whp.

    \label{th:super1}
\end{theorem}

Let us give a few comments on Theorem~\ref{th:super1}. Notice that $V(\mathcal{C}) \leq n$ is compulsory for $\pi$ to be defined. Property R1 is only used to achieve the conclusions of Claim~\ref{cl:partitions_number} and Claim~\ref{cl:num_of_spanning_trees}. 
The properties R2 and R3 are important for the proof of Theorem~\ref{th:super1}. The property R5 is only used to prove Claim~\ref{cl:3_r5}, while R4 and Claim~\ref{cl:3_r5} are only needed to show Claim~\ref{cl:3_r4_r5}. So, instead of property R5, we can require the conclusion of Claim~\ref{cl:3_r5} to hold. Also, both requirements R4 and R5 can be substituted by the conclusion of Claim~\ref{cl:3_r4_r5}.

The proof of Theorem~\ref{th:super1} can be summarised as follows. Roughly speaking, Claim~\ref{cl:3.4} says that if there is a short random path $\mathcal{P}$ with fixed end vertices the probability that it has a non-trivial $\mathcal{P}$-rigid map is at most $n^{-\Omega(1)}$. Thus, our aim is to show that if the conclusion of Theorem~\ref{th:super1} does not hold, then we are typically able to fix some values of $\pi$ on $V(\mathcal{C})$ in such a way that, for every $\pi(\mathcal{C})$-rigid map $\varphi: \pi(V(\mathcal{C})) \to \mathbb{R}$, there are many disjoint paths satisfying the requirements of Claim~\ref{cl:3.4}. More formally we do the following. We define a set of events covering the negation of $A$. Every event roughly says that we fixed $\pi$ and a binary representation $\sigma$ of $\varphi$ on some subtree $\mathcal{T_C}$. More precisely, the above-mentioned events $A_{\mathfrak{P}, \mathcal{T_K}, \sigma}$ are defined for every tuple $(\mathfrak{P}, \mathcal{T_K}, \sigma)$, where $\mathfrak{P}$ is a partition  of $V(\mathcal{K})$, $\mathcal{T_K} \subseteq \mathcal{K}$ is a spanning tree and $\sigma$ is a binary function on some subtree $\mathcal{T_C}(\mathcal{T_K}) =: \mathcal{T_C}\subseteq\mathcal{C}$ satisfying some good properties. We want to prove Theorem~\ref{th:super1} using a union bound, so we bound the number of tuples $(\mathfrak{P}, \mathcal{T_K}, \sigma)$ from above by, say, $k^{5\varepsilon k}$. Then, we consider an event $A_{\mathfrak{P}, \mathcal{T_K}, \sigma}$, which restricts the set of allowed $\pi(\mathcal{C})$-rigid maps $\varphi$ (we call them {\it proper}). We show that after fixing $\pi|_{V(\mathcal{T_C})}$, typically, for every proper $\pi(\mathcal{C})$-rigid map $\varphi$ we have linearly many, say $30\varepsilon k$, pairs $uv \in E(\mathcal{K})$ such that $u, v \in V(\mathcal{T_C})$, $\pi$ is not fixed for the inner vertices of the 2-path of $uv$, and $|\varphi(\pi(u)) - \varphi(\pi(v))| \neq |\pi(u) - \pi(v)|$. So, excluding the paths that are longer than $\ln n/100$, the paths are edge-disjoint and satisfy Claim~\ref{cl:3.4}, as required.

The rest of Section~\ref{sc:3} is split into three subsections. In Subsection~\ref{sc:3.1} we show that the random graph $\mathcal{C}\sim\mathcal{L}$ whp satisfies the requirements of Lemma~\ref{lm:1}.  
The remaining two subsections are dedicated to the proof of Theorem~\ref{th:super1}. In Subsection~\ref{sc:3.2} we define the event $A_{\mathfrak{P}, \mathcal{T_K}, \sigma}$, we show that the negation of $A$ implies $A_{\mathfrak{P}, \mathcal{T_K}, \sigma}$ for some tuple $(\mathfrak{P}, \mathcal{T_K}, \sigma)$, and we bound the number of tuples from above. So, it will remain to show that $\mathbb{P}(A_{\mathfrak{P}, \mathcal{T_K}, \sigma}) = o(k^{-5\varepsilon k})$ --- we do that in Subsection~\ref{sc:3.3}.

\subsection{$\mathcal{C}$ satisfies R1-R5}
\label{sc:3.1}

In this subsection we show that, provided $c>0$ is small enough, the random graph $\mathcal{C}\sim\mathcal{L}$, Claim~\ref{cl:L-contiguous}, satisfies the requirements of Theorem~\ref{th:super1} whp as long as $\lambda$ from Claim~\ref{cl:L-contiguous} satisfies $\lambda \geq 1+\frac{1}{c\ln n}$ for every large $n$. Note that $\mathcal{C}$ is indeed connected, has minimum degree $2$ and has size at most $n$ by Claim~\ref{cl:core_size}. R1 follows from S3, Claim~\ref{cl:S}. Also, taking $c>0$ small enough in the equation $\lambda \geq 1+\frac{1}{c\ln n}$, R2 follows from S1, Claim~\ref{cl:S}, and Claim~\ref{cl:core_size}.  Next, R3 holds whp since, by S2 from Claim~\ref{cl:S}, $|E(\mathcal{K})|$ is linear in $k$ and, for every edge $e \in E(\mathcal{K})$, independently, $\mathbb{P}(|\mathcal{P}_e|> \ln n/100) = o(1)$. 
Also, R4 is a corollary of S4, Claim~\ref{cl:S}. So, we only need to prove that R5 holds whp for $\mathcal{C}$. 

The following claim, together with Claim~\ref{cl:kernel1} and Claim~\ref{cl:kernel2}, shows that R5 holds whp for $\mathcal{C}$.

\begin{claim}
    Let a graph sequence $\mathcal{K} := \{\mathcal{K}_n\}$ with $k:=|V(\mathcal{K})|\to \infty$ satisfy the following properties:
    \begin{itemize}
        \item for some $c > 0$ and every $U \subseteq V(\mathcal{K})$ with $|U| \leqslant ck$, it holds that $|E(\mathcal{K}|_{U})| \leqslant 1.078|U|$;
        \item for all but $o(k)$ vertices $v \in V(\mathcal{K})$, the ball of radius $\lfloor \ln \ln k \rfloor$ centred at $v$ is a tree.
    \end{itemize}
    Then, R5 from Theorem~\ref{th:super1} holds for $\mathcal{K}$. In other words, for any $s>0$ and any disjoint sets $V_1, \ldots, V_s \subseteq V(\mathcal{K})$ with $|V_i| \leqslant ck$, $i \in [s]$, we have
         \begin{equation}
             \sum_{i=1}^s|E(\mathcal{K}|_{V_i})| \leqslant 1.078\left(\sum_{i=1}^s|V_i| - s\right) + o(k).
             \label{eq:3_rprove}
         \end{equation}
    \label{cl:3_rprove}
\end{claim}

\begin{proof}
     Consider arbitrary disjoint sets $V_1, \ldots, V_s \subseteq V(\mathcal{K})$ with $|V_i| \leqslant ck$, $i \in [s],$ and let us prove~\eqref{eq:3_rprove}. Wlog we can assume that, for every $i \in [s]$, $V_i$ induces a connected subgraph.
     
     Let $I \subseteq [s]$ be the set of indices $i \in [s]$ such that $V_i$ induces a tree in $\mathcal{K}$. Let $x := \sum_{i\in I} |V_i|,$ and let $y:= \sum_{i \in [s]\setminus I }|V_i|$. For $i \in I$, $|E(\mathcal{K}|_{V_i})|  = |V_i| -1$. Also, due to the first assumption from Claim~\ref{cl:3_rprove}, $|E(\mathcal{K}|_{V_i})|  \leqslant 1.078|V_i|$, for $i \notin I$. Applying these bounds we get $$\sum_{i=1}^s|E(\mathcal{K}|_{V_i})| \leqslant 1.078y + x - |I| = 1.078\sum_{i=1}^s|V_i|- 0.078x - |I| \stackrel{x \geqslant |I|}\leqslant 1.078\sum_{i=1}^s|V_i|- 1.078|I|.$$
    So, it remains for us to show that $|I| \geqslant s - o(k)$.

    All but $o(k)$ of the sets $V_1, \ldots, V_s$ have size less than $\lfloor \ln \ln k \rfloor$. Hence, by the second assumption from Claim~\ref{cl:3_rprove}, all but $o(k)$ of the sets $V_1, \ldots, V_s$ are trees, implying $|I| \geqslant s - o(k)$.
\end{proof}

\subsection{Reduction of Theorem~\ref{th:super1} to $\mathbb{P}(A_{\mathfrak{P}, \mathcal{T_K}, \sigma}) = o(k^{-5\varepsilon k})$}
\label{sc:3.2}

    In this section we define a collection $\mathcal{A}$ of events that covers the complement of the event $A$. We also bound from above the size of $\mathcal{A}$. First, we give several remarks on notation that simplify the explanation in both this and next subsections. Next, we describe the collection $\mathcal{A}$ in Lemma~\ref{lm:A_cover}. Then, we establish the bound on $|\mathcal{A}|$ in Claim~\ref{cl:A_bound}.

    Let $\mathfrak{P}$ be a partition of $V(\mathcal{K})$. For $u, v \in V(\mathcal{K})$ let us write $u \sim v$ when $u$ and $v$ belong to the same set of the partition and $u \nsim v$ otherwise.

    Recall that a  2-path is a path with every inner vertex of degree 2. Recall that every edge in the kernel $\mathcal{K}$ corresponds to a 2-path in the graph $\mathcal{C}$. Let us now define the events that make up $\mathcal{A}$.

    \begin{definition}[Event $A_{\mathfrak{P}, \mathcal{T_K}, \sigma}$]
        
        Recall that $\pi: V(\mathcal{C}) \to V$ is a uniformly random injection. For a tuple $(\mathfrak{P},\mathcal{T_K}, \sigma)$ that consists of
         \begin{itemize}
             \item a partition  $\mathfrak{P} = \{V_1, \ldots, V_s\}$ of $V(\mathcal{K})$,
             \item a spanning tree $\mathcal{T_K} \subseteq \mathcal{K}$, and
             \item a binary function $\sigma \in \{-1, 1\}^{E(\mathcal{T_C})}$, where $\mathcal{T_C}\subseteq\mathcal{C}$ is the tree we get from $\mathcal{T_K}$ after substituting every edge $uv \in E(\mathcal{T_K})$ by the corresponding 2-path $\mathcal{P}_{uv} \subseteq \mathcal{C}$,
         \end{itemize}
         let $A_{\mathfrak{P}, \mathcal{T_K}, \sigma}$ denote the event that there exists a $\pi(\mathcal{C})$-rigid map $\varphi: \pi(V(\mathcal{C})) \to \mathbb{R}$ with binary representation $\sigma$ on $\mathcal{T_C}$ such that the following holds:
         \begin{itemize}
             \item[P1] for every $i\in [s]$, $|V_i| < ck$ and $V_i$ induces a connected subgraph in $\mathcal{K}$;
             \item[P2] for every $u, v \in V(\mathcal{K})$, we have $u \sim v \implies \varphi(\pi(u)) - \varphi(\pi(v)) = \pi(u) - \pi(v)$;
            \item[P3] for every $uv \in E(\mathcal{K})$ we have $u\sim v \iff \varphi(\pi(u)) - \varphi(\pi(v)) = \pi(u) - \pi(v)$;
            \item[P4] $|\{uv \in E(\mathcal{K}): \varphi(\pi( u)) - \varphi(\pi (v)) = \pi (u) - \pi (v) \}| $

            \hspace*{\fill} $\geqslant | u v\in E(\mathcal{K}): \varphi(\pi (u)) - \varphi(\pi (v)) = -(\pi (u) - \pi( v)) \}|$;
            \item[P5] for every $i \in [s]$, $\mathcal{T_K}|_{V_i}$ is connected;
            \item[P6] if we remove from $\mathcal{T_K}$ every edge $uv$ such that $
            |\varphi(\pi(u)) - \varphi(\pi(v))| \neq |\pi(u) - \pi(v)|,
            $ then $\mathcal{T_K}$ splits into connected components such that, for every $w, r$ from different components, $wr\in E(\mathcal{K}) \implies |\varphi(\pi(w)) - \varphi(\pi(r))| \neq |\pi(w) - \pi(r)|.$
         \end{itemize}
         \label{df:3.5}
    \end{definition}

    The somewhat complicated condition P6 is illustrated in Figure~\ref{fig:8}.

    \begin{figure}[!h]
    \centering
    \includegraphics[scale=0.2]{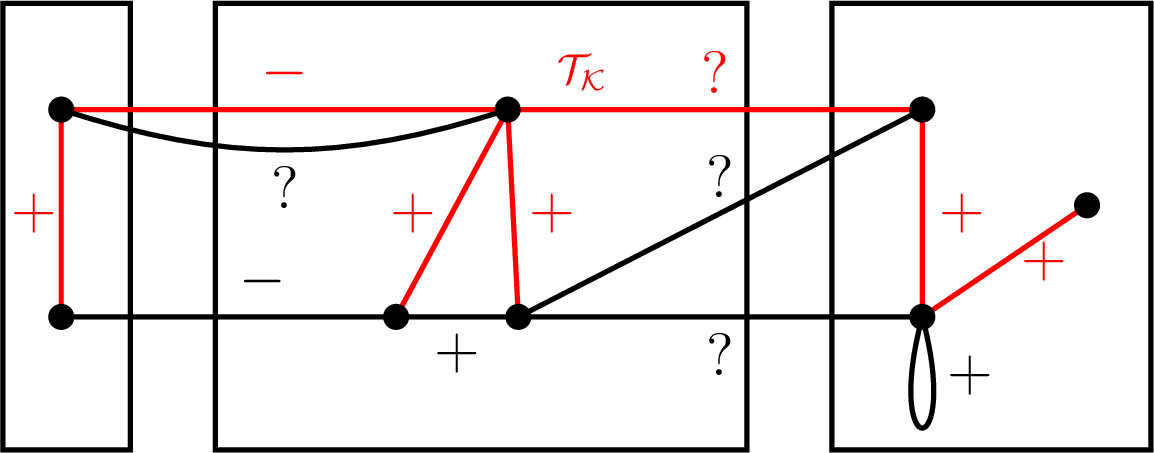}
    \caption{The figure illustrates the structure of $\mathcal{T_K}$ in $\mathcal{K}$ given by P6. We assign to each edge $E(\mathcal{K})$ a sign from $\{+, -, ?\}$ as follows. For an edge $uv \in E(\mathcal{K})$, we assign either ``$+$'' or ``$-$'' if $\frac{\varphi(\pi(u)) - \varphi(\pi(v))}{\pi(u)-\pi(v)}$ equals to $1$ or $-1$ respectively, otherwise we assign ``$?$''.
    In the figure, the boxes represent the equivalence classes of $\mathfrak{P}$ and the red edges make up the tree $\mathcal{T_K}$. Notice that, in the picture, the tree $\mathcal{T_K}$ should contain 1 edge connecting the middle box to the left-hand one, and this edge should be chosen among the edges with ``$-$''. Also, $\mathcal{T_K}$ should contain 1 edge connecting the middle box to the right-hand box and this edge can be chosen arbitrarily, since all edges between the middle and the right-hand boxes carry the sign ``$?$''.
}
    \label{fig:8}
\end{figure}

    \begin{lemma} 
         Suppose that the event $A$ defined in  does not hold. Then, there exists a tuple $(\mathfrak{P}, \mathcal{T_K}, \sigma)$ such that $A_{\mathfrak{P}, \mathcal{T_K}, \sigma}$ holds.
         \label{lm:A_cover}
    \end{lemma}

    \begin{proof}

    Throughout the proof we fix $\pi$ such that $A$ does not hold. Let $\varphi$ be a $\pi(\mathcal{C})$-rigid map disproving $A$. Since both $\varphi$ and $-\varphi$ satisfy $A$, wlog we can assume that P4 holds. Let us construct a tuple $(\mathfrak{P}, \mathcal{T_K}, \sigma)$ such that $\varphi$ verifies $A_{\mathfrak{P}, \mathcal{T_K}, \sigma}$.

    \begin{claim}
        There exists $\mathfrak{P}$ that verifies P1, P2, and P3.\label{cl:find_p}
    \end{claim}

    \begin{proof}

    For every $v \in V(\mathcal{K})$ let us call $u$ {\emph{equivalent}} to $v$ if the following holds: there is a path $\mathcal{P}\subseteq\mathcal{K}$ such that 
    \begin{itemize}
        \item $\mathcal{P}$ connects $u$ and $v$;
        \item $\varphi$ behaves as a translation on $\pi(V(\mathcal{P}))$ (i.e. for some $r \in \mathbb{R}$ and every $w \in \pi(V(\mathcal{P}))$, $\varphi(w) = w+r$).
    \end{itemize}

    Note that equivalence is indeed an equivalence relation. So, it creates a natural partition of the vertex set into equivalence classes $V_1 \sqcup \ldots \sqcup V_s = V(\mathcal{K})$. Note that every equivalence class $V_i$ is connected. Also, for every $u, v \in V_i$, it holds that $\varphi(\pi(u)) - \varphi(\pi(v)) = \pi(u) - \pi(v)$, due to the definition. So, $\max_{i\in[s]} |V_i| < ck$ as we assumed that $\varphi$ disproves $A$. Hence, P1 holds. Finally, both P2 and P3 hold by the definition of sets $V_1, \ldots, V_s$.

    \end{proof}

    Fix an arbitrary partition $\mathfrak{P}$ satisfying P1, P2, P3. Let $V_1 \sqcup \ldots \sqcup V_s = V(\mathcal{K})$ be the parts of ${\mathfrak{P}}$.

    \begin{claim}
        There exists a spanning tree $\mathcal{T_K}$ of $\mathcal{K}$ satisfying P5 and P6.
        \label{cl:F_T}
    \end{claim}
    
    \begin{proof}
        We construct a forest $\mathcal{F} \subseteq \mathcal{K}$ such that an arbitrary spanning tree $\mathcal{T_K} \subseteq \mathcal{K}$ containing $\mathcal{F}$ satisfies P5 and P6 with $\varphi$. Let $\mathcal{F}$ be a maximal forest among the forests such that
    \begin{itemize}
        \item[Q1] $\mathcal{F}|_{V_i}$ is a tree, for every $i$;
        \item[Q2] $\forall uv \in E(\mathcal{F}), u\nsim v \implies \varphi(\pi(u)) - \varphi(\pi(v)) = -(\pi(u) - \pi(v))$,
    \end{itemize}
    see Figure~\ref{fig:1}.
    
\begin{figure}[h]
    \centering

    \includegraphics[scale=0.1]{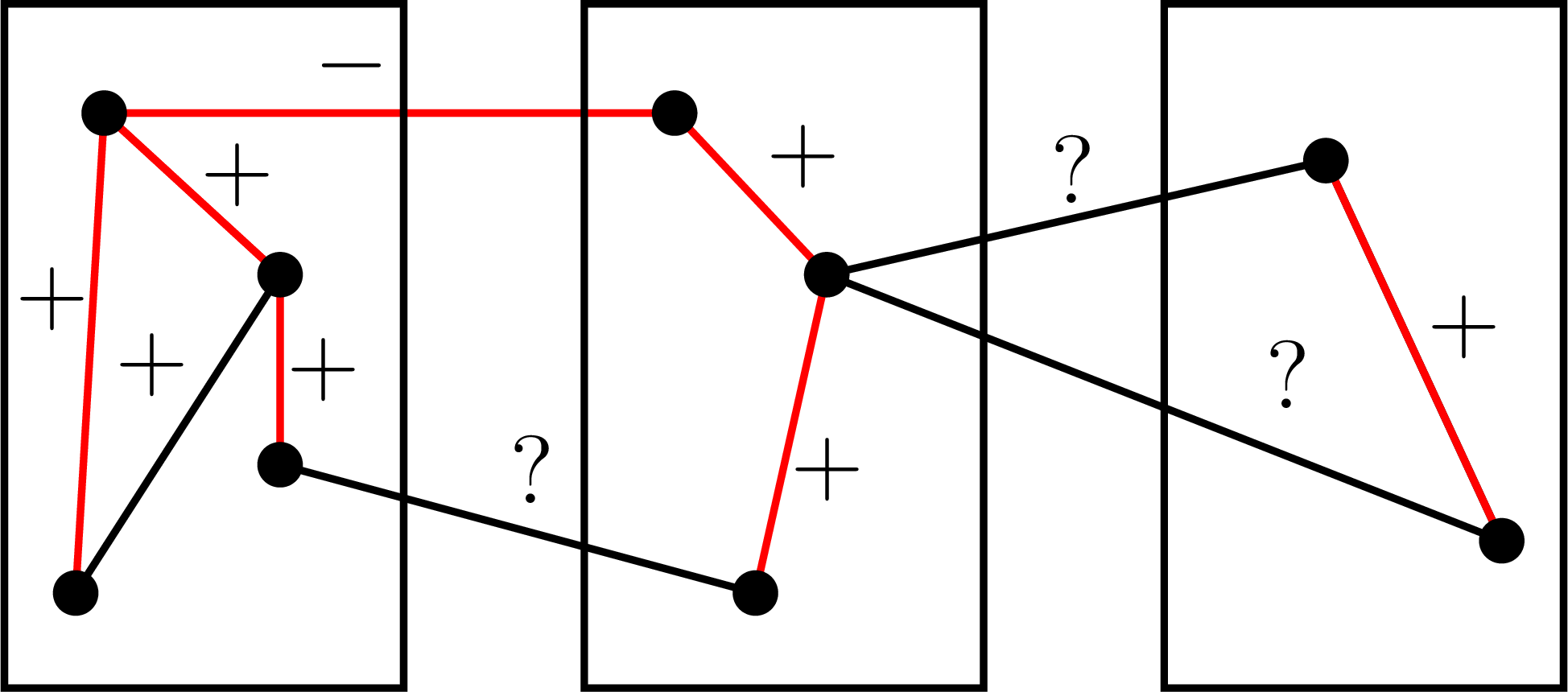}
    \caption{The figure illustrates the  structure of $\mathcal{F}$ in $\mathcal{K}$. For an edge $uv \in E(\mathcal{K})$, we assign either ``$+$'' or ``$-$'' if $\frac{\varphi(\pi(u)) - \varphi(\pi(v))}{\pi(u)-\pi(v)}$ equals to $1$ or $-1$ respectively, otherwise we assign ``$?$''.
    The boxes represent the equivalence classes of $\mathfrak{P}$ and the red edges make up the forest $\mathcal{F}$. 
}
    \label{fig:1}
\end{figure}
    Then, $\mathcal{F}$ suits the role described above. Indeed, let $\mathcal{T_K}$ be an arbitrary spanning tree containing $\mathcal{F}$. Then, P5 holds trivially from Q1. In order to verify P6, consider an edge $uv \in E(\mathcal{T_K})$ satisfying $|\varphi(\pi(u)) - \varphi(\pi(v))| \neq |\pi(u) - \pi(v)|$. By P3, $u \nsim v$. So, due to Q2, $uv \notin E(\mathcal{F})$. In addition, since $\mathcal{F}$ is maximal satisfying Q2, every edge $uv \in E(\mathcal{T_K})\setminus E(\mathcal{F})$ satisfies $|\varphi(\pi(u)) - \varphi(\pi(v))| \neq |\pi(u) - \pi(v)|$. So, the connected components after the split described in P6 are precisely the connected components of $\mathcal{F}$.

    Let us now show that for $w$ and $r$ from different connected components of $\mathcal{F}$, $$wr\in E(\mathcal{K}) \implies |\varphi(\pi(w)) - \varphi(\pi(r))| \neq |\pi(w) - \pi(r)|.$$
    Indeed, Q1 gives us $w \nsim r$. So, by P3, $wr\in E(\mathcal{K})$ implies $\varphi(\pi(w)) - \varphi(\pi(r)) \neq \pi(w) - \pi(r)$. Also, if $wr\in E(\mathcal{K})$ then $\varphi(\pi(w)) - \varphi(\pi(r)) \neq -(\pi(w) - \pi(r))$, since otherwise we can add $wr$ to $\mathcal{F}$ without breaking Q1 and Q2 --- a contradiction with maximality of $\mathcal{F}$.
    \end{proof}

    Now, we fix an arbitrary tree verifying Claim~\ref{cl:F_T} and a binary representation $\sigma$ of $\varphi|_{\mathcal{T_C}}$. Doing this, we conclude Lemma~\ref{lm:A_cover}.

    \end{proof}

    Recall that $k = |V(\mathcal{K})|$. Let us now bound the number of tuples $(\mathfrak{P},\mathcal{T_K}, \sigma)$ that have positive probability to satisfy the event $A_{\mathfrak{P},\mathcal{T_K}, \sigma}$, Definition~\ref{df:3.5}.
    
    \begin{claim}
        The number of tuples $(\mathfrak{P},\mathcal{T_K}, \sigma)$ from Definition~\ref{df:3.5} with $\mathfrak{P}$ satisfying P1 is at most $k^{5\varepsilon k}$. \label{cl:A_bound}
    \end{claim}

    \begin{proof}
        Recall that the maximum degree of the graph $\mathcal{K}$ is bounded by R1. So, for large enough $n > 0$, the number of partitions $\mathfrak{P}$ satisfying P1 does not exceed $k^{2\varepsilon k}$ by Claim~\ref{cl:partitions_number}.
        Also, for large enough $n$, the number of spanning trees of $\mathcal{K}$ is at most $k^{2\varepsilon k}$ by Claim~\ref{cl:num_of_spanning_trees}. Finally, the number of binary functions $\sigma$ is at most
        \begin{equation}2^{|E(\mathcal{T_{C}})|} \leqslant 2^{|V(\mathcal{T_C})|} \leqslant 2^{|V(\mathcal{C})|} \stackrel{R2}\leqslant \exp(\varepsilon k\ln k),
        \end{equation}
        which concludes the proof.
\end{proof}

\subsection{Proof of $\mathbb{P}(A_{\mathfrak{P}, \mathcal{T_K}, \sigma}) = o(k^{-5\varepsilon k})$}
\label{sc:3.3}

    Combining Lemma~\ref{lm:A_cover} and Claim~\ref{cl:A_bound}, the following lemma is sufficient to conclude Theorem~\ref{th:super1}.

    \begin{lemma}
        For a tuple $(\mathfrak{P},\mathcal{T_K}, \sigma)$ from Definition~\ref{df:3.5} we have
        \begin{equation}
        \mathbb{P}(A_{\mathfrak{P}, \mathcal{T_K}, \sigma}) = o(|V(\mathcal{K})|^{-5\varepsilon |V(\mathcal{K}|)}).
        \label{eq:3.5}
    \end{equation}
        \label{lm:A_formal}
    \end{lemma}

    In this subsection we prove Lemma~\ref{lm:A_formal}.

    Fix an arbitrary tuple $(\mathfrak{P},\mathcal{T_K}, \sigma)$ from Definition~\ref{df:3.5} satisfying P1 and P5.

    \begin{remark}
        Notice that we can assume that P1 and P5 hold for $\mathfrak{P}$ and $\mathcal{T_K}$, since otherwise $\mathbb{P}(A_{\mathfrak{P}, \mathcal{T_K}, \sigma}) = 0$. In particular, given $\pi|_{\mathcal{T_C}}$, $\sigma$ uniquely determines the $\pi(\mathcal{T_C})$-rigid map $\varphi:\pi(V(\mathcal{T_C})) \to \mathbb{R}$ up to translation.
    \end{remark}
    
    Let us now explain the way we want to show~\eqref{eq:3.5}. Recall that $\pi$ is a random injection. Below, we give definitions and claims with respect to the fixed $\mathfrak{P}$, $\mathcal{T_K}$, $\mathcal{T_C}$, $\sigma$ and the random $\pi$.

    In order to conclude the desired bound~\eqref{eq:3.5}, we want to fix $\pi|_{V(\mathcal{T_C})}$ and find $29\varepsilon k$ edge-disjoint 2-paths such that the condition of Claim~\ref{cl:3.4} holds for each of them. We find the paths in two steps. First, we show that there are many candidates for the paths. Second, we prove that there are indeed $30\varepsilon k$ edge-disjoint 2-paths satisfying every condition of Claim~\ref{cl:3.4} except, perhaps, for the requirement that the path is not too long. So, R3 instantly gives the desired set of $29\varepsilon k$ paths. Below, we formalise the two steps as two individual claims Claim~\ref{cl:3_r5} and Claim~\ref{cl:3_r4_r5} and prove the claims. Then, we finally derive~\eqref{eq:3.5} from Claim~\ref{cl:3.4} and Claim~\ref{cl:3_r4_r5}. However, first, we need to give some definitions.

    \begin{definition}
        For a given injection $\pi: V(\mathcal{C}) \to \mathbb{R}$, let us call a $\pi(\mathcal{T_C})$-rigid map $\varphi$ {\it proper} if it has binary representation $\sigma$ and satisfies P2, P3, P4, and P6. 
        
        \label{df:proper}
    \end{definition}

    Recall that for every $\pi(\mathcal{C})$-rigid map its restriction to $\pi(\mathcal{T_C})$ is also a $\pi(\mathcal{T_C})$-rigid map. Recall that our aim is to prove~\eqref{eq:3.5}, i.e. give an upper bound of $k^{-5\varepsilon k}$ on the probability that there exists a $\pi(\mathcal{C})$-rigid map whose restriction to $\pi(\mathcal{T_C})$ is proper.

    \begin{definition}
        For $uv \in E(\mathcal{K})$ let $D= D(uv)$ be the event that

        \begin{itemize}
            
            \item $u\nsim v$ and $uv \not\in E(\mathcal{T_K})$;
            \item $|\varphi(\pi(u)) - \varphi(\pi(v))| \neq |\pi(u) - \pi(v)|$ holds for {\it every  proper} $\pi(\mathcal{T_C})$-rigid map $\varphi$.
        \end{itemize}
        \label{df:3.}
    \end{definition}

    \begin{remark}
        Note that if $D(uv)$ holds for {\it some} proper $\pi(\mathcal{T_C})$-rigid map $\varphi$, then it holds for every proper $\pi(\mathcal{T_C})$-rigid map $\varphi$, since given $\pi$ all proper $\pi(\mathcal{T_C})$-rigid maps are equivalent on $V(\mathcal{T_C}) \supseteq V(\mathcal{K})$ up to translations.
        \label{rm:some_every}
    \end{remark}

    \begin{remark}
        Everywhere below (including Claim~\ref{cl:3_r5} and Claim~\ref{cl:3_r4_r5} stated below) the probability space is described as follows: we have fixed the values of $\pi|_{U}$ for some $U \subseteq V(\mathcal{C})$, and the set $U$ grows as the proof proceeds. Hence, the probability space is described by a uniformly random injection $\pi|_{V(\mathcal{C})\setminus U}: V(\mathcal{C})\setminus U \to V \setminus\pi(U).$
    \end{remark}

    \begin{claim}
    Let
    $$E_1 := \{uv \in  E(\mathcal{T_K}) : u\nsim v\},$$
    and let
    $$E_2 := \{uv \in E(\mathcal{K})\setminus E(\mathcal{T_K}) : u \nsim v\textnormal{ and }u\textnormal{ is not incident to any edge from }E_1\}.$$
    
    Then, either $|E_2| \geqslant 10^{-7} k$ or, for every choice of $\pi$, the number of edges from $E(\mathcal{K})$ satisfying~$D$ is at least $30 \varepsilon k$.
    \label{cl:3_r5}
\end{claim}

\begin{figure}[!h]
    \centering
    \includegraphics[scale=0.2]{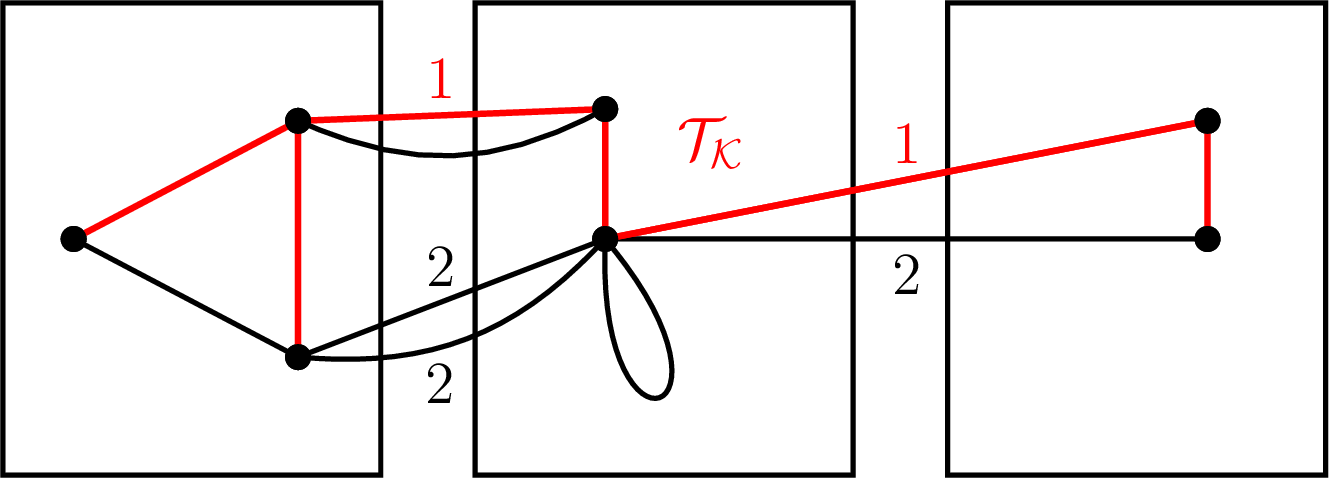}
    \caption{The figure illustrates the structure of $\mathcal{T_K}, E_1,$ and $E_2$ in $\mathcal{K}$ in the proof of Claim~\ref{cl:3_r5}.  In the figure, the boxes represent the equivalence classes of $\mathfrak{P}$ and the red edges make up the tree $\mathcal{T_K}$. The edges lying in $E_1$ and $E_2$ are labelled $1$ and $2$ respectively. 
} 
    \label{fig:6}
\end{figure}
    
    \begin{proof}

    Let us give a brief plan of the proof. We suppose that, for some choice of $\pi$, the number of edges from $E(\mathcal{K})$ satisfying~$D$ is less than $30 \varepsilon k$. We first show that the number of edges $uv \in E(\mathcal{T_K})$ that satisfy $|\varphi(\pi(u)) - \varphi(\pi(v))| \neq |\pi(u) - \pi(v)|$ is small. This allows us to conclude, by P4, that many edges lie inside the equivalence classes of $\mathfrak{P}$ (see Figure~\ref{fig:6}). Then, we conclude from R5 that $|\mathfrak{P}|$ is small and $|E_1| = |\mathfrak{P}| - 1$ is small too. Lastly, this allows us to deduce that $E_2$ is big enough.
    
    Suppose for a contradiction that Claim~\ref{cl:3_r5} does not hold. Recall that $\varepsilon \leqslant 10^{-9}$ and hence $30\varepsilon k \leqslant 10^{-7}k$. 
    Choose and fix $\pi:V(T_\mathcal{C})\to \mathbb{R}$ and $\varphi:\pi(V(\mathcal{T_C}))\to \mathbb{R}$ such that
    \begin{itemize}
        \item $\varphi$ is a proper $\pi(\mathcal{T_C})$-rigid map with respect to $\pi$;
        \item for $\varphi$ and $\pi$, $D$ holds for at most $30\varepsilon k < 10^{-7} k$ edges.
    \end{itemize}
    Our aim is to show that $|E_2| \geqslant 10^{-7}k$.

    Let us define
    $$E_3 :=\{uv \in E(\mathcal{T_K}): |\varphi(\pi(u)) - \varphi(\pi(v))| \neq |\pi(u) - \pi(v)|\}.$$
    First, we claim that, for large $k>0$, 
    \begin{equation}
        |E_3| \leqslant 4 \cdot 10^{-7}\cdot k.
        \label{eq:3.3}
    \end{equation}
    Suppose for a contradiction that $|E_3| > 4 \cdot 10^{-7}\cdot k$. Recall that $\varphi$ is proper and consider the partition $\mathfrak{Q}$ described in P6. Namely, $\mathfrak{Q}$ is the natural partition of $\mathcal{T_K}$ into connected components after removing all edges $uv \in E(\mathcal{T_K})$ satisfying $|\varphi(\pi(u))- \varphi(\pi(v))| \neq |\pi(v)-\pi(u)|$. Notice that, by P6, every edge $uv \in E(\mathcal{K})$ with $u$ and $v$ from different classes of $\mathfrak{Q}$ satisfies $|\varphi(\pi(u)) - \varphi(\pi(v))| \neq |\pi(u) - \pi(v)|$. Note that every edge that we delete from $\mathcal{T_K}$ when constructing the partition $\mathfrak{Q}$ increases the number of classes in $\mathfrak{Q}$ by 1. So,  $|\mathfrak{Q}| = |E_3|+1$. Note that $\mathfrak{Q}$ contains at most $c^{-1} = O(1)$ classes of size at least $ck$. Let $\mathfrak{Q}' \subseteq\mathfrak{Q}$ be the subset of classes that contain at most $ck$ vertices, hence $|\mathfrak{Q}'| \geqslant |E_3| - O(1)$. Due to R5 applied to $\mathfrak{Q'}$,
    \begin{equation}
        \sum_{U \in \mathfrak{Q}'}|E(\mathcal{K}|_{U})| \leqslant 1.078\left(\sum_{U \in \mathfrak{Q}'}|U| - |\mathfrak{Q}'|\right) + o(k).
        \label{eq:3.33}
    \end{equation}
    Let us define
    $$E_4 :=\{uv \in E(\mathcal{K}): \exists U \in \mathfrak{Q}'\text{ that }u \in U,\ v \notin U \}.$$
    So, since every vertex has degree at least $3$,~\eqref{eq:3.33} implies
    \begin{align*}
        |E_4| &\geqslant \frac{3}{2}\sum_{U \in \mathfrak{Q}'}|U| - \sum_{U \in \mathfrak{Q}'}|E(\mathcal{K}|_{U})| \\
        &\geqslant \frac{(3 - 2 \cdot 1.078)}2\left(\sum_{U \in \mathfrak{Q}'}|U|\right) + 1.078|\mathfrak{Q}'| + o(k) \\
        &\geqslant \frac{(3 - 2 \cdot 1.078)}2|\mathfrak{Q}'|+ 1.078|\mathfrak{Q}'| + o(k)\\
        &\geqslant \frac{(3 - 2 \cdot 1.078)}2|E_3| + 1.078|E_3| + o(k)\\
        &\geqslant \frac{(3 - 2 \cdot 1.078)}2 \cdot 4 \cdot 10^{-7} \cdot k + 1.078|E_3| + o(k).
    \end{align*}
    Let us show that $D$ holds for every $uv \in E_4\setminus E_3$.
    Recall that $\mathfrak{Q}' \subseteq \mathfrak{Q}$ and, so, every $uv \in E_4$ satisfies $|\varphi(\pi(u)) - \varphi(\pi(v))| \neq |\pi(u) - \pi(v)|$ by P6. By Remark~\ref{rm:some_every}, it remains to show that $u\nsim v$ and $uv \notin E(\mathcal{T_K})$. The former is a corollary of $P3$ and the latter holds since $uv \notin E_3$.
    
    Since $D$ holds for every $uv \in E_4\setminus E_3$, we get that the number of edges satisfying $D$ is at least
    $|E_4| - |E_3| \geqslant \frac{(3 - 2 \cdot 1.078)}2 \cdot 4 \cdot 10^{-7}k + o(k) > 10^{-7}\cdot k$  --- a contradiction with assumption that the number does not exceed $10^{-7}\cdot k$. This concludes the proof of~\eqref{eq:3.3}.

    Second, recall that, by the definition of $\varphi$ and $\pi$, $D$ holds for at most $10^{-7} k$ edges. Together with P3 and~\eqref{eq:3.3}, this gives us that
    \begin{equation}
        |\{uv \in E(\mathcal{K}): |\varphi(\pi(u)) - \varphi(\pi(v))| \neq |\pi(u) - \pi(v)|\}| \leqslant 5 \cdot 10^{-7}k.
        \label{eq:3.4}
    \end{equation}
    Recall that $V_1 \sqcup \ldots \sqcup V_s = V(\mathcal{K})$ is the partition $\mathfrak{P}$ and that $|E(\mathcal{K})| \geqslant 1.5 k$. So, by P4 and~\eqref{eq:3.4},
    \begin{equation}
        \sum_{i \in [s]}|E(\mathcal{K}|_{V_i})| \stackrel{P3}=|\{uv \in E(\mathcal{K}): \varphi(\pi(u)) - \varphi(\pi(v)) = \pi(u) - \pi(v)\}| \geqslant 0.5(1.5- 5\cdot 10^{-7})k.
        \label{eq:3.6}
    \end{equation}
    Combining R5 with \eqref{eq:3.6}, we get
    $$1.078(k-s) +o(k)\geqslant (0.75- 2.5\cdot 10^{-7})k,$$
    and hence
    \begin{equation}
        s < 0.305k,
        \label{eq:3.7}
    \end{equation}
    for large enough $k$. From the definition of $E_1$ from Claim~\ref{cl:3_r5} and the fact that $\mathcal{T_K}$ is a tree, we get $|E_1| \leq |\mathfrak{P}| -1 = s - 1$. Hence, the number of vertices covered by $E_1$ is at most $2s$. Let $U$ denote the set of vertices $u \in V(\mathcal{K})$ such that $u$ is not incident to any edge of $E_1$. So, 
    \begin{equation}
        |U| \geqslant k -2s.
        \label{eq:3.9}
    \end{equation}

    \begin{figure}[!h]
    \centering
    \includegraphics[scale=0.2]{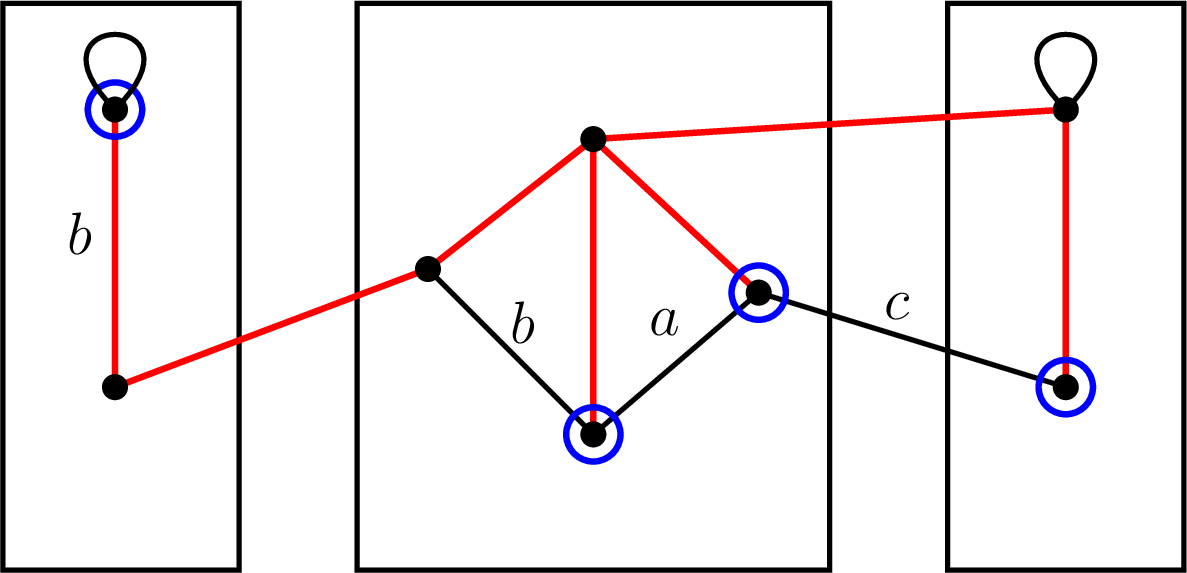}
    \caption{The figure illustrates the  four types of edges, $U$, and $\mathcal{T_K}$ in $\mathcal{K}$. In the figure, the boxes represent the equivalence classes of $\mathfrak{P}$ and the red edges make up the tree $\mathcal{T_K}$. The vertices of $U$ are encircled in blue. The edges of the first, second, and third type have labels $a$, $b$, and $c$ respectively in the figure. The edges of the forth type do not exist.
}
    \label{fig:7}
\end{figure}

    The edges incident to at least one vertex of $U$ are our candidates to be in $E_2$. There are four types of edges $uw \in E(\mathcal{K})$ that are incident to a vertex $u \in U$ (see Figure~\ref{fig:7}). The first type: $w\sim u$ and $w\in U$. The second type: $w\sim u$ and $w \notin U$. The third type: $u \nsim w$ and $uw \notin E(\mathcal{T_K})$. The edges of the third type make up the set $E_2$. The fourth type: $u \nsim w$ and $uw \in E(\mathcal{T_K})$. There are no edges of the fourth type, since $u \in U$ and $uw \in E_1$ --- that contradicts the definition of $U$. Let $T_1$ and $T_2$ be the number of edges of the first and the second type respectively.

    Note that the sum of the degrees of the vertices $u$ from $U$ is at least $3|U|$. Also, note that every edge of the second type counts once towards the sum, and every other edge connected to some $u \in U$ counts at most twice towards the sum. So, we get

    \begin{equation}
      3|U| \leqslant 2T_1 + T_2 + 2|E_2|.
      \label{eq:3.99}
    \end{equation}

    Next, $T_1 + T_2$ does not exceed the number of edges $uw$ with $u\sim w$. So,

    \begin{equation}
      T_1 + T_2 \leqslant|E(\mathcal{K}|_{V_1}) \sqcup \ldots \sqcup E(\mathcal{K}|_{V_s})| \stackrel{R5}\leqslant 1.078(k -s) +o(k).
      \label{eq:3.999}
    \end{equation}

    Finally, every edge that counts towards $T_1$ is contained in $E(\mathcal{K}|_{U \cap V_i})$, for some $i \in [s]$. So,
    \begin{equation}
      T_1 \stackrel{R5}\leqslant 1.078|U| + o(k).
      \label{eq:3.9999}
    \end{equation}

    Altogether, equations~\eqref{eq:3.9},~\eqref{eq:3.99},~\eqref{eq:3.999} and~\eqref{eq:3.9999} give us $$2|E_2| \geqslant  3(k -2s) - 1.078(k - s + o(k)) - (1.078(k-2s) + o(k)) = 0.844 k- 2.766 s +o(k),$$ which exceeds $2 \cdot 10^{-7}k$, as $n \to \infty$, due to~\eqref{eq:3.7}.

    \end{proof}

\begin{claim}
    
    For the random $\pi|_{V(\mathcal{T_C})}$ the following event happens with probability $1 - o(k^{-10\varepsilon k})$. 
    
    For any choice of $\pi|_{V(\mathcal{C})\setminus V(\mathcal{T_C})}$, at least $30\varepsilon k$ edges $uv \in E(\mathcal{K})$ satisfy $D$.

    \label{cl:3_r4_r5}
\end{claim}

    \begin{proof}

    Let us give a brief plan of the proof. Define $E_1$ and $E_2$ as in Claim~\ref{cl:3_r5}, notice that the sets $E_1$ and $E_2$ do not depend on $\pi$. By Claim~\ref{cl:3_r5}, we can assume that $E_2$ is large. We fix $\pi$ for every vertex of $V(\mathcal{T_C})$ one by one in some order. Let $uv \in E_2$ with $u$ after $v$ in the order. We choose the order in a way that the moment before we fix $u$, the random $\pi(u)$ has at most one solution for $|\varphi(\pi(u)) - \varphi(\pi(v))| = |\pi(u) - \pi(v)|$. Choosing one such edge for every vertex $u$, when possible, we make up the set $E_5$. We then show that $|E_5|$ is large. Next, we conclude that the probability that $D(uv)$ does not hold is small, conditioned on the whole history of the preceding process. This allows us to conclude that many edges satisfy $D$.
    
    Let $E_1$ and $E_2$ be as in Claim~\ref{cl:3_r5}. Note that if $|E_2|< 10^{-7}k$ then Claim~\ref{cl:3_r5} completes the proof deterministically. So, let $|E_2| \geqslant 10^{-7}k$.
    
    Recall that, for $e\in E(\mathcal{K})$, $\mathcal{P}_e$ is the corresponding 2-path in $\mathcal{C}$.

    We first fix the random positions $\pi$ of a carefully chosen subset of $V(\mathcal{T_C})$. Then, we describe the properties that this choice gives us. Lastly, we sequentially fix the random positions of remaining vertices from $V(\mathcal{T_C})$.
    
    For every $e \in E_1$, fix $\pi$ for every vertex $v$ lying on the path $\mathcal{P}_e$. Let $U \subseteq V(\mathcal{K})$ denote the set of $v\in V(\mathcal{K})$ such that we did not fix $\pi(v)$. Due to the definition of $U$ and of sets $E_1, E_2$, the following holds.
    \begin{equation}
        E_2 \subseteq \{uv \in E(\mathcal{K}): u  \in U\}.
        \label{eq:3.18}
    \end{equation}
    In fact, the set $U$ coincides with $U$ from the proof of Claim~\ref{cl:3_r5}.
    
    Notice that we fixed $\pi$ just on a subset of ${\pi(V(\mathcal{T_C}))}$. Let us now introduce the following notation: for an arbitrary extension of $\pi$ to $V(\mathcal{T_C})$ and arbitrary proper $\pi(\mathcal{T_C})$-rigid map $\varphi$ let us denote
    $$\delta(u) := \varphi(\pi(u)) - \pi(u), \textnormal{ for every }u\in V(\mathcal{K}).$$
    Notice that it is not instantly clear whether $\delta$ is defined, i.e. whether there exists a proper $\pi(\mathcal{T_C})$-rigid map. However, throughout the proof we can assume that $\delta$ is defined: if the final fixed positions $\pi|_{V(\mathcal{T_C})}$ do not have a proper $\pi(\mathcal{T_C})$-rigid map, then the conclusion of Claim~\ref{cl:3_r4_r5} follows deterministically. Indeed, we simply need to show that there are at least $30\varepsilon k$ edges $uv \in E(\mathcal{K})\setminus E(\mathcal{T_K})$ such that $u \nsim v$, which, in turn, is an instant corollary of R5 and P1 as, for large $n$, 
    \begin{equation*}
        |E(\mathcal{K})| - |\{uv \in E(\mathcal{K}): u\sim v\}| - |E(\mathcal{T_K})|  \stackrel{P1, R5}\geq(1.5k - 1.078)k - o(k) \geq 30\varepsilon k.
    \end{equation*}
        
    We claim that, regardless of the choices of both the extension of $\pi$ to $V(\mathcal{T_C})$ and  the proper $\pi(\mathcal{T_C})$-rigid $\varphi$, the value of $\delta(u)-\delta(v)$ is already determined for every $u, v \in V(\mathcal{K})$:
    \begin{equation}
        \textnormal{the data fixed so far uniquely determines }\delta(u) - \delta(v)\textnormal{, for every }u, v\in V(\mathcal{K}).\label{eq:data_uniqely_determines}
    \end{equation}
    Note that every $u,v \in V(\mathcal{K})$ are connected by a path in $\mathcal{T_K}$. So, it is clearly sufficient for us to show that $\delta(u) - \delta(v)$ is determined for the case $uv \in E(\mathcal{T_K})$. If $u\sim v$, then, by P2, 
    $$\delta(u) - \delta(v) = (\varphi(\pi(u)) - \varphi(\pi(v))) - (\pi(u) - \pi(v)) = 0.$$
    If $u\nsim v$, then $uv \in E_1$ and so $\delta(u) - \delta(v)$ is determined by $\sigma$, since we fixed $\pi$ for the whole path $\mathcal{P}_{uv} \subseteq\mathcal{C}$.

    Now, consider an arbitrary order on the remaining vertices of $V(\mathcal{T_C})$. Below, we start a process that sequentially fixes $\pi$ for the remaining unfixed vertices $v \in V(\mathcal{T_C})$ in this order. 

    Let $W \subseteq U$ be the subset of vertices $u \in U$ such that there is a vertex $v := v(u)$ where $uv \in E_2$ and $v$ has a smaller order than $u$. Let 
    $$E_5:=\{uv: u  \in W,\ v:=v(u) \}.$$
    We will show that 
    $$|E_5| \geq 100 \varepsilon k.$$
    As $|E_5| = |W|$, we shall show $|W| \geq 100 \varepsilon k$. Instead, let us prove that every edge from $E_2$ is incident to at least one vertex of $W$, since then $W$ has at least $|E_2| \geqslant 10^{-7}k$ incident edges in total and hence, by R4,
    $|W| \geqslant 100 \varepsilon k$. Consider $uv \in E_2$. By~\eqref{eq:3.18}, wlog $u \in U$. Also, if $v \in U$ wlog we can assume that $v$ has a smaller order than $u$. Hence, regardless of whether $v \in U$, $\pi(u)$ is fixed after $\pi(v)$. Hence, $uv$ is the edge that proves $u \in W$, as desired.
    
    Below, for every edge $uv \in E_5$, we show that the negation of $D(uv)$ implies the equation $\pi(u) = \pi(v) - \frac{\delta(u)-\delta(v)}{2}$. Then, we bound from above the probability that the equation holds conditioned on any values of the same equation for the previous in $E_5$. This will allow us to conclude the desired bound on probability in Claim~\ref{cl:3_r4_r5}.
    
    Consider the time step when we fix $\pi(u)$ for $u \in W$ and let $v:=v(u)$. Note that $u\nsim v$ and $ uv \in E(\mathcal{K})\setminus E(\mathcal{T_K})$, by the definition of $E_2$. So, the negation of $D(uv)$ implies that there exists a proper $\pi(\mathcal{T_C})$-rigid map $\varphi$ satisfying
    $$|\varphi(\pi(u)) - \varphi(\pi(v))| = |\pi(u) - \pi(v)|.$$ As $u \nsim v$ and $uv \in E(\mathcal{K})$, P3 implies
    $$
        \varphi(\pi(u)) - \varphi(\pi(v)) = -(\pi(u) - \pi(v)).
        $$
    
    This is equivalent to
    \begin{equation}
        \pi(u) = \pi(v) - \frac{\delta(u)-\delta(v)}{2} \stackrel{\eqref{eq:data_uniqely_determines}}= {\textnormal{const}},
    \label{eq:3.21}    
    \end{equation} 
    as $\pi(v)$ is already determined. 
    
    Recall that $|W| \geq 100\varepsilon k$. Let us call a vertex from $W$ {\it early} if it is among the first $50\varepsilon k$ vertices from $W$ in the order. At every step of the process, when we fix the position $\pi$ of an early vertex $w \in W$, the number of remaining unoccupied points of $V$ is always at least $50\varepsilon k$. So, at every step, the probability that $\pi(w)$ hits any particular vertex is at most $1/50\varepsilon k$. So, for every early vertex $w \in W$, we defined the equation~\eqref{eq:3.21}, which holds with probability at most $1/50\varepsilon k$. Furthermore, for every early vertex $w \in W$, conditioned on any values of $\pi(w')$ for the previous vertices $w' \in W$ in the order, the probability  that~\eqref{eq:3.21} holds for $w$ is at most $1/50\varepsilon k$. As a consequence, the probability that among the first $50\varepsilon k$ of the edges $E_5$ there are at least $20\varepsilon k$ of edges that do not satisfy $D$ is at most $2^{50\varepsilon k}(50\varepsilon k)^{-20\varepsilon k} \stackrel{R2}= o(k^{-10\varepsilon k})$,
    so we have at least $30\varepsilon k$ edges satisfying $D$ with probability $1 - o(k^{-10\varepsilon k})$, as desired.

    \end{proof}

    \begin{proof}[Proof of Lemma~\ref{lm:A_formal}]
        Let us show that probability that there exists an $\pi(\mathcal{C})$-rigid map with proper restriction to $\pi(\mathcal{T_C})$ is at most $o(k^{-5\varepsilon k})$. Fix $\pi|_{V(\mathcal{T_C})}$ satisfying the conclusion of Claim~\ref{cl:3_r4_r5}. By R3, there are $u_1v_1, \ldots, u_{29\varepsilon k}v_{29\varepsilon k} \in E(\mathcal{K})$ edges that
        \begin{itemize}
            \item satisfy $D$, Definition~\ref{df:3.};
            \item the corresponding $2$-path in $\mathcal{C}$ has order less than $\ln n /100$.
        \end{itemize}
    
    Note that the remaining probability space is described by a uniformly random injection $$\pi|_{V(\mathcal{C})\setminus V(\mathcal{ T_C})}:V(\mathcal{C})\setminus V(\mathcal{T_C}) \to V\setminus\pi(V(\mathcal{T_C})).$$

    Let us explain the way we can apply Claim~\ref{cl:3.4} to the paths $\mathcal{P}_{uv}$, for $uv \in \{u_1v_1, \ldots, u_{29\varepsilon k}v_{29\varepsilon k}\}$, independently from each other.

    Note that we did not fix the position $\pi$ of the inner vertices of $\mathcal{P}_{uv}$. 
    Also, note that we fixed $\sigma$ and $\pi|_{V(\mathcal{T_C})}$. So, for every proper $\pi(\mathcal{T_C})$-rigid map $\varphi$, we know the exact distance $|\varphi(\pi(v)) - \varphi(\pi(u))|$, which in turn does not equal $|\pi(v) - \pi(u)|$ due to $D(uv)$, Definition~\ref{df:3.}. In particular, if for some $uv \in \{u_1v_1, \ldots, u_{29\varepsilon k}v_{29\varepsilon k}\}$ the path $\mathcal{P}_{uv}$ does not contain inner vertices then $uv \in E(\mathcal{C})$ and $|\varphi(\pi(u)) - \varphi(\pi(v))| \neq |\pi(u) - \pi(v)|$ contradicts to the event in~\eqref{eq:3.5} deterministically. So, we can assume that, for every $uv \in \{u_1v_1, \ldots, u_{29\varepsilon k}v_{29\varepsilon k}\}$, the path $\mathcal{P}_{uv}$ contains at least one inner vertex. 
    Hence, as we sequentially fix the positions $\pi$ of the paths $\mathcal{P}_{u_1v_1}, \ldots, \mathcal{P}_{u_{29\varepsilon k}v_{29\varepsilon k}}$, the moment we do that for the first $28\varepsilon k$ paths there are at least $\varepsilon k$ vertices unoccupied  by $\pi$. Thus, Claim~\ref{cl:3.4} is applicable and gives us an upper bound of $(\varepsilon k)^{-1/2}$ on the probability of the event described in~\eqref{eq:3.5} for each of them independent of the previous paths. So, multiplying the bound over the paths
    we get an altogether upper bound of $(\varepsilon k)^{-28\varepsilon k/2}$, which implies~\eqref{eq:3.5}.
    \end{proof}

\section{Describing the reconstructible subset}
\label{sc:4}

In this section we define a property $D$ over vertices of $V(\mathcal{K})$ and then state and prove in Lemma~\ref{lm:2} that the set of the vertices satisfying $D$ is reconstructible whp. Let us start with the necessary definitions.

Recall that a \emph{2-path} is a path in which all internal vertices have degree $2$. Also, recall that for a vertex subset $S \subseteq V(\mathcal{K})$, $N_{\mathcal{K}}(S)$ denotes the neighbourhood set of $S$.

Let us define a quantity $\mathbf{E}(S)$. Since edges represent known distances between vertices, our intuition for this quantity is that $\mathbf{E}(S)$, in a certain sense, bounds the number of $\pi(S)$-rigid maps.

\begin{definition}[$\mathbf{E}(S)$]
    Let $\mathcal{C}$ be a connected graph with minimum degree $2$, and let $\mathcal{K}$ be its kernel. Let $S \subseteq V(\mathcal{K})$ induce a connected multigraph. Define $\mathbf{E}(S)$ to be the number of edges $e \in E(\mathcal{C})$ such that $e$ lies on a 2-path in $\mathcal{C}$ that has at least one end-vertex from $S$.
    \label{df:E}
\end{definition}

\begin{definition}[Event D]
    Let $\beta >0$, let $\mathcal{C}$ be a connected graph with minimum degree $2$, and let $\mathcal{K}$ be its kernel. For $v \in V(\mathcal{C})$, let $D:= D_{\beta}(v)$ denote the event that
    \begin{itemize}
        \item $v \in V(\mathcal{K})$,
        \item for all $S \subseteq V(\mathcal{K})$ inducing a connected subgraph in $\mathcal{K}$ with $v \in S$ and $|S| \leq (1- \beta)|V(\mathcal{K})|$ it holds that:
            $$\mathbf{E}(S) \leqslant \beta|S| \ln n \quad\textnormal{ and }\quad |N_{\mathcal{K}}(S)| \geqslant 3.$$
    \end{itemize} 

    \label{df:A}
\end{definition}

When it is clear from the context, we will omit writing $\beta$ for the event $D$.

Section~\ref{sc:4} is dedicated to the following lemma.

\begin{lemma}
There exists $\beta>0$ such that, for $\lambda$ and $\mathcal{G} \sim \mathcal{U}$ from Claim~\ref{cl:contiguous}, if $\lambda$ satisfies $\lambda \geq 1+\frac{1}{\beta\ln n}$ for every large enough $n$, the set $\{v \in V(\mathcal{G}): D_{\beta}(v)\textnormal{ holds}\}$ is reconstructible whp.
\label{lm:2}
\end{lemma}

The strength of this lemma is that, for any $\beta>0$, in a slightly more restricted context ($\lambda = 1+ \Omega_{\beta}\left(\frac{\ln \ln n}{\ln n}\right)$), whp $D_{\beta}(v)$ holds for almost all $v \in V(\mathcal{K})$ --- this will be shown in the next section.

Lemma~\ref{lm:2} will be an instant corollary of the following Lemma~\ref{lm:super2}. 
\begin{lemma}
     Let $V \subseteq \mathbb{R}$, $|V| = n$. Let a connected graph  $\mathcal{C}$ with minimum degree $2$ of size at most $n$, its kernel $\mathcal{K}$, and a uniformly random injection $\pi: V(\mathcal{C}) \to V$ satisfy the following properties: for some $\alpha > 0$, some $c>0$, and $k:= |V(\mathcal{K})|$
    \begin{itemize}
        \item[R1$'$] $\mathcal{K}$ is $(1-c, \alpha)$ vertex expander with $\ln k = (1-o(1)) \ln n$;
        \item[R2$'$] for $s\geqslant 1$, $|\{S \subseteq V(\mathcal{K}) \textnormal{  induces connected graph},\  |S| = s\}| \leqslant k\exp((s-1)\cdot o(\ln n))$;
        \item[R3$'$] Whp, the random injection $\pi$ satisfies the following property:
        for every $\pi(\mathcal{C})$-rigid map $\varphi:V \to \mathbb{R}$ there is a set $U \subseteq V(\mathcal{K})$ with $|U| \geqslant ck$  such that
        $$|u - v| = |\varphi(u) - \varphi(v)|\textnormal{ for every $u, v \in \pi(U)$.}$$
    \end{itemize}
    Then, for small enough $\beta:=\beta(\alpha, c) \in (0, c)$,  whp
$$\pi(\{v \in V(\mathcal{K})\mid D_{\beta}(v)\textnormal{ holds}\}),$$
    is reconstructible in $\pi(\mathcal{C})$.
    \label{lm:super2}
\end{lemma}

Notice that $V(\mathcal{C}) \leq n$ is compulsory for $\pi$ to be defined. Also, we will show that R2$'$ follows from the $o(\ln n)$ upper bound on the maximal degree of $\mathcal{K}$ and, hence, R2$'$ can be substituted by that property.

Let us show that Lemma~\ref{lm:super2} implies Lemma~\ref{lm:2}. The size of $\mathcal{C}$ is at most $n$ by Claim~\ref{cl:core_size}, R1$'$ is a corollary of Claim~\ref{cl:expansion}, R2$'$ follows from Claim~\ref{cl:graph_bound} and S3 of Claim~\ref{cl:S}. R3$'$ is shown in Lemma~\ref{lm:1}.

So, in what remains, we show Lemma~\ref{lm:super2}.

The proof of Lemma~\ref{lm:super2} is organised similarly to the proof of Theorem~\ref{th:super1}. We first describe an event $B$ that implied by the failure of Lemma~\ref{lm:super2}. Next, we prove the implication. Lastly, we prove that $\mathbb{P}(B) = o(1)$.

\begin{definition}[Event $B$]
    There exists a subset $S \subseteq V(\mathcal{K})$ and a (not necessarily injective) mapping $\varphi: \pi(V(\mathcal{C})) \to \mathbb{R}$ that preserves distances over the edges of $\pi(\mathcal{C})$ such that
\begin{itemize}
    \item[B1] $S$ induces a connected multigraph in $\mathcal{K}$;
    \item[B2]$\bar V := \{v \in V(\mathcal{K}) \mid \varphi(\pi(v)) = \pi(v)\}$ contains at least $ck$ vertices;
    \item[B3] $\bar V \cap S = \emptyset$;
    \item[B4] $N:= N_{\mathcal{K}}(S) \subseteq \bar V$;
    \item[B5] $\mathbf{E}(S) \leqslant \beta|S| \ln n$;
    \item[B6] $|N| \geqslant 3$.
\end{itemize}
\end{definition} 

\begin{claim}
    Let $\pi$ satisfy the property given in R3$'$. If the set
    $\pi(\{v \in V(\mathcal{K})\mid D(v)\textnormal{ holds}\})$
    is not reconstructible, then $B$ holds.
    \label{cl:B_implication}
\end{claim}

\begin{proof}
    Consider the $\pi(\mathcal{C})$-rigid map $\varphi:V\to\mathbb{R}$ disproving that $\pi(\{v \in V(\mathcal{K})\mid D(v)\textnormal{ holds}\})$ is reconstructible. Let us show that the event $B$ holds.
    
    By the property given in R3$'$, there is a linearly large subset $U \subseteq V(\mathcal{K})$ such that $\varphi$ preserves the pairwise distances between vertices from $\pi(U)$. Let us apply an isometry (i.e. a symmetry and/or adding a constant) to $\varphi$ in such a way that $\varphi|_{\pi(U)} = id$. Note that the resulting $\pi(\mathcal{C})$-rigid map $\varphi$ also disproves that $\pi(\{v \in V(\mathcal{K})\mid D(v)\textnormal{ holds}\})$ is reconstructible. Therefore, we can assume that
    \begin{equation}
        |\{v \in \pi(V(\mathcal{K})): \varphi(v) = v\}| \geq ck.
        \label{eq:B_event}
    \end{equation}

\begin{figure}[!h]
    \centering
    \includegraphics[scale=0.2]{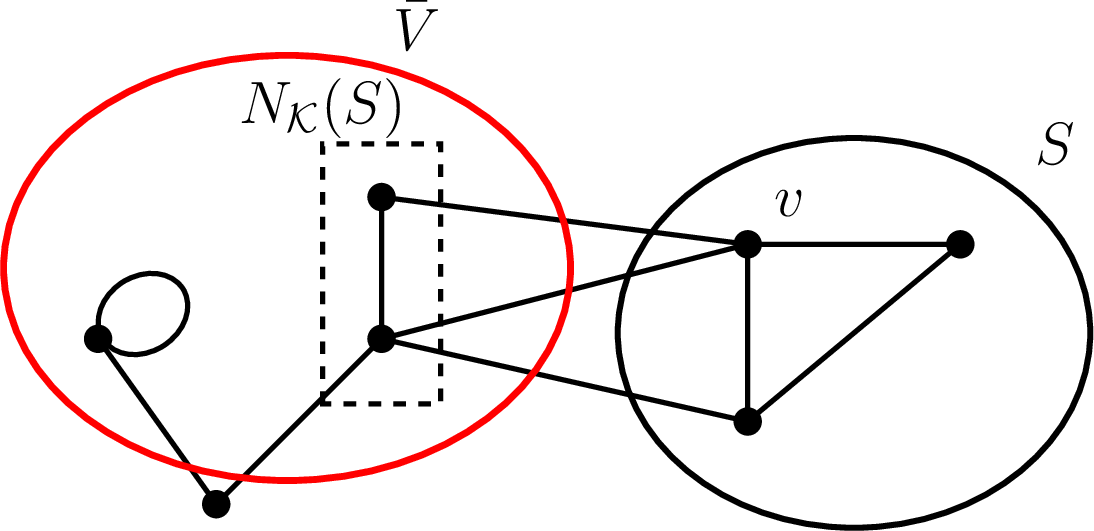}
    \caption{The figure illustrates the structure of $S, N_{\mathcal{K}}(S)$, and $\bar V$ in $\mathcal{K}$.
}
    \label{fig:3}
\end{figure}

Let us set $\bar V = \{v \in V(\mathcal{K}) \mid \varphi(\pi(v)) = \pi(v)\}$, so B2 holds due to~\eqref{eq:B_event}.
Since $\varphi$ disproves that $\pi(\{v \in V(\mathcal{K})\mid D(v)\textnormal{ holds}\})$ is reconstructible, there exists a vertex $v \in V(\mathcal{K})$ such that $D(v)$ holds but $\varphi(\pi(v)) \neq \pi(v)$. Let $S \subseteq V(\mathcal{K})$ be the vertex set of the component of $\mathcal{K} - \bar V$ containing $v$ (see Figure~\ref{fig:3}). Then, B1 holds trivially. By definition of $S$, the equation $\varphi(\pi(u)) = \pi(u)$ holds for every vertex $u \in N_{\mathcal{K}}(S)$ and, hence, the properties B3 and B4 hold. Also, B5 and B6 are instant corollaries of the event $D(v)$ applied to the set $S$, recalling $\beta \leq c$.

\end{proof}

Due to Claim~\ref{cl:B_implication}, in order to conclude Lemma~\ref{lm:super2} it remains to show the following lemma.

\begin{lemma}
    $\mathbb{P}(B) = o(1)$.
    \label{lm:4.b}
\end{lemma}

\begin{proof}

    We bound $\mathbb{P}(B)$ using a union bound over subsets $S$ and mappings $\varphi$ in the event. More precisely, for a vertex set $\tilde S \subseteq V(\mathcal{K})$ and $\sigma \in \{-1, 1\}^{E(\mathcal{K})}$, let $B(\tilde S, \sigma)$ denote the event that there is a pair $(S, \varphi)$ verifying
    $B$ such that \begin{itemize}
        \item $S = \tilde S$;
        \item $\varphi$ has binary representation $\sigma$.
    \end{itemize}
    First, in Claim~\ref{cl:4_prob}, we will show that $\mathbb{P}(B(\tilde S, \sigma))$ is sufficiently small. Next, in Claim~\ref{cl:4_union}, we will show the upper bound on $\mathbb{P}(B)$ using the described union bound. Lastly, we will verify that Lemma~\ref{lm:4.b} indeed follows from the bounds shown in Claim~\ref{cl:4_prob} and Claim~\ref{cl:4_union}.

    \begin{claim}
    For every $1 \leq s \leq (1-c)k$, every $\tilde S \subseteq V(\mathcal{K}) ,\  |\tilde S| = s$, and every $\sigma \in \{-1, 1\}^{E(\mathcal{K})}$ it holds that
    $$ \mathbb{P}(B(\tilde S, \sigma)) \leqslant \frac{1}{\Omega(n)^{\max(2, \Omega(s))}}.$$
    \label{cl:4_prob}
\end{claim}

\begin{proof}[Proof of Claim~\ref{cl:4_prob}]

Consider arbitrary $s \in [\lfloor(1-c)k\rfloor]$ and arbitrary $\tilde S \subseteq V(\mathcal{K}) ,\  |\tilde S| = s$, and arbitrary $\sigma \in \{-1, 1\}^{E(\mathcal{K})}$. Let $N := N_{\mathcal{K}}(\tilde S)$. Let us assume that $\mathbb{P}(B(\tilde S, \sigma)) > 0$, since otherwise Claim~\ref{cl:4_prob} follows instantly. Let us show that $\mathbb{P}(B(\tilde S, \sigma)) \leqslant \frac{1}{\Omega(n)^{\max(2, \Omega(s))}}.$

The proof is structured as follows. First, we describe a set of vertices $F \subseteq V(\mathcal{C})$ in such a way that after fixing the random positions $\pi|_{F}$ the positions of the remaining vertices become highly constrained. More precisely, we will next show that there are $\Omega(s)$ of vertices $u \in V(\mathcal{C})\setminus F$ such that $u$ has just one position on the real line where the value of $\pi(u)$ can land. Lastly, we will conclude from that the desired bound $\mathbb{P}(B(\tilde S, \sigma)) \leqslant \frac{1}{\Omega(n)^{\max(2, \Omega(s))}}$.

We now want to describe the set $F \subseteq V(\mathcal{C})$. In the course of doing so, we will also do some preparatory work for finding a linearly large subset of vertices $u \in V(\mathcal{C})\setminus F$ with constrained positions $\pi(u)$.

Consider an arbitrary 2-path $\mathcal{P}$ connecting a vertex $v \in N$ to a vertex $v' \in \tilde S$. Define $u := u(\mathcal{P}) \in V(\mathcal{P})$ to be the closest vertex to $v$ (in $\mathcal{P}$) among the vertices incident to an edge $e \in E(\mathcal{P})$ satisfying $\sigma(e) = -1$. If no such edge exists, we say that $u(\mathcal{P})$ is not defined.

Let us show that $u(\mathcal{P})$ is always defined. Recall, that we assumed that $\mathbb{P}(B(\tilde S, \sigma)) > 0$. Consider here the mappings $\pi:\mathcal{C} \to V$ and $\varphi:\pi(\mathcal{C}) \to \mathbb{R}$ that verify $B(\tilde S, \sigma)$. Observe that every edge of $\mathcal{P}$ is assigned a sign by $\sigma$. Let $\varphi$ be the mapping that witnesses $B(\tilde S, \sigma)$ and has binary representation $\sigma$. By properties B2 and B3, we know that $\varphi(v) = v$ and $\varphi(v') \neq v'$. If $\sigma(e) = 1$ for all $e \in E(\mathcal{P})$, then it clearly follows by induction that $\varphi(w) = w$ for every $w \in V(\mathcal{P})$, and, in particular, $\varphi(v') = v'$, contradicting the assumption. Therefore, for every 2-path $\mathcal{P}$ connecting some $v \in N$ to some $v' \in S$, there must exist an edge $e \in E(\mathcal{P})$ with $\sigma(e) = -1$.

We now sequentially construct two sets of vertices, $U$ and $\delta U$, which, in particular, will be used for constructing $F$. For each $v \in N$, choose an arbitrary 2-path $\mathcal{P}$ from $v$ to some arbitrary $v' \in \tilde  S$. We add $u(\mathcal{P})$ to $\delta U$, and we add every internal vertex of the subpath from $u$ to $v'$ to $U$.

\begin{figure}[!h]
  \centering
  \raisebox{+20mm}{\begin{subfigure}[b]{0.38\textwidth}\includegraphics[width=\textwidth]{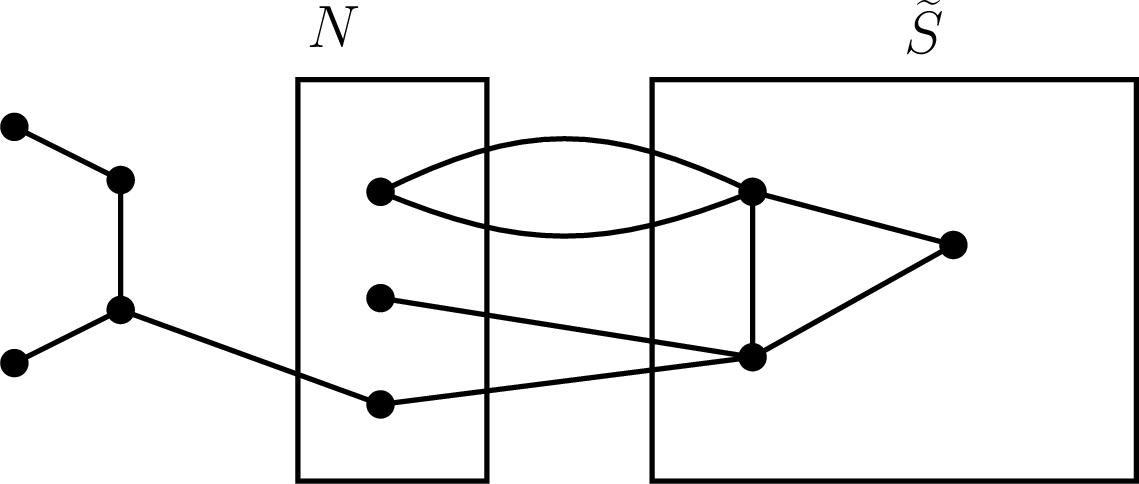}
  \end{subfigure}}
  \hspace{+5pt}
  \begin{subfigure}[b]{0.58\textwidth}\includegraphics[width=\textwidth]{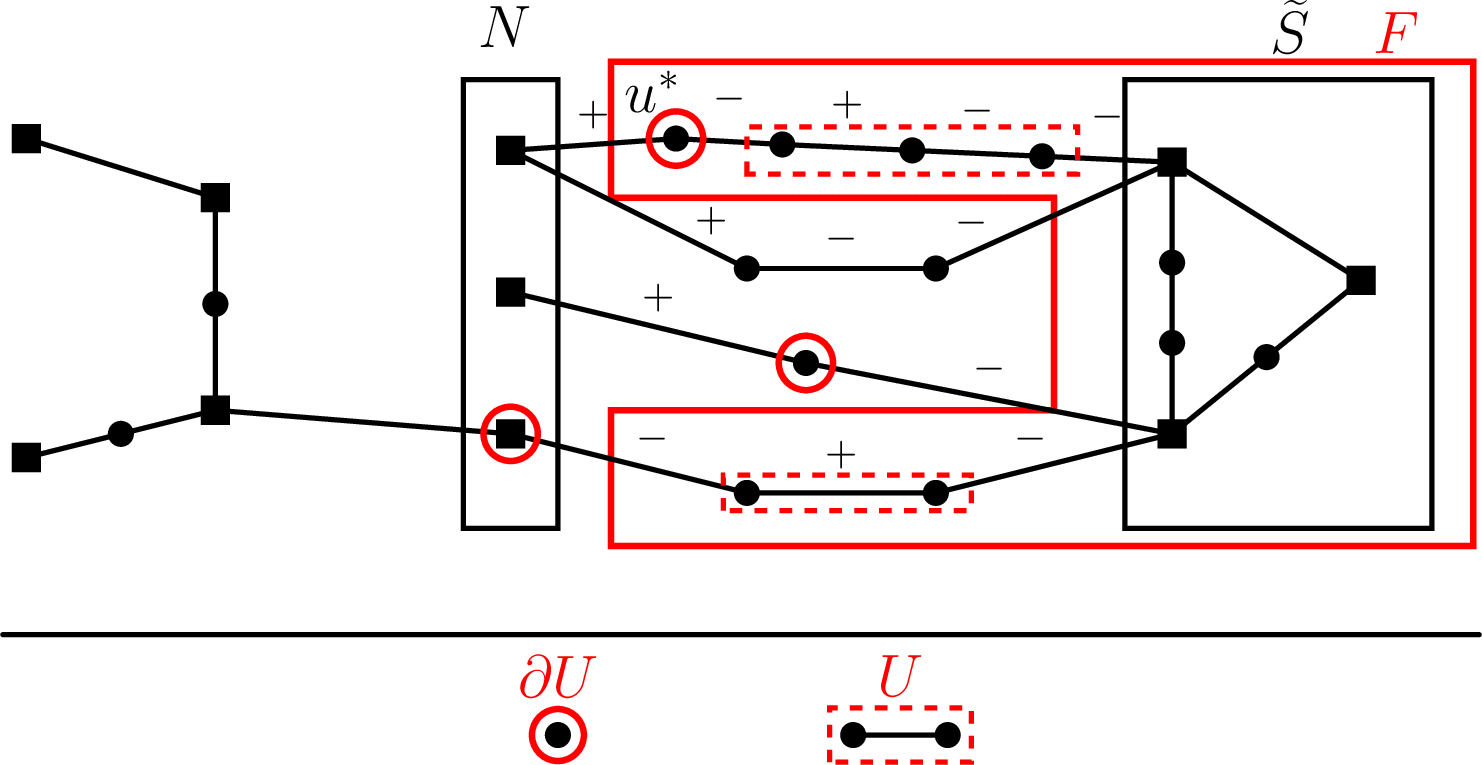}
  \end{subfigure}
  \caption{The figure illustrates the structure of $\tilde S, N, \delta U, U$, and $F$ in $\mathcal{K}$ and $\mathcal{C}$. The kernel $\mathcal{K}$ in the left figure is similar to $\mathcal{K}$ from the toy example, Figure~\ref{fig:2}, up to the following facts. Now, $u^*$ can lie outside of the kernel and the set $F$ of vertices $v \in \mathcal{C}$ with fixed value of $\pi(v)$ has a more complex description --- see the image on the right. The image on the right shows the graph $\mathcal{C}$. The common vertices of $\mathcal{K}$ and $\mathcal{C}$ are indicated by filled squares and the other vertices of $\mathcal{C}$ are indicated by filled dots. For every $e \in E(\mathcal{K})$ and every edge $uv \in \mathcal{P}_e$, we assign a symbol from $\{+, -\}$, which corresponds to the sign of the formula $\frac{\varphi(\pi(v))-\varphi(\pi(u))}{\pi(v)-\pi(u)}$.
  The set $U$ consists of the vertices lying in dashed red rectangles. The set $\delta U$ consists of individual vertices encircled in red.}
  \label{fig:5}
\end{figure}

Note that, as $\mathbb{P}(B(\tilde S, \sigma))>0$,
\begin{equation}
    |\delta U| = |N| \stackrel{\text{B6}}{\geqslant} 3.
    \label{eq:4.3}
\end{equation}

The set $F$ is defined as follows:
\begin{itemize}
    \item First, include all vertices of $\tilde S$;
    \item Next, include every internal vertex of each 2-path whose endpoints lie in $\tilde S$;
    \item Then, include all vertices in the set $U$;
    \item Finally, if $\delta U \neq \emptyset$, add one arbitrary vertex $u^* \in \delta U$ to $F$,
\end{itemize}
see Figure~\ref{fig:5}.

The following properties of the set $F$ will be important for us:

\begin{itemize}
    \item $F$ is connected in $\mathcal{C}$ and contains $\tilde S$;
    \item For every $u \in \delta U \setminus\{u^*\}$, there is a vertex $w \in F$ adjacent to $u$.
\end{itemize}

The first property follows from B1 and the second property is trivial from the definition of $F$.

Now, let us show the bound $ \mathbb{P}(B(\tilde S, \sigma)) \leqslant \frac{1}{\Omega(n)^{\max(2, \Omega(s))}}$.

Fix $\pi|_{F}$. So, the remaining probability space is described by the uniformly random injection
$$\pi|_{V(\mathcal{C})\setminus F}:V(\mathcal{C})\setminus F \to V\setminus\pi(F).$$

Below, we prove that, for $B(\tilde S, \sigma)$ to hold, every vertex $u \in \delta U \setminus\{u^*\}$ must satisfy $\pi(u) = r_{u}$, for some constant $r_u \in \mathbb{R}$. Thus, by fixing sequentially the values of $\pi(\delta U \setminus\{u^*\})$ we will get the desired upper bound on the probability.

Suppose that $B(\tilde S, \sigma)$ holds and let $\varphi$ be the mapping with binary representation $\sigma$ such that the pair $(\tilde S, \varphi)$ witnesses $B$. Let us show that $\varphi(\pi(u)) = \pi(u)$ for every $u \in \delta U$. Recall that $u$ lies on a 2-path connecting a vertex $v \in N$ to a vertex $v' \in S$. Also, $u$ was chosen that way that $\sigma(e) = 1$ for every edge $e$ in the subpath from $u$ to $v$. Also, by B2 and B4, $\varphi(\pi(v)) = \pi(v)$. Therefore, $\varphi(\pi(w)) = \pi(w)$ holds for every vertex $w$ on the 2-path between $u$ and $v$ by induction, as desired. 

The property we just proved gives us, in particular, $\varphi(\pi(u^*)) = \pi(u^*)$. Hence, for every $v \in F$ both $\pi(u)$ is fixed and, via $\sigma$, the value of $\varphi(\pi(u))$ is uniquely determined on the real line.

Recall that every $u \in \delta U$, $u \neq u^*$ has a neighbour $w \in F$, due to the definition of $F$. Recall, that $\sigma(uw) = -1$, due to the definition of $u$. So, first, this means that $\varphi(\pi(w)) \neq \pi(w)$, as $\varphi(\pi(u)) = \pi(u)$. Next, this means that the only position that $\pi(u)$ can have is the middle point between $\pi(w)$ and $\varphi(\pi(w))$, since we have 
$$|\varphi(\pi(w)) - \pi(u)|=|\varphi(\pi(w)) - \varphi(\pi(u))| \stackrel{\textnormal{the definition of }\varphi}= |\pi(w)- \pi(u)|.$$ 
So, the probability that $\pi(u)$ is exactly equal to $\frac{\pi(w) + \varphi(\pi(w))}{2}$ is at most the inverse of the number of unoccupied positions of $V$.

So, let us start the process fixing the values $\pi(u)$ sequentially for $u \in \delta U\setminus\{u^*\}$ and apply the above mentioned bound on $\mathbb{P}\left(\pi(u) = \frac{\pi(w) + \varphi(\pi(w))}{2}\right)$. Note that, at the start of the process, the only vertices of $\mathcal{K}$ with fixed position $\pi$ are $\tilde S$ and, perhaps $u^*$, in case it lies inside the kernel. As $|\tilde S| \leq (1-c)k$, there are at least $ck-1$ unoccupied vertices in $V$. Next, at the step $1 \leq i \leq \min(|\delta U \setminus \{u^*\}|, ck-1)$ the number of unoccupied positions of $V$ is always at least $ck - i$, so $\mathbb{P}\left(\pi(u) = \frac{\pi(w) + \varphi(\pi(w))}{2}\right) \leq \frac{1}{ck - i}$. Also, by~\eqref{eq:4.3}, $|\delta U| = |N|$ and $|N| \geqslant \alpha|\tilde S| = \alpha s$ by R1$'$. So, we get the overall upper bound
\begin{align}
    \max_{\sigma \in \{-1, 1\}^{E(\mathcal{K})}} \mathbb{P}(B(\tilde S, \sigma)) &\leqslant \frac{1}{ck-1} \cdot \frac{1}{ck-2} \ldots \cdot \frac{1}{\max(1, \lceil ck - |\delta U \setminus \{u^*\}| \rceil) } \notag \\
    &\stackrel{\eqref{eq:4.3}}=\frac{1}{\Omega(k)^{\max(2, \Omega(s))}} \stackrel{R1'}= \frac{1}{\Omega(n)^{\max(2, \Omega(s))}},
    \label{eq:4.4}
\end{align}
as desired.

\end{proof}

Now, we explain our union bound. It is given and proven in the following claim.

\begin{claim}
    $$\mathbb{P}(B) \leq \sum_{s = 1}^{(1-c)|V(\mathcal{K})|} |V(\mathcal{K})|\cdot  e^{(|S|-1) \cdot o(\ln n)} \cdot n^{\beta \ln 2 \cdot |S|} \cdot  \max_{S \subseteq V(\mathcal{K}),\ |S| = s,\ \sigma \in \{-1, 1\}^{E(\mathcal{K})}} \mathbb{P}(B(\tilde S, \sigma)).$$
    \label{cl:4_union}
\end{claim}

\begin{proof}

For a set $\tilde S \subseteq V(\mathcal{K})$, let us denote by $B(\tilde S)$ the event $B$ holds with on $S:= \tilde S$. Note that B2 and B3 imply $|\tilde S| \leqslant (1-c)|V(\mathcal{K})|$. So, we have the following bound

\begin{align}
    \mathbb{P}(B) &\leqslant \sum_{s =1}^{(1-c)|V(\mathcal{K})|}\sum_{\tilde S \subseteq V(\mathcal{K}) \textnormal{ induces connected graph},\  |\tilde S| = s} \mathbb{P}(B(\tilde S)) \nonumber \\
    &\leqslant \sum_{s =1}^{(1-c)|V(\mathcal{K})|} |\{\tilde  S \subseteq V(\mathcal{K}) \textnormal{ induces connected graph},\  |\tilde S| = s\}|\max_{\tilde S \subseteq V(\mathcal{K}) ,\  |\tilde S| = s}\mathbb{P}(B(\tilde S)) \nonumber \\
    &\stackrel{R2'}\leqslant \sum_{s =1}^{(1-c)|V(\mathcal{K})|} |V(\mathcal{K})|\exp((s-1) \cdot o(\ln n))\max_{\tilde S \subseteq V(\mathcal{K}) ,\  |\tilde S| = s}\mathbb{P}(B(\tilde S)).
\label{eq:4.1}
\end{align}

Consider an arbitrary $\tilde S \subseteq V(\mathcal{K})$ that induces connected subgraph and let us now bound $\mathbb{P}(B(\tilde S))$.

Let $H$ be the set of edges that contribute towards the sum $\mathbf{E}(S)$. Suppose that $B(\tilde S, \sigma)$ holds for some $\sigma \in \{-1, +1\}^{E(\mathcal{C})}$. Then, for $\tilde \sigma \in \{-1, +1\}^{E(\mathcal{C})}$ such that $\tilde \sigma|_{H} = \sigma|_{H}$ and $\tilde \sigma|_{E(\mathcal{C})\setminus H} = 1$ the event $B(\tilde S, \tilde \sigma)$ holds as well, since that does not break any properties from B1-B6. Thus, it is sufficient for us to consider a set of $2^{\mathbf{E}(\tilde S)}$ binary representations. Hence, we get

\begin{align}
\mathbb{P}(B(\tilde S)) 
&\leqslant 2^{\mathbf{E}(\tilde S)} \max_{\sigma \in \{-1, 1\}^{E(\mathcal{C})}} \mathbb{P}(B(\tilde S, \sigma)) \nonumber\\
&\stackrel{B5}\leqslant n^{\beta \ln 2 \cdot |\tilde S|} \max_{\sigma \in \{-1, 1\}^{E(\mathcal{C})}} \mathbb{P}(B(\tilde S, \sigma)). \label{eq:4.2}
\end{align}

So,~\eqref{eq:4.1} and~\eqref{eq:4.2} conclude the desired bound.

\end{proof}

We can now complete the proof of Lemma~\ref{lm:4.b}: by combining Claim~\ref{cl:4_prob} and Claim~\ref{cl:4_union} we get
\begin{align*}
    \mathbb{P}(B) 
    &\leqslant \sum_{s = 1}^{(1-c)|V(\mathcal{K})|} |V(\mathcal{K})|\exp((s-1) \cdot o(\ln n)) \cdot n^{\beta s \ln2} \cdot \frac{1}{\Omega(n)^{\max(2, \Omega(s))}} \\
    &\stackrel{|V(\mathcal{K})| \leq |V(\mathcal{C})|\leq n}\leqslant  \sum_{s =1}^{(1-c)|V(\mathcal{K})|} \exp((s-1) \cdot o(\ln n)) \cdot n^{\beta s \ln2}
    \frac{1}{\Omega(n)^{\max(1, \Omega(s))}}= o(1),
\end{align*}
for small enough $\beta > 0$.

This completes the proof of Lemma~\ref{lm:4.b}.

\end{proof}

\section{Bounding the size of the reconstructible subset}
\label{sc:5}

In this section, we shall prove Lemma~\ref{lm:3} and Lemma~\ref{lm:4}, which are given below. Roughly speaking, Lemma~\ref{lm:3} says that almost all vertices of the kernel of $\mathcal{C}$ satisfy the event $D$, Definition~\ref{df:A}. So, using Claim~\ref{cl:contiguous}, Lemma~\ref{lm:2}, and Lemma~\ref{lm:3}, we will conclude that almost all vertices of the kernel are reconstructible. Lemma~\ref{lm:3}, combined with the contiguity model Claim~\ref{cl:contiguous} and Lemma~\ref{lm:2}, implies the conjecture by Gir\~ao, Illingworth, Michel, Powierski and Scott. Meanwhile, Lemma~\ref{lm:4} will extend the statement to almost all vertices of $\mathcal{C}$, which completes the proof of Theorem~\ref{th:1} 
due to Claim~\ref{cl:contiguous}. 

Let us note that in this section we work in a more restricted context compared to Section~\ref{sc:3} and Section~\ref{sc:4}. Both Lemma~\ref{lm:1} and Lemma~\ref{lm:2} were shown for $\lambda = 1 + \Omega(1/\ln n)$. Both Lemma~\ref{lm:3} and Lemma~\ref{lm:4}, however, are given for $\lambda>1$ being constant. This is because the proof of the analogues of Lemma~\ref{lm:3} and Lemma~\ref{lm:4} for $\lambda \to 1$ will be shown in Section~\ref{sc:ad}. Let us however mention that one could easily generalise the lemmas to the case of $\lambda = 1 + \Omega(\ln \ln n/\ln n)$, as was described in Section~\ref{sc:1}.

\begin{lemma}
Let $\beta > 0$. Let $\lambda$ and $\mathcal{K}$ be from Claim~\ref{cl:L-contiguous}. Let $\lambda>1$ be constant. Whp 
$$|\{v \in V(\mathcal{K}) : \lnot D_{\beta}(v)\}| = o(|V(\mathcal{K})|).$$
\label{lm:3}
\end{lemma} 

\begin{lemma}
    Let $\lambda$ and $\mathcal{G} \sim \mathcal{U}$ be from Claim~\ref{cl:contiguous}. Let $\lambda>1$ be constant. Whp there exists a reconstructible subset in $\mathcal{G}$ of size $(1-o(1))|V(\mathcal{G})|$.
    \label{lm:4}
\end{lemma}

Below we give Lemma~\ref{lm:super4} that makes the same conclusions as Lemma~\ref{lm:4} for every appropriate sequence of graphs. We will prove Lemma~\ref{lm:super4} and deduce Lemma~\ref{lm:4} from it. Note that in the following lemma the set $U$ only depends on the graph $\mathcal{C}$ and $\lambda$ and, in contrast to $U$ from Definition~\ref{df:A(G, c)}, does not depend on the injection $\pi: V(\mathcal{C}) \to \mathbb{R}$ and on the $\pi(\mathcal{C})$-rigid map $\varphi: V \to \mathbb{R}$, see R3$'$ from Lemma~\ref{lm:super2}.

\begin{lemma}
    Let $V \subseteq \mathbb{R}$ with $|V| = n$. Let a connected graph $\mathcal{C}$ with minimum degree $2$ of size at most $n$, its kernel $\mathcal{K}$, a subset $U \subseteq V(\mathcal{K})$, and a uniformly random injection $\pi: V(\mathcal{C}) \to V$ satisfy the following properties:
    \begin{itemize}
        \item[R1$^*$] whp, $\pi(U)$ is a reconstructible subset in $\pi(\mathcal{C})$;

        \item[R2$^*$] the graph induced by $U$ contains all but $o(|E(\mathcal{K})|)$ edges of $\mathcal{K}$;

        \item[R3$^*$] $|\{e\in E(\mathcal{K}): \textnormal{the corresponding 2-path }\mathcal{P}_e\textnormal{ in }\mathcal{C}\textnormal{ satisfies }|V(\mathcal{P}_e)| > \ln n/100\}| = o(|E(\mathcal{K})|)$;

        \item[R4$^*$] for every $F \subseteq E(\mathcal{K})$ with $|F| = o(|E(\mathcal{K})|)$, we have $|\cup_{e \in F} V(\mathcal{P}_e)| = o(|V(\mathcal{C})|)$.
    \end{itemize}
    Then, whp, there exists a reconstructible subset containing all but $o(|V(\mathcal{C})|)$ of vertices from $\pi(\mathcal{C})$.
    \label{lm:super4}
\end{lemma}

The rest of Section~\ref{sc:5} is split into three subsections. In Subsection~\ref{sc:5.1} we prove Lemma~\ref{lm:3}. In Subsection~\ref{sc:5.2} we deduce Lemma~\ref{lm:4} from Lemma~\ref{lm:super4}. Lemma~\ref{lm:super4} is proved in Subsection~\ref{sc:5.3}.

\subsection{Proof of Lemma~\ref{lm:3}}
\label{sc:5.1}

In order to prove Lemma~\ref{lm:3}, we must show that, for $\beta > 0$, almost all vertices of the kernel $\mathcal{K}$ satisfy $D_{\beta}$, Definition~\ref{df:A}. If a vertex does not satisfy $D_{\beta}$ then it belongs to a vertex subset $S \subseteq V(\mathcal{K})$ which induces a connected subgraph, has size at most $(1-\beta)|V(\mathcal{K})|$ and, also, the set $S$ violates either $|N_{\mathcal{K}}(S)| \geqslant 3$ or $\mathbf{E}(S) \leqslant \beta|S|\ln n$. The later gives us two separate cases --- Lemma~\ref{lm:3} follows directly from the following two statements.

\begin{claim}
        Let $\beta > 0$. Let $\lambda$ and $\mathcal{C}$ be from Claim~\ref{cl:L-contiguous}. Let $\lambda$ satisfy $\lambda = 1+\Omega(\frac{1}{\ln n})$. Let $\mathcal{K}$ be the kernel of $\mathcal{C}$. Then, the number of vertices $v \in V(\mathcal{K})$ such that there exists $S \subseteq V(\mathcal{K})$ satisfying the following properties
        \begin{itemize}
            \item $v \in S$,
            \item $S$ induces a connected subgraph in $\mathcal{K}$,
            \item $|S| \leq (1- \beta)|V(\mathcal{K})|$, and
            \item $|N_{\mathcal{K}}(S)| < 3$
        \end{itemize}
        is $O(\ln^2 n)$ whp.
        \label{cl:5_NSgeq}
\end{claim}

One can actually prove a stronger bound of $O(1)$ instead of $O(\ln^2n)$ in Claim~\ref{cl:5_NSgeq} but we do not need it here and we give a weaker bound in order to simplify the proof.

\begin{lemma}
    Let $\beta > 0$. Let $\lambda$ and $\mathcal{C}$ be from Claim~\ref{cl:L-contiguous}. Let $\lambda >1$ be constant. Let $\mathcal{K}$ be the kernel of $\mathcal{C}$. Then, the number of vertices $v \in V(\mathcal{K})$ such that there exists $S \subseteq V(\mathcal{K})$ satisfying the following properties
        \begin{itemize}
            \item $v \in S$,
            \item $S$ induces a connected subgraph in $\mathcal{K}$,
            \item $|S| \leq (1- \beta)|V(\mathcal{K})|$, and
            \item $\mathbf{E}(S) > \beta|S|\ln n$
        \end{itemize}
    is $o(|V(\mathcal{K})|)$ whp.
    \label{lm:5_ESleq}
\end{lemma}

We first prove Claim~\ref{cl:5_NSgeq}. Then we give auxiliary definition and claims that will be helpful for the proof of Lemma~\ref{lm:5_ESleq} and, lastly, we prove the Lemma~\ref{lm:5_ESleq}. 

\begin{proof}[Proof of Claim~\ref{cl:5_NSgeq}]

Let us note that, by the expansion properties, Claim~\ref{cl:expansion}, we have that, whp, every subset $S \subseteq V(\mathcal{K})$ with $|N_{\mathcal{K}}(S)| \leqslant 2$ has size either $O(1)$ or $n - O(1)$. Indeed, either $S$ or $V(\mathcal{K}) \setminus (S \cup N_{\mathcal{K}})$ have size at most $|V(\mathcal{K})|/2$ and, hence, the set is of size $O(1)$ by expansion. Thus, it suffices to prove that, whp, the number of subsets $S \subseteq V(\mathcal{K})$ with $|N_{\mathcal{K}}(S)| \leqslant 2$ is
\begin{itemize}
    \item zero, for $|S| \in \{3, \ldots, O(1)\}$;
    \item $O(\ln^2n)$, for $|S| \leqslant 2$.
\end{itemize}

Since $\mathcal{K}$ has minimum degree at least $3$, Claim~\ref{cl:kernel2} immediately implies the former property. Also, for the latter property, Claim~\ref{cl:kernel2} shows that there are just two possible options:
\begin{itemize}
    \item $S= \{v\}$, where $\deg v \leqslant 4$ and $v$ is incident to a loop or a multiedge;
    \item $S = \{u, v\}$, where $\deg u = \deg v = 3$, $u$ and $v$ are adjacent, and there is a loop or multiedge incident to either $u$ or $v$.
\end{itemize}

By S1 and S3 of Claim~\ref{cl:S}, whp $$\max_{v \in V(\mathcal{K})} \deg v = O(\ln |V(\mathcal{K})|) = O(\ln n).$$
So, Claim~\ref{cl:edge_prob}
and Markov's inequality together imply that whp at most $O(\ln^2n)$ vertices $v \in V(\mathcal{K})$ satisfy any of the options listed above. 

\end{proof}

In order to show Lemma~\ref{lm:5_ESleq}, let us state an auxiliary definition and a claim.

Let $\mu$ and $\mathcal{K}$ be from the model $\mathcal{L}$, Claim~\ref{cl:L-contiguous}. Let $\nu := 1-\mu$. Recall that, for each edge $e \in E(\mathcal{K})$, the corresponding 2-path $\mathcal{P}_e$ generated in the third step of the model $\mathcal{L}$ has length of $\mathrm{Geom}(\nu)$, independently of other edges. Let $\mathrm{NB}(r, \nu)$ denote the negative binomial distribution with parameters $r$ and $\nu$, i.e., the sum of $r$ independent $\mathrm{Geom}(\nu)$ random variables. In particular, 
$$\mathbb{E} NB(r, \nu) = r\nu^{-1}.$$ We will need a standard Chernoff-type concentration bound for this distribution (see~\cite{Br}, Brown) given below.

\begin{claim}
    For any $\gamma > 1$,
\begin{equation}
    \mathbb{P}(\mathrm{NB}(r, \nu) > \gamma r \nu^{-1}) \leqslant \exp\left(-\frac{\gamma r (1 - 1/\gamma)^2}{2}\right).
    \label{eq:5.1}
\end{equation}
    \label{cl:nb_chernoff}.
\end{claim}

\begin{definition}
    Let $\mathcal{C}$ be a connected graph with minimum degree 2 and let $\mathcal{K}$ be its kernel. For a subset $S \subseteq V(\mathcal{K})$ inducing a connected multigraph, let $\mathbf{D}(S)$ denote the number of edges in $E(\mathcal{K})$ that are incident to at least one vertex from $S$.
    \label{df:D_defined}
\end{definition}

This quantity is significant because $\mathbf{E}(S)$ can be naturally reformulated using $\mathbf{D}(S)$. Namely, 
\begin{equation}
    \mathbf{E}(S) = \sum_{e \in E(\mathcal{K}),\ e \textnormal{ contirbutes towards the sum }\mathbf{D}(S)}|E(\mathcal{P}_e)|.
    \label{eq:6.35}
\end{equation}

\begin{proof}[Proof of Lemma~\ref{lm:5_ESleq}]
    
    Let $\mathcal{K}$ be generated in the first two steps of the model $\mathcal{L}$, Claim~\ref{cl:L-contiguous}. So, the remaining probability space is described by the third step of $\mathcal{L}$.
    
    Assume that the maximal degree of $\mathcal{K}$ is $o(\ln n)$, which holds whp due to S3, Claim~\ref{cl:S}. As $\lambda$ is constant, there exists $\varepsilon \to 0$ such that
\begin{equation}
        \mathbf{D}(S) \leqslant |S| \cdot  \varepsilon(\lambda-1) \ln n.
        \label{eq:DS}
\end{equation}  
    
    Consider an arbitrary vertex $v \in V(\mathcal{K})$. Let $F(v)$ be the event that there exists a subset $S \subseteq V(\mathcal{K})$ inducing a connected subgraph and containing $v$ such that 
$$
\mathbf{E}(S) > \beta |S| \ln n.
$$
In what remains we show that, for every vertex $v \in V(\mathcal{K})$, 
it holds that $\mathbb{P}(F(v)) = o(1)$, since then Lemma~\ref{lm:3} follows from Markov's inequality. We will prove $\mathbb{P}(F(v)) = o(1)$ using a union bound. 

First, since the maximal degree of $\mathcal{K}$ is $o(\ln n)$, by Claim~\ref{cl:graph_bound}, for every $v \in V(\mathcal{K})$ and every integer $s \geq 0$ we have
\begin{equation}
    |\{S \subseteq V(\mathcal{\mathcal{K}}): v \in S,\ |S| = s,\ S\textnormal{ induces connected subgraph in }\mathcal{K}\}| = (o(\ln n))^{s-1}. 
    \label{eq:5.2.S_num}
\end{equation}

Now, consider a vertex $v \in V(\mathcal{K})$ and
let us show that $\mathbb{P}(F(v)) = o(1)$. If $F(v)$ holds then there exists a subset $S$ inducing a connected subgraph and containing $v$ such that 
\begin{equation}
    \mathbf{E}(S) > \beta |S| \ln n \quad \textnormal{and}\quad  \mathbf{D}(S) \le |S| \cdot  \varepsilon(\lambda-1) \ln n.
    \label{eq:5.2_implication}
\end{equation}
Notice that $\mathbf{D}(S)$ is a fixed number and $\mathbf{E}(S)$ is a random one. Recalling~\eqref{eq:6.35}, 
$$\mathbf{E}(S) = \sum_{e \in E(\mathcal{K}),\ e \textnormal{ contributes towards the sum }\mathbf{D}(S)}|E(\mathcal{P}_e)|,$$
and that every path $\mathcal{P}_e$ has a random number of edges distributed according to the third step of~$\mathcal{L}$, Claim~\ref{cl:L-contiguous}, for $\nu = 1 - \mu$, where $\mu$ is from Claim~\ref{cl:L-contiguous}, we have
\begin{equation}
    \mathbf{E}(S) = \textnormal{NB}(\mathbf{D}(S), \nu).
    \label{eq:E_D_relation}
\end{equation}
So, for every $S$ containing $v$ and inducing a connected graph we have an upper bound: the probability that it satisfies $F(v)$ is at most  
\begin{align}
    \mathbb{P}(S\textnormal{ satisfies~\eqref{eq:5.2_implication}}) &\stackrel{\eqref{eq:5.2_implication}}\leq \mathbb{P}(\mathbf{E}(S) > \beta |S| \ln n)\notag \\
    &\stackrel{\eqref{eq:E_D_relation}}= \mathbb{P}(\textnormal{NB}\biggl(\mathbf{D}(S), \nu) > \beta |S| \ln n\biggr)\notag \\    &\stackrel{\eqref{eq:5.2_implication}}\leq \mathbb{P}\biggl(\textnormal{NB}(|S| \cdot  \varepsilon(\lambda-1) \ln n, \nu) > \beta |S| \ln n\biggr) \notag \\
    &\stackrel{\textnormal{Claim}~\ref{cl:nb_chernoff}}\leq \exp\left(\frac{-\nu \beta |S| \ln n(1 - o(1))}{2}\right) \notag \\
    &= \exp(-\omega(|S| \ln \ln n)),
    \label{eq:5.2_main}
\end{align}
where the last equality holds as $\nu$ is constant.

Let us take a union bound of the event that $S$ satisfies~\eqref{eq:5.2_implication} over all $s \leq (1-\beta)|V(\mathcal{K})|$ and over all subsets $S$ inducing connected subgraph of size $s$ that contain $v$. We conclude that

\begin{align*}
    \mathbb{P}(F(v)) &\leqslant \sum_{s = 1}^{(1-\beta)|V(\mathcal{K})|}\sum_{S \subseteq V(\mathcal{K}) \textnormal{ induces connected subgraph},\ |S| = s,\ v \in S}\mathbb{P}(S \textnormal{ satisfies~\eqref{eq:5.2_implication}}) \\ &\stackrel{\eqref{eq:5.2.S_num}}\leqslant \sum_{s = 1}^{(1-\beta)|V(\mathcal{K})|}(o(\ln n))^{s-1}\max_{S \subseteq V(\mathcal{K}) \textnormal{ induces connected subgraph},\ |S| = s,\ v \in S}\mathbb{P}(S \textnormal{ satisfies~\eqref{eq:5.2_implication}}) \\
    &\stackrel{~\eqref{eq:5.2_main}}\leqslant \sum_{s = 1}^{(1-\beta)|V(\mathcal{K})|}(o(\ln n))^{s-1} \exp\left(-\omega(s \ln \ln n)\right) = o(1).
    \end{align*}

\end{proof}

\subsection{Concluding Lemma~\ref{lm:4}}
\label{sc:5.2}

Let $\mathcal{C}$ and $\lambda$ be from Claim~\ref{cl:contiguous}, let $\mathcal{K}$ be the kernel of $\mathcal{C}$ and let $\lambda>1$ be constant. Let $\beta >0$ be from Lemma~\ref{lm:2}. In order to prove Lemma~\ref{lm:4} we shall show R1$^*$---R4$^*$ from Lemma~\ref{lm:super4} for $\mathcal{C}$. The size of $\mathcal{C}$ is at most $n$ by Claim~\ref{cl:core_size}. Let us set $U : =\{v \in V(\mathcal{K}) : D_{\beta}(v)\}$, so R1$^*$ follows from Lemma~\ref{lm:2}. We now need to show R2$^*$, R3$^*$, and R4$^*$. Let us show slightly stronger statements as we will need them in Section~\ref{sc:ad}. First, R2$^*$ follows from the following claim applied to $U$.

\begin{claim}
    Let $\lambda$, $\mathcal{K}$, and $\mathcal{C}$ be from Claim~\ref{cl:L-contiguous}. If $\lambda = 1 + \omega(\ln^{-1}n)$ then, whp, for every $U \subseteq V(\mathcal{K})$ with $|U| = (1-o(1))|V(\mathcal{K})|$, the graph induced by $U$ contains all but $o(|E(\mathcal{K})|)$ edges of $\mathcal{K}$;
    \label{cl:triv1}
\end{claim}

\begin{proof}
    Recall that, by S1 and S2 from Claim~\ref{cl:S}, $|E(\mathcal{K})|$ is linear in $|V(\mathcal{K})|$ whp. So, in order to conclude the claim, we need to show that the graph induced by $U$ contains all but $o(|V(\mathcal{K})|)$ edges. This indeed holds as, by Lemma~\ref{lm:3}, $|V(\mathcal{K}) \setminus U| = o(|V(\mathcal{K})|)$ whp so, combined with S1 and S4 from Claim~\ref{cl:S} the number of edges incident to $V(\mathcal{K}) \setminus U$ is $o(|V(\mathcal{K})|)$ whp.
\end{proof}

 Next, the following claim implies R3$^*$.

 \begin{claim}
    Let $\lambda$, $\mathcal{K}$, and $\mathcal{C}$ be from Claim~\ref{cl:L-contiguous}. If $\lambda = 1 +\omega(\ln ^{-1}n) $, then, whp,
     $$|\{e\in E(\mathcal{K}): \textnormal{the corresponding 2-path }\mathcal{P}_e\textnormal{ in }\mathcal{C}\textnormal{ satisfies }|V(\mathcal{P}_e)| > \ln n/100\}| = o(|E(\mathcal{K})|).$$
     \label{cl:triv2}
 \end{claim}

 \begin{proof}
     The claim follows instantly from Markov's inequality, as due to the third step of the model~$\mathcal{L}$ from Claim~\ref{cl:L-contiguous}, for every edge $e \in E(\mathcal{K})$, independently, $\mathbb{P}(|\mathcal{P}_e|> \ln n/100) = o(1)$.
 \end{proof}
 
  Finally, R4$^*$ follows from the following claim.

\begin{claim}
    Let $\lambda$, $\mathcal{K}$, and $\mathcal{C}$ be from Claim~\ref{cl:L-contiguous}. If $\lambda = 1 +\omega(\ln^{-1}n)$ then, whp, for every $F \subseteq E(\mathcal{K})$ with $|F| = o(|E(\mathcal{K})|)$, we have $|\cup_{e \in F} V(\mathcal{P}_e)| = o(|V(\mathcal{C})|)$.
    \label{cl:cl59}
\end{claim}

\begin{proof}
    Let us fix $\mathcal{K}$ generated in the first two steps of the model $\mathcal{L}$, Claim~\ref{cl:L-contiguous}. So, the remaining probability space is described by the third step of Claim~\ref{cl:L-contiguous}.

    Let $\alpha \to 0$ arbitrarily slowly and set $r := \lfloor\alpha^2|E(\mathcal{K})|\rfloor$. Let $\mu$ be as in Claim~\ref{cl:L-contiguous}. It is clear from the description of the third step of $\mathcal{L}$, $|V(\mathcal{C})| = \Omega((1-\mu)^{-1}|E(\mathcal{K})|)$ whp. 
    
    Let us show by a union bound that whp for all $F \subseteq E(\mathcal{K})$ with $|F| =r$ we have
    \begin{equation}
        |\cup_{e \in F} E(\mathcal{P}_e)|  \leq \alpha (1-\mu)^{-1}|E(\mathcal{K})|.
        \label{eq:6.10_main}
    \end{equation} 
    This is sufficient to conclude Claim~\ref{cl:cl59}, since then 
    $$|\cup_{e \in F} V(\mathcal{P}_e)| \leq |\cup_{e \in F} E(\mathcal{P}_e)| + |F| = o(|V(\mathcal{C})|).$$
    
    The number of subsets of $E(\mathcal{K})$ of size $r$ is at most 
    \begin{equation}
        {|E(\mathcal{K})| \choose \lfloor \alpha^2|E(\mathcal{K})| \rfloor} \leq \exp\biggl(\left(1+2\ln\left(\alpha^{-1}\right)\right)\alpha^2|E(\mathcal{K})|\biggr).
        \label{eq:6.10.1}
    \end{equation}
    Also, for every $F \subseteq E(\mathcal{K})$ with $|F| =r$, due to the description of the third step of $\mathcal{L}$ from Claim~\ref{cl:L-contiguous} we have
    \begin{align}
    \mathbb{P}\biggl(\textnormal{the set }F \textnormal{ breaks~\eqref{eq:6.10_main}}\biggr)
    &\leq \mathbb{P}\biggl(\mathrm{NB}(r, (1-\mu)) > \alpha(1-\mu)^{-1}|E(\mathcal{K})|\biggr) \notag \\
    &\stackrel{\text{Claim}~\ref{cl:nb_chernoff}}\leq \exp\biggl(- \frac{\alpha|E(\mathcal{K})|(1 - o(1))}{2}\biggr) \notag \\
    &\stackrel{\alpha \to 0}= \exp\biggl(-\omega\left(\left(1+2\ln\left(\alpha^{-1}\right)\right)\alpha^2|E(\mathcal{K})|\right)\biggr).
        \label{eq:6.10.2}
    \end{align}
    So, taking the union bound over all subsets $F\subseteq V(\mathcal{K})$ with $|F| = r$ we conclude
    $$\mathbb{P}(\textnormal{there exists a set }F \subseteq E(\mathcal{K}) \textnormal{ with }|F| = r\textnormal{ that breaks~\eqref{eq:6.10_main}}) \stackrel{\eqref{eq:6.10.1},~\eqref{eq:6.10.2}}= o(1).$$
\end{proof}

\subsection{Proof of Lemma~\ref{lm:super4}}
\label{sc:5.3}

\begin{proof}[Proof of Lemma~\ref{lm:super4}]    
    Let us show that under the assumption of Lemma~\ref{lm:super4} the following claim holds.
    \begin{claim}
        Whp, for all but $o(|E(\mathcal{K})|)$ edges $uv \in E(\mathcal{K})$ it holds that any maximal reconstructible vertex subset in $\pi(\mathcal{C})$ containing $u$ and $v$ also contains $V(\mathcal{P}_{uv})$.
        \label{cl:th1_1}
    \end{claim}

    \begin{proof}[Proof of Claim~\ref{cl:th1_1}]
        
        Consider an edge $uv \in E(\mathcal{C})$ and a maximal reconstructible subset $W$ such that $u, v \in W$ but $V(\mathcal{P}_{uv}) \nsubseteq W$. Then, for the trivial $\pi(\mathcal{C})$-rigid map $\varphi = id$ of $W$ there exist at least two ways to extend $\varphi$ to $V(\mathcal{P}_{uv})$. Hence, the edge $uv$ satisfies the following property: 
        \begin{equation}
            \textnormal{There is a non-trivial way to extend the trivial }\pi(\mathcal{C})\textnormal{-rigid map }
            \varphi: \{u, v\} \to \mathbb{R}, \varphi = id,
            \textnormal{ to }\mathcal{P}_{uv}.
            \label{eq:6.11_trivial}
        \end{equation}

        It remains to prove that whp the number of edges in $E(\mathcal{K})$ that satisfy~\eqref{eq:6.11_trivial} is $o(|E(\mathcal{K})|)$. 
        Recall that $\pi:V(\mathcal{C})\to V$ is a uniformly random injection. Notice that, by R3$^*$, all but $o(|E(\mathcal{K})|)$ edges of $\mathcal{K}$ satisfy the requirement of Claim~\ref{cl:3.4}. So, due to Claim~\ref{cl:3.4}, for all but $o(|E(\mathcal{K})|)$ edges the probability that the event~\eqref{eq:6.11_trivial} holds is $o(1)$. Therefore, by Markov's inequality we conclude Claim~\ref{cl:th1_1}.
    \end{proof}

    Let $\pi:V(\mathcal{C})\to V$ be an arbitrary injection that satisfies the two properties that hold whp by R1$^*$ and Claim~\ref{cl:th1_1}. Consider a maximal subset $W \subseteq V(\mathcal{C})$ such that $\pi(W)$ is reconstructible in $\mathcal{C}$. The set $W$ is correctly defined as $\pi(U)$ is reconstructible in $\mathcal{C}$. Let us set 
    \begin{align*}
        F_1&=\{uv \in E(\mathcal{K}): v \notin W \cap V(\mathcal{K})\},\\
        F_2&=\{uv \in E(\mathcal{K}): u, v \in  W \cap V(\mathcal{K})\textnormal{ and }V(\mathcal{P}_{uv})\nsubseteq W\}.
    \end{align*}
    Then, $|F_1| = o(|E(\mathcal{K})|)$ by R2$^*$, as $U \subseteq W$. Also, $|F_2| = o(|E(\mathcal{K})|)$ as $\pi$ satisfies the conclusion of Claim~\ref{cl:th1_1}. Now, note that every vertex $v \in V(\mathcal{C})$, $v \notin W$, lies on some path $\mathcal{P}_e$ for some $e \in E(\mathcal{K})$. So, 
    $$V(\mathcal{C}) \setminus W \subseteq \cup_{e \in F_1 \cup F_2} V(\mathcal{P}_e),$$
    so we conclude the desired upper bound
    $$|V(\mathcal{C}) \setminus W| \leq \sum_{e \in F_1 \cup F_2} |V(\mathcal{P}_e)| \stackrel{|F_1|, |F_2| = o(|E(\mathcal{K
    })|),\ R4^*}= o(|V(\mathcal{K})|).$$ 
    
\end{proof}

\section{Proof of Theorem~\ref{th:2}}
\label{sc:ad}

In this section we shall prove Theorem~\ref{th:2}

We plan to show that there exists an induced subgraph $\mathcal{\tilde C} \subseteq \mathcal{C}$ such that $\mathcal{\tilde C}$ and its kernel $\mathcal{\tilde K}$ satisfy some good properties and also have maximum degree at most $3$.
Then, we plan to apply our standard pipeline in order to show  Theorem~\ref{th:2}.  

The rest of the section is organised as follows. First, we define the graphs $\mathcal{\tilde C}$ and $\mathcal{\tilde K}$. Second, we prove a couple of simple properties for the graphs $\mathcal{\tilde C}$ and $\mathcal{\tilde K}$. Lastly, we establish Theorem~\ref{th:2} using Theorem~\ref{th:super1}, Lemma~\ref{lm:super2}, and Lemma~\ref{lm:super4}.

\subsection{Definition of $\mathcal{\tilde K}$ and $\mathcal{\tilde C}$}

In this subsection we describe our primary objects we work with in this section.

\begin{definition}
    Let $\lambda$ and $\mathcal{C}\sim \mathcal{L}$ be from Claim~\ref{cl:L-contiguous} and let $\mathcal{K}$ be the kernel of $\mathcal{C}$.

    The graph $\mathcal{\tilde K}$ is generated as follows.
        \begin{itemize}
        
        \item[P1] We delete every vertex of degree at least $4$ from $\mathcal{K}$;
        \item[P2] We delete vertices of degree at most $1$ as long as they exist;
        \item[P3] For every vertex $w$ of degree $2$ with edges $wu$ and $wv$ attached we delete $w$ and add the edge $uv$;
        \item[P4] We delete everything except for an arbitrary largest component.
        \end{itemize}
        \label{rm:P1-P4}

    Notice that every component at the end of P3 is the kernel of some induced subgraph of $\mathcal{C}$. So, let $\mathcal{\tilde C}$ be the induced subgraph corresponding to $\mathcal{\tilde K}$.
    
    \label{df:C_degree_2-3}
\end{definition}

\begin{remark}
    Notice that both $\mathcal{\tilde K}$ and $\mathcal{\tilde C}$ are connected.
\end{remark}

\subsection{Properties of $\mathcal{\tilde C}$}

    In this subsection we describe several properties of $\mathcal{\tilde C}$ and $\mathcal{\tilde K}$ from Definition~\ref{df:C_degree_2-3} that will be needed in order to show Theorem~\ref{th:2}.

\begin{claim}
    Let $\mathcal{C}$, $\mathcal{K}$, $\mathcal{\tilde C}$, and $\mathcal{\tilde K}$ be as in Definition~\ref{df:C_degree_2-3}. Then, whp,
\begin{itemize}
    \item $|V(\mathcal{\tilde K})| = (1 - o(1))|V(\mathcal{K})|$;
    \item $|E(\mathcal{\tilde K})| = (1 - o(1))|E(\mathcal{K})|$;
    \item $|V(\mathcal{\tilde C})| = (1 - o(1))|V(\mathcal{C})|$;
    \item $|E(\mathcal{\tilde C})| = (1 - o(1))|E(\mathcal{C})|$.
\end{itemize}
    \label{cl:almost_same}
\end{claim}

\begin{proof}
    
    Let us show that among the four properties of Claim~\ref{cl:almost_same}, each implies the next.

    \begin{itemize}
        \item[$1 \implies 2$.]First, notice that $V(\mathcal{\tilde K}) \subseteq V(\mathcal{K})$ and $E(\mathcal{K}|_{V(\mathcal{\tilde K})}) \subseteq E(\mathcal{\tilde K})$, where the latter inclusion may be strict.    

    By S4 of Claim~\ref{cl:S}, whp, for every $S \subseteq \mathcal{K}$ if $|S| = o(|V(\mathcal{K})|)$ then the total degree of vertices from $S$ is $o(|\mathcal{K}|)$ as well. Hence, the first property from Claim~\ref{cl:almost_same} implies the second property.
    
        \item[$3 \implies 4$.] Second, notice that every vertex of $\mathcal{C}$ is either from $V(\mathcal{K})$ or has degree $2$. So, applying the same reasoning and S4 of Claim~\ref{cl:S}, whp, the third property of Claim~\ref{cl:almost_same} implies the fourth.
        
        \item[$2 \implies 3$.] Finally, we need to show that, for every $S \subseteq E(\mathcal{K})$ satisfying $|S| = o( E(\mathcal{K}))$ we have 
    $$\sum_{e \in S}|V(\mathcal{P}_e)| = o(|V(\mathcal{C})|).$$
    This fact, however, was proved in Claim~\ref{cl:cl59}.
    \end{itemize}

    In order to conclude Claim~\ref{cl:almost_same}, it remains for us to show the first property
    $|V(\mathcal{\tilde K})| = (1 - o(1))|V(\mathcal{K})|$. Recall the process P1---P4 from Definition~\ref{df:C_degree_2-3}. Due to the implications we showed at the start of the proof, we need to prove that, after steps P1---P4, the resulting graph contains $1-o(1)$ proportion of vertices of $\mathcal{K}$. In order to do that, for each of the steps P1, P2, and P3 we show that at the end of that step there is a component of size $(1-o(1))|V(\mathcal{K})|$ whp.
    
    {\bf Step P1.} 
    By Markov's inequality and the degree sequence distribution of $\mathcal{K}$ given by the first step of $\mathcal{L}$ from Claim~\ref{cl:L-contiguous}, whp $$|\{v \in V(\mathcal{K}): \deg v \geq 4\}| = o(|V(\mathcal{K})|).$$
    Recall that, by S4 of Claim~\ref{cl:S}, whp the total sum of the degrees of these vertices is $o(|V(\mathcal{K})|)$. 
    Hence,
    \begin{equation}
        \textnormal{Whp, we deleted at most }o(|V(\mathcal{K})|)\textnormal{ edges during step P1.}
        \label{eq:aux_7}
    \end{equation}
    
    Let $\mathcal{K}_1$ be the resulting graph after step P1. 
    Let us show that, by the end of step P1, there is a component of size $(1-o(1))|V(\mathcal{K})|$ whp. By~\eqref{eq:aux_7}, $|V(\mathcal{K}_1)| = (1-o(1))|V(\mathcal{K})|$ whp. Suppose that there is no component of size $(1 - o(1))|V(\mathcal{K})|$ in $\mathcal{K}_1$. Then, by a standard argument, the graph $\mathcal{K}_1$ can be split into two linearly large disjoint unions $\mathcal{S}_1$ and  $\mathcal{S}_2$ of connected components.
    Due to vertex expansion, Claim~\ref{cl:expansion}, we conclude that, whp, the number of edges between $\mathcal{S}_1$ and $\mathcal{K}\setminus \mathcal{S}_1$ is linear and, also, all these edges were deleted in step P1, which contradicts to~\eqref{eq:aux_7}.

    {\bf Step P2.}  Consider the largest component $\mathcal{ K}_2$ after step P1. During step P2, some vertices of $\mathcal{K}_2$ are deleted, but $\mathcal{K}_2$ will never be split into two unconnected components. So, it is sufficient for us to show that the total number of vertices deleted in step P2 is $o(|V(\mathcal{K})|)$. 
    
    Let $F \subseteq E(\mathcal{K})$ denote the edge boundary of $\mathcal{K}_2$ in $\mathcal{K}$ that is updated every time the process deletes a vertex during step P2. In order to conclude that step P2 deletes at most $o(|V(\mathcal{K})|)$ vertices, it is sufficient for us to show that 
    \begin{itemize}
        \item[Q1] $|F| = o(|V(\mathcal{K})|)$, at the beginning of P3;
        \item[Q2] Every time we delete a vertex during P3, the size of $|F|$ decreases.
    \end{itemize}
    First, $F$ is a subset of the edges adjacent to $V(\mathcal{K}) \setminus V(\mathcal{K}_2)$. So, Q1 whp holds, since whp
    $$|V(\mathcal{K}) \setminus V(\mathcal{K}_2)| = o(|V(\mathcal{K})|)$$and, by S4 of  Claim~\ref{cl:S}, whp the number of edges adjacent to $V(\mathcal{K}) \setminus V(\mathcal{K}_2)$ is $o(|V(\mathcal{K})|)$. 
    
    Second, suppose that at some point during step P2 the current resulting graph $\mathcal{K}_2$ has size at least $2$ and suppose that the process deletes a leaf vertex $v \in V(\mathcal{K}_2)$ this very moment. Then, the very moment before we delete $v$, the number of edges in $F$ adjacent to $v$ is at least $2$. Deleting $v$ from $\mathcal{K}_2$, we remove these edges from $F$ and add at most 1 edge that connects $v$ to $\mathcal{K}_2$. Hence, Q2 holds.
    
    {\bf Step P3.} 
    Consider the largest component $\mathcal{ K}_3$ after step P2.
    During step P3, some vertices of $\mathcal{K}_3$ are deleted, but $\mathcal{K}_3$ will never be split into two unconnected components. So, it is sufficient for us to show that the total number of vertices deleted from $\mathcal{ K}_3$ in step P3 is $o(|V(\mathcal{K})|)$. 
    
    Notice that during step P3 we only delete the vertices of $\mathcal{K}_3$ that have degree 2 at the start of step P3. So, we need to show that the number of such vertices is at most $o(|V(\mathcal{K})|)$. The number of those vertices is actually at most the  size of the boundary $F$ defined in step P2. Indeed, every vertex of degree $2$ in $\mathcal{K}_3$ is incident to a unique edge from the boundary. So, it remains to recall that, at the beginning of step P3, the boundary $F$ satisfies $|F| = o(|V(\mathcal{K})|)$, due to Q1 and Q2. This concludes Claim~\ref{cl:S}.
    
\end{proof}

\begin{claim}
    Let $\lambda$, $\mathcal{L}$, $\mathcal{C}$, $\mathcal{K}$, $\mathcal{\tilde C}$, and $\mathcal{\tilde K}$ be as in Definition~\ref{df:C_degree_2-3} and let $\lambda = 1 +\omega(\ln ^{-1}n)$. Then $\mathcal{\tilde K}$ and $\mathcal{\tilde C}$ are contiguous to the following model generating graphs $\mathcal{\hat K}$ and $\mathcal{\hat C}$:

    \begin{itemize}
        \item[T1] generate $(1-o(1))|V(\mathcal{K})| \le\tilde k \le |V(\mathcal{K})|$ according to some distribution;
        \item[T2] generate $\mathcal{\hat K}$ to be a uniformly random 3-regular multigraph of size $\tilde k$;
        \item[T3] generate $\mathcal{\hat C}$ to be some random graph with kernel $\mathcal{\hat K}$ with distribution such that, for arbitrary $\beta \in (0, 0.1)$ and for every subset $F \subseteq E(\mathcal{\hat K})$ we have
        $$\mathbb{P}\biggl(\sum_{e \in F}|V(\mathcal{P}_e)| \ge\beta|F|\ln n\biggr) = 1 - \exp\biggl(\omega(|F|)\biggr).$$
    \end{itemize}
    \label{cl:distr_bound}
\end{claim}

\begin{proof}
    Fix a typical result of the first step of the model generating $\mathcal{K}$. At this point, we have the degree sequence of $\mathcal{K}$ but the graph itself is not revealed. Now, we run steps P1---P3 from Definition~\ref{rm:P1-P4} in order to generate $\mathcal{\tilde K}$. In order to run them properly we will need to reveal some edges of $\mathcal{K}$. We aim to reveal these edges of $\mathcal{K}$ according to the second step of the model from Claim~\ref{cl:contiguous} in such a way that at the start of each step P1, P2, and P3 none of the edges inside the remaining graph are exposed. During the third step we will show that, if the steps P1, P2, and P3 were typical, then the resulting graphs $\mathcal{\tilde K}$ and $\mathcal{\tilde C}$ are generated according to the random model described in Claim~\ref{cl:distr_bound}.
    
    The rest of the proof is split into the three steps corresponding to P1---P3 of Definition~\ref{rm:P1-P4}. In steps P1 and P2 we shall expose the edges in a way that, at the end of the given step, the remaining graph is a uniformly random multigraph with given degree sequence. Then, in the third step, we shall conclude Claim~\ref{cl:distr_bound}.

    {\bf The first step.} During the first step let us reveal all edges for all vertices of $\mathcal{K}$ with degree at least $4$, according to the second step of $\mathcal{L}$. This allows us to delete these vertices properly. The resulting graph after step P1 is a uniformly random graph given its degree sequence with maximum degree 3.

    {\bf The second step.} In step P2 we shall delete all vertices of degree at most 1. Every time we delete a vertex of degree $1$, we reveal its only unrevealed edge, according to the second step of $\mathcal{L}$. This allows us to update the degree of the remaining vertices and, hence, continue running step P2 properly. This step happens until there are no vertices of degree 1 in the resulting graph. Then, we simply additionally delete all vertices of degree 0 in the resulting graph. So, the final resulting graph of step P2 is uniformly random graph given a degree sequence consisting of values from $\{2, 3\}$. This step is similar to the process applied in~\cite{JaLu} (see Section~3 for details).

    {\bf The third step.} At the beginning of step P3 we have a set of vertices, whose degrees are $2$ or $3$, but none of the edges inside them is revealed. By Claim~\ref{cl:almost_same}, whp the number of vertices of degree $3$ is $\tilde k = (1-o(1))|V(\mathcal{K})|$ and the number of vertices of degree $2$ is $o(|V(\mathcal{K})|)$. Below we work under these conditions and our aim is to show T2 and T3. 

    In order to simplify the argument, we will need a well known contiguity model for the uniformly random multigraph given its degree sequence --- the pairing model.

    \begin{definition}[The pairing model]
        Let $\mathbf{k}_n = (k_{n, 1}, \ldots , k_{n, m(n)})$ be a degree sequence. Consider $m(n)$ bins, where the $i$-th bin, $i\in[m(n)]$, contains $k_{n, i}$ distinct half edges. Consider a uniformly random matching on all half-edges.
        Then, substituting every pair of matched half-edges by an edge generates a distribution over multigraphs with degree sequence $\mathbf{k}$. 
    \end{definition}
        
    A well-known contiguity fact is as follows.
    
    \begin{remark}
        Let $\mathbf{k}_n = (k_{n, 1}, \ldots , k_{n, m(n)})$ whp satisfy $m(n) \to \infty$ and $k_{n, i} \leq 3$ for every $i \in [m(n)]$. 
        Then, a uniformly random multigraph with degree sequence $\mathbf{k}$ and a graph generated according to the pairing model with degree sequence $\mathbf{k}$ are contiguous to each other.
        \label{rm:pairing_model}
    \end{remark}

    Roughly speaking, Remark~\ref{rm:pairing_model} holds as, whp, the random graphs generated in both models have $O(1)$ multiple edges and loops and, for every $3$-regular multigraph $\mathcal{G}$ with $m(n)$ vertices, the ratio of the probabilities of $G$ in two given distributions only depends on the number of loops, double edges and triple edges in $\mathcal{G}$.

    By Remark~\ref{rm:pairing_model}, let us now assume that $\mathcal{\tilde K}$ is generated according to the pairing model.
    
    We now need to show T2 ---i.e. that $\mathcal{\tilde K}$ is contiguous to a random $3$-regular multigraph $\mathcal{K}_3$ of size $\tilde k$ generated according to the pairing model.
    For a graph $G$ let $\psi(G)$ denote the graph after contracting all of its 2-paths. Notice that, in the pairing model, the distribution of $\psi(\mathcal{K})$ is precisely the same as the distribution of $\mathcal{K}_3$, so T2 follows by Remark~\ref{rm:pairing_model}.

    Now, condition on the fact that $\psi(\tilde K) = \mathcal{\bar K}$ for some fixed $3$-regular multigraph $\bar K$ of size $\tilde k$ and let us show T3 under that condition. So, the remaining probability space is described by the conditioned pairing model and independent step three of $\mathcal{L}$. Using this probability space, we need to show that, for every $\beta \in (0, 0.1)$ and for every subset $F \subseteq E(\mathcal{\bar K})$, we have
        \begin{equation}
            \mathbb{P}\biggl(\sum_{e \in F}|V(\mathcal{P}_e)| \ge\beta|F|\ln n\biggr) = 1 - \exp\biggl(\omega(|F|)\biggr).
            \label{eq:paths_7.5}
        \end{equation}
    
    Fix $F \subseteq E(\mathcal{\bar K})$. Let $r=o(|V(\mathcal{K})|)$ be the number of vertices of degree $2$ in $\mathcal{\tilde K}$. For $i \in [r]$, let $Y_i$ be the event that the $i$-th vertex of degree $2$ lies on some path contracting into some edge of $F$. Let $Y = \sum_{i \in [r]} Y_i$ be the total number of vertices of degree $2$ that lie on some 2-path contracting to some edge of $F$. Let us show that
    $$\mathbb{P}\biggl(Y \leq \frac{|F|}{100}\biggr) = 1 -\exp\biggl(\omega(|F|)\biggr).$$ 
    Recall that we work in a pairing model conditioned on $\psi(\tilde K) = \mathcal{\bar K}$. Notice, first, that some vertices of degree $2$ lie on a $2$-path corresponding to an edge of $\mathcal{\bar K}$, while the other vertices of degree $2$ do not contribute towards $Y$. So, let $S$ denote the random set of the former vertices and let us us reveal the subgraph of $\mathcal{\tilde K}$ on the latter vertices. One can easily verify that the remaining distribution on the graphs coincides with the following distribution given by a process on $S$. At the beginning of the step $i$ we have a component with vertex set consisting of  $V(\mathcal{\bar K})$ and first $i-1$ vertices of $S$. In step $i$ we reveal which of the current edges of the component is split by the $i$-th vertex of $S$; this edge is uniformly distributed among the edges that exist in the component at the moment.
    
    Notice that the step $i$ of the above-mentioned process determines the value of $Y_i$. Notice that, conditioned on the fact that $\sum_{j \in [i-1]}Y_j \leq \frac{|F|}{100}$, during step $i$ we have $$\mathbb{P}(Y_i) \leq \frac{2.02|F|}{3\tilde k},$$
    independent of the previous history of the process. Hence,
    $$\mathbb{P}\biggl(Y \leq \frac{|F|}{100}\biggr) \geq \mathbb{P}\biggl(\textnormal{Bin}\left(r, \frac{2.02|F|}{3\tilde k}\right) \leq\frac{|F|}{100} \biggr) \stackrel{\textnormal{Chernoff bound}}= 1 -\exp\biggl(\omega(|F|)\biggr),$$

    Now, by the third step of $\mathcal{L}$,
    $$\sum_{e \in F}|V(\mathcal{P}_e)|\sim \textnormal{NB}(|F|+Y, 1-\mu),$$ 
    where $\mu$ is the conjugate of $\lambda$ and whp satisfies $1 - \mu = \omega(\ln^{-1}n)$. So, supposing $Y \leq 0.01|F|$,~\eqref{eq:paths_7.5} indeed holds by Claim~\ref{cl:nb_chernoff}.  
\end{proof}
    
\subsection{Proof of Theorem~\ref{th:2}}

\begin{proof}[Proof of Theorem~\ref{th:2}]
     Let $\lambda$, $\mathcal{L}$, $\mathcal{C}$, $\mathcal{K}$, $\mathcal{\tilde C}$, and $\mathcal{\tilde K}$ be as in Definition~\ref{df:C_degree_2-3} and let $\lambda$ satisfy $\lambda \ge 1$, $\lambda = 1 +\omega(\ln ^{-1}n)$, and $\lambda = 1 + o(1)$. 
     
     Let us explain our strategy to conclude Theorem~\ref{th:2}. 
     \begin{itemize}
         \item[X1] First, we want to apply Theorem~\ref{th:super1} and then Lemma~\ref{lm:2} to $\mathcal{\tilde C}$. This shows that, whp, there exists $\beta>0$ such that the set of points of $\mathcal{\tilde K}$ satisfying the event $D_{\beta}$, Definition~\ref{df:A}, is reconstructible.\item[X2] Second, we aim to show that, whp, $(1-o(1))|V(\mathcal{\tilde K})|$ of vertices of $\mathcal{\tilde K}$ satisfy $D_{\beta}$. Due to Claim~\ref{cl:almost_same}, this whp implies the existence of a reconstructible subset of $\mathcal{K}$ of size $(1-o(1))|V(\mathcal{K})|$.
         \item[X3] Last, Lemma~\ref{lm:super4}  applied to $\mathcal{C}$ concludes whp the existence of a reconstructible subset of $\mathcal{C}$ of size $(1-o(1))|V(\mathcal{C})|$, as desired.
     \end{itemize}

    Below, we show X1, X2, and X3 consecutively.

    {\bf Proof of X1.} Let us prove that both Theorem~\ref{th:super1} and Lemma~\ref{lm:2} can be applied to $\mathcal{\tilde C}$ whp. In order to do that, we need to verify the properties R1---R5  of Theorem~\ref{th:super1} and R1$'$ and R2$'$ of Lemma~\ref{lm:2} for $\mathcal{\tilde C}$. Notice that R1 and R4 of Theorem~\ref{th:super1} hold since $\mathcal{\tilde K}$ is $3$-regular. Also, R2$'$ of Lemma~\ref{lm:2} holds due to as Claim~\ref{cl:graph_bound} as $\mathcal{\tilde K}$ is $3$-regular. Next, recall that in Section~\ref{sc:3.1} we proved whp the properties for $\mathcal{C}$ that are slightly stronger than the  remaining R2, R3, R5 of Theorem~\ref{th:super1}. So, combined with Claim~\ref{cl:almost_same}, $\mathcal{\tilde C}$ satisfies R2, R3, R5 of Theorem~\ref{th:super1} whp. The remaining property R1$'$ of Lemma~\ref{lm:2} follows from R2 of Theorem~\ref{th:super1} proved above and vertex expansion properties for uniformly random 3-regular multigraph (see~\cite{Bo88}) due to Claim~\ref{cl:distr_bound}.

     {\bf Proof of X2.} Suppose that X1 holds for some $\beta>0$. We prove that whp $(1-o(1))|V(\mathcal{\tilde K})|$ of vertices of $\mathcal{\tilde K}$ satisfy $D_{\beta}$. Instead we prove the same statement for the contiguous model of $\mathcal{\tilde K}$ and $\mathcal{\tilde C }$ as described by T1, T2, and T3 from Claim~\ref{cl:distr_bound}. Below, we assume that $\mathcal{\tilde K}$ and $\mathcal{\tilde C }$ are generated according to that model.
     
    Recall that the proof of Lemma~\ref{lm:3} consisted of the proof of two statements: Claim~\ref{cl:5_NSgeq} and Lemma~\ref{lm:5_ESleq}. The statement of Claim~\ref{cl:5_NSgeq} can be shown for a uniformly random 3-regular multigraph $\mathcal{\tilde K}$ exactly the same way as for $\mathcal{K}$. Indeed, the proof uses the following properties of $\mathcal{K}$:
    \begin{itemize}
        \item $\mathcal{K}$ is a uniformly random multigraph given its degree sequence;
        \item $\max_{v \in V(\mathcal{K})} \deg v = O(\ln |V(\mathcal{K})|) = O(\ln n)$;
        \item $\mathcal{K}$ whp satisfies the expansion property from Claim~\ref{cl:expansion}.
    \end{itemize} 
    In case of $\mathcal{\tilde K}$, the only non-trivial property is  the last one, which however holds due to~\cite{Bo88}.

    It remains for us to make the conclusion of Lemma~\ref{lm:5_ESleq} for $\mathcal{\tilde K}$ and $\mathcal{\tilde C}$. Actually, one can easily verify that the same proof of Lemma~\ref{lm:5_ESleq} holds for $\mathcal{\tilde K}$ and $\mathcal{\tilde C}$ but with the changes described below. The equation
    $$\mathbf{D}(S) \leqslant |S| \cdot  \varepsilon(\lambda-1) \ln n$$
    holds as $\mathbf{D}(S) \leq 3|S|$ and $\lambda - 1 = \omega(\ln ^{-1} n)$.
    We have a stronger upper bound:
    $$|\{S \subseteq V(\mathcal{\mathcal{K}}): v \in S,\ |S| = s,\ S\textnormal{ induces connected subgraph in }\mathcal{K}\}| \leq (2e)^{k-1},$$
    as the maximal degree of $\mathcal{\tilde K}$ is $3$. Notice that~\eqref{eq:E_D_relation} does not hold anymore. Consequently, instead of~\eqref{eq:5.2_main}, the following weaker bound follows from T3 of Claim~\ref{cl:distr_bound}.
\begin{align*}
    \mathbb{P}(S\textnormal{ satisfies~\eqref{eq:5.2_implication}}) = \exp(-\omega(|S|)).
\end{align*}
    The final union bound of Lemma~\ref{lm:5_ESleq} is changed accordingly:
\begin{align*}
    \mathbb{P}(F(v)) &\leqslant \sum_{s = 1}^{(1-\beta)|V(\mathcal{\tilde K})|}\sum_{S \subseteq V(\mathcal{\tilde K}) \textnormal{ induces connected subgraph},\ |S| = s,\ v \in S}\mathbb{P}(S \textnormal{ satisfies~\eqref{eq:5.2_implication}}) \\ &\leqslant \sum_{s = 1}^{(1-\beta)|V(\mathcal{\tilde K})|}(2e)^{s-1} \cdot \exp(-\omega(s)) = o(1).
    \end{align*}

    {\bf Proof of X3.} Suppose that we proved X1 and X2.  Recall that, during X3, $\mathcal{C}$ satisfies R2$^*$, R3$^*$, and R4$^*$ of Claim~\ref{lm:super4} (see Section~\ref{sc:5.2} and, in particular, Claim~\ref{cl:triv1}, Claim~\ref{cl:triv2}, and Claim~\ref{cl:cl59}). Also, X2 implies R1$^*$ of Lemma~\ref{lm:super4} for $\mathcal{C}$. Hence, Lemma~\ref{lm:super4} applied to $\mathcal{C}$ concludes Theorem~\ref{th:2}.

\end{proof}

\section{Discussion}
\label{sc:7}

In this work we studied the asymptotic of the size $\mathsf{R}$ of a largest reconstructible subset of the random graph. In Section~\ref{sc:1} we noted that, for $V \subseteq \mathbb{R}$ linearly independent over $\mathbb{Q}$, one cannot hope to reconstruct a subset larger than the size $\mathsf{S}$ of a largest connected component of the 2-core, for $\mathsf{S}\geq 2$. In general, though, it is easy to construct sets $V \subseteq\mathbb{R}$ such that $\mathsf{R} \geq (1+\varepsilon)\mathsf{S}$ for small $\varepsilon>0$. So, is it true that $\mathsf{R} \geq \mathsf{S}(1-o(1))$? We proved that the equation indeed holds for $p = \frac{1+\omega({\ln^{-1} n})}{n}$ --- in this case one can whp reconstruct all but $o(1)-$proportion of vertices from a largest component of the 2-core.

So, what is the asymptotic of $\mathsf{R}$ in the remaining cases? 
We think that the regime $p:=p(C) = \frac{1+\frac{C}{\ln n}}{n}$ is especially interesting.
We suspect that our methods give some insight about this case. However, our methods will most likely stop working when $C>0$ is small. So, how does the asymptotic value of $\mathsf{R}$ depend on $C$? Does $\frac{\mathsf{R}}{\mathsf{S}}$ converge to some constant (perhaps depending on $C$ and $V$)?  Does it hold that $\mathsf{R} = o(\mathsf{S})$ for small enough $C>0$? We conjecture that the property that $G(n, p)$ has a linearly large reconstructible subset of the core experiences a phase transition at $p(C)$ for some $C>0$.

\begin{conjecture}
    Let $V \subseteq \mathbb{R}$ with $|V|=n$ and let the set of known distances be distributed as in $G(V, \lambda/n)$ for some parameter $\lambda>0$.
    There exists $C>0$ such that, for every $\varepsilon > 0$, 
    \begin{itemize}
        \item for $\lambda = 1 +\frac{C+\varepsilon}{\ln n}$, whp there exists a reconstructible subset of size $\Omega_{\varepsilon}\left(\frac{n}{\ln^2n}\right)$;
        \item  for $\lambda = 1 +\frac{C-\varepsilon}{\ln n}$, whp every reconstructible subset has size at most $o\left(\frac{n}{\ln^2n}\right)$.
    \end{itemize}
\end{conjecture}

Let us also note another interesting study direction. In their work~\cite{MNPS}, Montgomery, Nenadov, Portier and Szab\'o studied the largest reconstructible subset of the random graph in the globally rigid setting --- see below. First, one is given the random graph and then the graph is placed on the real line in a way that minimises the size $\mathsf{R}$ of a largest reconstructible component. In the setting, they proved that $\mathsf{R} = o(n)$ for $pn < 1.1$. They also show that, in this setting, one can reconstruct a linearly large component for $pn>C$ for large enough $C>0$. We repeat the question they ask: what is the infimum among the constants $C > 0$ so that for $pn=C$ there exists a linearly large reconstructible subset in the global rigidity setting? We suspect that the answer should be related to the appearance of the 3-core in the random graph (which is close to $3.35/n$, see~\cite{PSW}).

\section*{Acknowledgments}

We thank Ben Green for useful discussions. We thank Tomasz Przyby{\l}owski for greatly inspiring us on this work. We thank Oliver Riordan for his invaluable assistance in verifying the proof and the text of this article.

\end{document}